\documentclass[reqno,10pt]{amsart}

\usepackage{amssymb,amsmath,amsthm,amsfonts,color}

\usepackage{mathrsfs,dsfont,comment,mathscinet,url}

\usepackage{ulem}
\usepackage{enumerate,esint}
\usepackage[final]{showlabels}
\usepackage[left=2.8cm,right=2.8cm,top=2.8cm,bottom=2.8cm]{geometry}
\usepackage{mathtools}
\usepackage{graphicx,tikz}
\usepackage[format=hang,labelfont=bf]{caption}

\usepackage{epstopdf}
\usepackage[colorlinks=true, pdfstartview=FitV, linkcolor=blue, 
            citecolor=blue, urlcolor=blue]{hyperref}

\allowdisplaybreaks
\numberwithin{equation}{section}
\theoremstyle{plain}
\newtheorem{theorem}{Theorem}[section]
\newtheorem{proposition}[theorem]{Proposition}
\newtheorem{lemma}[theorem]{Lemma}

\theoremstyle{definition}
\newtheorem{definition}[theorem]{Definition}
\newtheorem{remark}[theorem]{Remark}

\newtheorem*{theorem*}{Theorem}

\makeatletter
\def\th@plain{%
  \thm@notefont{}
  \itshape 
}
\def\th@definition{%
  \thm@notefont{}
  \normalfont 
}
\makeatother

\newcommand\R{\mathbb R}
\newcommand\N{\mathbb N}

\renewcommand\S{\mathbb S}

\renewcommand{\d}{\mathrm{d}}

\newcommand{\restr}[1]{|_{#1}}

\def\XXint#1#2#3{{\setbox0=\hbox{$#1{#2#3}{\int}$ }
\vcenter{\hbox{$#2#3$ }}\kern-.6\wd0}}

\newcommand{\tht}{\vartheta}

\newcommand{\loc}{\mathrm{loc}}
\newcommand{\tr}{\mathrm{tr}}
\newcommand{\cof}{\mathrm{cof}}
\newcommand{\adj}{\mathrm{adj}}
\newcommand{\sym}{\mathrm{sym}}
\renewcommand{\deg}{\mathrm{deg}}
\renewcommand{\div}{\mathrm{div}}

\newcommand{\dist}{\mathrm{dist}}
\newcommand\wk{\rightharpoonup}

\newcommand{\rt}{\R^3}
\newcommand{\rtt}{\R^{3\times 3}}
\newcommand{\rtwtw}{\R^{2 \times 2}}

\newcommand*\closure[1]{\overline{#1}}

\newcommand{\lebt}{\mathscr{L}^2}
\newcommand{\leb}{\mathscr{L}^3}

\definecolor{darkred}{rgb}{0.55, 0.0, 0.0}
\definecolor{cerulean}{rgb}{0.0, 0.48, 0.65}

\newcommand\MMM{\color{black}}

\title [A reduced model for plates in nonlinear magnetoelasticity]
{A reduced model for plates arising as low-energy $\boldsymbol{\Gamma}$-limit in nonlinear magnetoelasticity}

\author[M. Bresciani]{Marco Bresciani}
\address[M. Bresciani]{
		Department of Mathematics,
		Friedrich-Alexander Universit\"{a}t Erlangen-N\"{u}rnberg,
		Cauerstrasse 11, 91058 Erlangen, Germany.
		Email: \href{mailto:marco.bresciani@fau.de}{\tt marco.bresciani@fau.de}.
}

\author[M. Kru\v{z}\'ik]{Martin Kru\v{z}\'ik}
\address[M. Kru\v{z}\'ik]{Institute of Information Theory and Automation, Czech Academy of Sciences, 
Pod Vod\'arenskou V\v{e}\v{z}\'i 4, 182 00 Prague, Czechia and Faculty of Civil Engineering, Czech Technical University, Th\'{a}kurova 7, 166 29 Prague, Czechia.
		Email: \href{mailto: kruzik@utia.cas.cz}{\tt kruzik@utia.cas.cz}.
}

\makeatletter
\@namedef{subjclassname@2020}{%
  \textup{2020} Mathematics Subject Classification}
\makeatother

\begin{document}
\normalem

\setlength\parindent{0pt}

\date{\today}

\keywords{magnetoelasticity, Eulerian-Lagrangian energies, dimension reduction, $\Gamma$-convergence, quasistatic evolution, rate-independent processes, evolutionary $\Gamma$-convergence}

\subjclass[2020]{74F15, 74K20, 74H80}

\vskip .2truecm
\begin{abstract}
We investigate the problem of dimension reduction for plates in nonlinear magnetoelasticity. The model features a mixed Eulerian-Lagrangian formulation, as magnetizations are defined on the deformed set in the actual space.
We consider low-energy configurations by rescaling the elastic energy according to the linearized von K\'{a}rm\'{a}n regime.
First, we identify a reduced model by computing the $\Gamma$-limit of the magnetoelastic energy, as the thickness of the plate goes to zero. This extends a previous result obtained by the first author in the incompressible case to the compressible one. Then, we introduce applied loads given by mechanical forces and external magnetic fields and we prove that  sequences of almost minimizers of the total energy converge to minimizers of the corresponding energy in the reduced model. 
Subsequently, we study quasistatic evolutions driven by time-dependent applied loads and a rate-independent dissipation. We prove that {\MMM energetic solutions for the bulk model converge to energetic solutions for the reduced model and we establish a similar result for} 
solutions of the approximate incremental minimization problem. Both these results provide a further justification of the reduced model in the spirit of the evolutionary $\Gamma$-convergence.
\end{abstract}
\maketitle

\section{Introduction}
\label{sec:intro}

Magnetoelasticity concerns the interaction between magnetic fields and deformable solids \cite{brown,dorfmann.ogden,maugin}.
Indeed, it is known that magnetic materials can change their strain upon the application magnetic fields and this behaviour is termed magnetostriction. Conversely, it is possible to modify the magnetic response of such materials by means of mechanical loads. {\MMM Both these phenomena are of great interest in engineering since they constitute the basic operating principle of many technological devices such as sensors and actuators.}

Internally, magnetic  materials are subdivided into regions of uniform magnetization called magnetic domains \cite{hubert.schaefer}. {\MMM This structure originates from the competition of two effects: the magnetocrystalline anisotropy, that is, the existence of preferred magnetization directions, called easy axes, determined by the underlying crystal lattice; and the long-range interactions between magnetic dipoles which favour configurations with divergence-free  magnetization throughout the specimen.}
When a magnetic field is applied, a reorganization of the domain structure is observed: the boundaries between these domains shift and the domains themselves rotate. This is  mainly due to the more favourable orientation of certain easy axes with respect to the direction of the external field \cite{kruzik.prohl}. As a result, these movements produce a macroscopic strain and, in turn, lead to the deformation of the specimen.

According to the variational theory of Brown \cite{brown}, the magnetoelastic energy is a function of deformations and magnetizations, and
equilibrium states correspond to minima of the energy functional.
The model contemplates finite strains. Therefore, while deformations are defined  on the reference configuration (Lagrangian), magnetizations are defined on the deformed set in the actual space (Eulerian). 

Apart form magnetoelasticity \cite{barchiesi.henao.moracorral,bresciani.davoli.kruzik,kruzik.stefanelli.zeman}, mixed Eulerian-Lagrangian formulations appear also in other contexts, such as the theory of liquid crystals \cite{barchiesi.desimone,henao.stroffolini}, phase transitions \cite{grandi.kruzik.mainini.stefanelli,silhavy} and finite plasticity \cite{kruzik.melching.stefanelli,stefanelli}. From the mathematical point of view, the analysis of such models is very challenging. Indeed, several standard techniques are no longer available in this setting, so that novel strategies are required. For these reasons,  in recent years, mixed Eulerian-Lagrangian variational problems got the attention of the mathematical community.

Rigorously derived lower-dimensional models of continuum mechanics play an important role in applications because they preserve the main features of the bulk model but they are usually simpler from the computational point of view \cite{liakhova.luskin.zhang,luskin.zhang}.  Fundamental results obtained in  \cite{friesecke.james.mueller1,friesecke.james.mueller2}  have initiated a remarkable progress in this area and have established the prominent role of $\Gamma$-convergence \cite{braides} in the validation of reduced models for thin structures. For micromagnetics, among others, important results have been achieved in \cite{carbou,gioia.james}. 
However, in the case of magnetoelasticity, few rigorous results are available.
Two-dimensional models were first derived in \cite{kruzik.stefanelli.zanini} for Kirchhoff-Love plates starting from linearized magnetoelasticity, and then in \cite{davoli.kruzik.piovano.stefanelli} for non-simple materials in the fully nonlinear membrane regime. In the first case, rate-independent evolutions were also studied.

In this contribution, we derive a reduced model for plates in the
linearized von K\'{a}rm\'{a}n regime starting from nonlinear magnetoelasticity. Our results develop  the investigations initiated in \cite{bresciani} for incompressible materials to various extents. We consider compressible materials and we make more realistic assumptions on the elastic energy density. Also, we include applied loads given by mechanical forces and external magnetic fields. Unlike \cite{bresciani}, our analysis covers both the static the quasistatic setting. In the first case, we employ $\Gamma$-convergence techniques to study the asymptotic behaviour of minimizers of the magnetoelastic energy, as the thickness of the plate goes to zero. 
In the latter one, we investigate the dimension reduction in the framework of evolutionary $\Gamma$-convergence \cite{mielke.roubicek.stefanelli}. 

Let $h>0$ denote the thickness of a thin magnetoelastic plate $\Omega_h\coloneqq S \times hI\subset \R^3$, where $S \subset \R^2$ represents the section and $I\coloneqq (-1/2,1/2)$. Deformations are maps $\boldsymbol{\phi}\colon \Omega_h \to \R^3$ while magnetizations are given by maps $\boldsymbol{m}\colon \boldsymbol{\phi}(\Omega_h)\to \rt$.  Deformations are assumed to be injective and orientation-preserving in order to exclude the interpenetration of matter, and satisfy clamped boundary conditions \cite{lecumberry.mueller}. 
Also, for sufficiently low constant temperature, {\MMM deformations and magnetizations are subject to the magnetic saturation constraint \cite{brown,james.kinderlehrer,rybka.luskin} which, up to normalization, reads
\begin{equation*}
    \text{$|\boldsymbol{m}\circ \boldsymbol{\phi}|\det \nabla \boldsymbol{\phi}=1$ in $\Omega_h$.}
\end{equation*}
The magnetoelastic energy per unit volume is accounted by the functional 
\begin{equation}
\label{eqn:intro-energy}
    (\boldsymbol{\phi},\boldsymbol{m})\mapsto \frac{1}{h^{\beta+1}}\int_{\Omega_h} W_h(\nabla \boldsymbol{\phi},\boldsymbol{m}\circ \boldsymbol{\phi})\,\d\boldsymbol{X}+\frac{1}{h}\int_{\boldsymbol{\phi}(\Omega_h)}|\nabla \boldsymbol{m}|^2\,\d\boldsymbol{\xi}+\frac{1}{2h}\int_{\rt} |\nabla \psi_{\boldsymbol{m}}|^2\,\d\boldsymbol{\xi}.
\end{equation}
This consists of three contributions: the elastic energy, which is rescaled by $h^{\beta+1}$ with $\beta>6$ and depends on the elastic density $W_h$; the exchange energy, that penalizes spatial changes of magnetizations; and the magnetostatic energy, which involves the stray field potential $\psi_{\boldsymbol{m}}\colon \R^3 \to \R^3$ given by a solution of  Maxwell equations:
\begin{equation*}
    \text{$\Delta \psi_{\boldsymbol{m}}=\div \left (\chi_{\boldsymbol{\phi}(\Omega_h)}\boldsymbol{m}   \right )$ in $\rt$.}
\end{equation*}
In particular, we specify the structure of the elastic energy density $W_h$ in \eqref{eqn:intro-energy}, which is assumed to take the form 
\begin{equation}
    \label{eqn:intro-elastic-density}
    W_h(\boldsymbol{F},\boldsymbol{\lambda})\coloneqq \Phi \left ( \Big (\boldsymbol{I}+c_h \boldsymbol{\lambda}(\det \boldsymbol{F})\otimes \boldsymbol{\lambda}(\det \boldsymbol{F}) \Big )^{-1} \boldsymbol{F}   \right ),
\end{equation}
where $\Phi$ is a frame indifferent energy density satisfying physical growth conditions and $c_h\sim h^{\beta/2}$ is a positive parameter. Similar expressions, with the nematic director in place of the magnetization, are widely accepted in the context of liquid crystals \cite{desimone.teresi} and we refer to \cite{agostiniani.desimone} for a specific example in the case of plates. Expression of the elastic energy that are formally analogous to the one in \eqref{eqn:intro-elastic-density} have already been considered in the dimension reduction of prestrained materials \cite{lewicka1,lewicka2} and heterogeneous multilayers \cite{schmidt1,schmidt2}.

Assuming that $\Phi$ is minimized at the identity, formula \eqref{eqn:intro-elastic-density} induces the competition of deformation gradient and magnetization in the minimization of the elastic energy which characterizes magnetostrictive effects. Therefore, the constitutive assumption in \eqref{eqn:intro-elastic-density} makes this setting more realistic compared with the one in \cite{bresciani}, where the elastic energy is minimized at the identity independently on the magnetization. }

Our main results are contained in Theorem \ref{thm:gamma-conv} and Theorem \ref{thm:convergence-almost-minimizers} for the static setting, and in Theorem \ref{thm:evolutionary-gamma-convergence} and Theorem \ref{thm:convergence-AIMP} for the quasistatic setting. The enunciation of these results requires the specification of the setting and the introduction of a considerable amount of notation. Therefore, we limit ourselves to briefly describe them and we postpone the precise statements to Sections \ref{sec:static-setting} and \ref{sec:quasistatic-setting}. 

In Theorem \ref{thm:gamma-conv}, we compute the $\Gamma$-limit of the magnetoelastic energy in \eqref{eqn:intro-energy}, as $h\to 0^+$. This is computed with respect to the convergence of the averaged displacements \cite{friesecke.james.mueller2}, and of the composition of magnetizations with deformations.
The limiting energy that we obtain is purely Lagrangian and is naturally given by integrals on the section $S$. The elastic term in the reduced model is obtained by the linearization of $\Phi$ at the identity similarly to \cite{friesecke.james.mueller2}. In contrast with \cite{bresciani}, this term exhibits a strong coupling between elastic and magnetic variables {\MMM in agreement with models of linearized magnetoelasticity \cite{desimone.james}}. 
Also, as in \cite{carbou,gioia.james}, the  magnetostatic term simplifies substantially in the reduced model.

In Theorem \ref{thm:convergence-almost-minimizers}, we consider applied loads given by mechanical forces and external magnetic fields, all dependent on the thickness of the plate. In particular, the energy contribution determined by the external magnetic field, usually called Zeeman energy, is of Eulerian type. 
The total energy is given by the difference between the magnetoelastic energy and the work of applied loads. {\MMM Having prescribed the asymptotic behaviour of the applied loads,} we prove that sequence of almost minimizers of the total energy converge, as $h \to 0^+$, to minimizers of the corresponding energy functional in the reduced model. Because of the rescaling of the elastic energy, the analysis in quite involved since the coercivity of the total energy functional is not immediate.  

In the quasistatic setting we adopt the framework of  rate-independent processes with the notion of energetic solution \cite{mielke.roubicek}. We consider time-dependent applied loads and {\MMM we introduce the dissipation distance
\begin{equation*}
    \left((\boldsymbol{\phi},\boldsymbol{m}),(\widehat{\boldsymbol{\phi}},\widehat{\boldsymbol{m}})\right )\mapsto \frac{1}{h}\int_{\Omega_h} |\boldsymbol{m}\circ \boldsymbol{\phi}\det\nabla\boldsymbol{\phi}-\widehat{\boldsymbol{m}}\circ \widehat{\boldsymbol{\phi}}\det\nabla\widehat{\boldsymbol{\phi}}|\,\d\boldsymbol{X}.
\end{equation*}

Our results in the quasistatic setting are in the spirit of evolutionary $\Gamma$-convergence \cite{mielke.roubicek.stefanelli}. The first one, namely Theorem \ref{thm:evolutionary-gamma-convergence}, states the convergence energetic solutions for the bulk model to energetic solutions of the reduced model, as $h\to 0^+$. Here,  the existence of energetic solutions for the bulk model is part of the assumptions, while the existence for the reduced model follows as a byproduct. However, our setting is compatible with the existence of energetic solutions for the bulk model in the sense that, under additional assumption on the density $\Phi$,  this can be ensured. We refer to  Remark \ref{rem:ex-ensol-bulk} for more details.  
}

Subsequently, we present a variant of the previous convergence result. For every $h>0$, we consider the approximate incremental minimization problem \cite{mielke.roubicek.stefanelli}, a relaxed version of incremental minimization problem that has been introduced in order to cope with the possible lack of minimizers of energy functionals. Indeed, this approximate problem always admits solutions. 
In Theorem \ref{thm:convergence-AIMP}, we show that, given a sequence of partitions of the time interval whose size vanish together with some tolerances, as $h \to 0^+$, the piecewise-constant interpolants corresponding to solutions of the approximate incremental minimization problems for suitably well prepared initial data converge, as $h \to 0^+$, to an energetic solution of the reduced model. 

We emphasize that all the results in this paper are achieved without resorting on any regularization of the energy. However, our argument to prove the compactness of magnetizations works only under some restriction on  the scalings. Precisely, the scaling of the elastic energy in \eqref{eqn:intro-energy} has to satisfy the condition $\beta>6 \vee p$, where $p>3$ is the integrability exponent of deformations, while the linearized von K\'{a}rm\'{a}n regime corresponds to $\beta>4$. {\MMM Note that the restriction $p>3$ is merely technical. Indeed, although some techniques to tackle the case $p>2$ have been developed in the literature \cite{barchiesi.henao.moracorral}, these are not sufficient for our arguments which rely on explicit estimates on the rate of convergence of the deformation towards the identity.}

The paper is structured as follows. In Section \ref{sec:basic},  we introduce the mathematical model and we list all the assumptions. In Section \ref{sec:static-setting}, we address the static setting: Theorem \ref{thm:gamma-conv} and Theorem \ref{thm:convergence-almost-minimizers} are stated and proved in Subsection \ref{subsec:gamma-convergence} and Subsection \ref{subsec:cam}, respectively. Finally, Section \ref{sec:quasistatic-setting} is devoted to the quasistatic setting.  Theorem \ref{thm:evolutionary-gamma-convergence} and Theorem \ref{thm:convergence-AIMP} are presented in Subsections \ref{subsec:egamma} and Subsection \ref{subsec:conv-AIMP}, respectively. We conclude with Subsection \ref{subsec:alternative-dissipation} by briefly mentioning an alternative choice for the dissipation distance.

\subsection*{Notation} For scalars $a,b\in \R$, we use the notation $a \wedge b\coloneqq \min \{a,b\}$ and $a \vee b \coloneqq \max \{a,b\}$. 
Given $\boldsymbol{a}=(a^1,a^2,a^3)^\top\in \R^3$, we set $\boldsymbol{a}'\coloneqq (a^1,a^2)^\top\in \R^2$. The null vector in $\R^3$ is denoted by $\boldsymbol{0}$, so that $\boldsymbol{0}'$ is the null vector in $\R^2$. The same notation applies also to space variables, and $\nabla'$ and $(\nabla')^2$ denote the gradient and the Hessian with respect to the first two variables, respectively. Given $\boldsymbol{A}=(A^i_j)^{i=1,2,3}_{j=1,2,3}\in \rtt$, we set $\boldsymbol{A}''\coloneqq (A^i_j)^{i=1,2}_{j=1,2}\in \rtwtw$. The null matrix and the identity matrix in $\rtt$ are denoted by $\boldsymbol{O}$ and $\boldsymbol{I}$, thus $\boldsymbol{O}''$ and $\boldsymbol{I}''$ are the corresponding matrices in $\rtwtw$. The tensor product of $\boldsymbol{a},\boldsymbol{b}\in \R^3$ is given by $\boldsymbol{a}\otimes \boldsymbol{b}\in \rtt$ where $(\boldsymbol{a}\otimes \boldsymbol{b})^i_j\coloneqq a^ib^j$ for every $i,j\in \{1,2,3\}$. The identity map on $\rt$ is denoted by $\boldsymbol{id}$. 

We denote general points in the physical (unscaled) space, in the reference space and in the actual space by $\boldsymbol{X}$, $\boldsymbol{x}$ and $\boldsymbol{\xi}$, respectively. Accordingly, the integration with respect to the three-dimensional Lebesgue measure will be denoted by $\d\boldsymbol{X}$, $\d \boldsymbol{x}$ and $\d \boldsymbol{\xi}$, respectively. The integration with respect to the one and the two-dimensional Hausdorff measure in the reference space will be denoted by $\d\boldsymbol{l}$ and $\d\boldsymbol{a}$, respectively.
We denote by $\chi_A$ the characteristic function of a set $A\subset \R^N$, where $N\in\{1,2,3\}$.
We will use standard notation for Lebesgue, Sobolev and Bochner spaces, and for spaces of functions of bounded variation. 
Given $S \subset \R^2$ open  and an embedded submanifold $\mathcal{M}\subset \R^M$, where $M \in \N$, we denote by $W^{1,q}(S;\mathcal{M})$, where $1 \leq q <\infty$, the set of maps $\boldsymbol{\eta}\in W^{1,q}(S;\R^M)$ such that $\boldsymbol{\eta}(\boldsymbol{x}')\in \mathcal{M}$ for almost every  $\boldsymbol{x}'\in S$.
 In the following, $\mathcal{M}$ will be either  the unit sphere $\S^2 \subset\R^3$ or the special orthogonal group  $SO(3)\subset \rtt$.
Finally, the topological degree of a map $\boldsymbol{y}\in C^0(\closure{\Omega};\R^3)$, where $\Omega\subset \R^3$ is open and bounded, on $\Omega$ at $\boldsymbol{\xi}\in \R^3\setminus \boldsymbol{y}(\partial \Omega)$ will be denoted by $\deg(\boldsymbol{y},\Omega,\boldsymbol{\xi})$.

We make use of the Landau symbols `$o$' and `$O$'. When referred to vectors or matrices, these are to
be understood with respect to the maximum of their components. We will adopt the common convention
of denoting by $C, C_1, C_2 \dots$  positive constants that can change from line to line. We will identify functions defined on the plane with functions defined on the three-dimensional space that are independent on the third variable. In general, we will think at the parameter $h>0$ as varying along a sequence even if this is not mentioned. The particular sequence of thicknesses considered will be specified only in a few circumstances, when this is particularly important for the understanding. 

\section{Basic setting}
\label{sec:basic}

In this section we describe the general setting of the paper. First, in Subsection \ref{subsec:mechanical-model}, we introduce the mechanical model and we list all the assumptions. Then, in Subsection \ref{subsec:cov-rescaling}, we perform a standard change of variables in order to work on a fixed domain.

\subsection{The static model}
\label{subsec:mechanical-model}
Let $\Omega_h\coloneqq S \times hI$  represent a thin magnetoelastic plate in its reference configuration. The section $S \subset \R^2$ is a bounded connected Lipschitz domain, while the parameter $h>0$ gives the thickness of the plate and $I\coloneqq (-1/2,1/2)$. 

The plate experiences elastic deformations given by maps $\boldsymbol{\phi}\in W^{1,p}(\Omega_h;\R^3)$ for some fixed $p>3$. By the Morrey embedding, any such map admits a continuous representative with whom it is systematically identified. Every deformation $\boldsymbol{\phi}$ is required to be {orientation-preserving}, namely to satisfy the constraint $\det \nabla \boldsymbol{\phi}>0$ almost everywhere in $\Omega_h$, and to be {almost everywhere injective}. 
This means that there exists a set $X\subset \Omega_h$ with $\leb(X)=0$ such that $\boldsymbol{\phi}\restr{\Omega_h \setminus X}$ is injective.
Recall that any such map $\boldsymbol{\phi}$ has both {Lusin properties} (N) and (N${}^{-1}$), that is, $\leb(\boldsymbol{\phi}(X))=0$ for every $X \subset \Omega_h$ with $\leb(X)=0$ and $\leb(\boldsymbol{\phi}^{-1}(Y))=0$ for every $Y \subset \R^3$ with $\leb(Y)=0$. Also,  the area formula and the change-of-variable formula hold for such a map \cite{marcus.mizel}. We impose clamped boundary conditions by requiring each deformation $\boldsymbol{\phi}$ to satisfy
\begin{equation}
\label{eqn:clamped-unscaled}
    \text{$\boldsymbol{\phi}=\boldsymbol{id}$ on $\partial S \times h I$.}
\end{equation}

Given a deformation $\boldsymbol{\phi}$, we define the corresponding deformed configuration as $\Omega_h^{\boldsymbol{\phi}}\coloneqq \boldsymbol{\phi}({\Omega}_h)\setminus \boldsymbol{\phi}(\partial \Omega_h)$. This set is open \cite[Lemma 2.1]{bresciani.davoli.kruzik} and we have $\leb(\boldsymbol{\phi}(\Omega_h)\setminus \Omega_h^{\boldsymbol{\phi}})=0$ thanks the Lusin property (N). {\MMM Magnetizations are defined as maps $\boldsymbol{m}\in W^{1,2}(\Omega_h^{\boldsymbol{\phi}};\R^3)$ subject to the saturation constraint:
\begin{equation}
\label{eqn:saturation-unscaled}
    \text{$|\boldsymbol{m}\circ \boldsymbol{\phi}|\det \nabla \boldsymbol{\phi}=1$ a.e. in $\Omega_h$.}
\end{equation}
}

Neglecting the material parameters, the energy corresponding to a deformation $\boldsymbol{\phi}\in W^{1,p}(\Omega_h;\R^3)$ and a magnetization $\boldsymbol{m}\in W^{1,2}(\Omega_h^{\boldsymbol{\phi}};\R^3)$,  is given by
\begin{equation}
    \label{eqn:energy-G_h}
    G_h(\boldsymbol{\phi},\boldsymbol{m})\coloneqq \frac{1}{h^{\beta}}\int_{\Omega_h} W_h(\nabla \boldsymbol{\phi},\boldsymbol{m}\circ \boldsymbol{\phi})\,\d \boldsymbol{X}+ \int_{\Omega_h^{\boldsymbol{\phi}}} |\nabla \boldsymbol{m}|^2\,\d \boldsymbol{\xi}+\frac{1}{2} \int_{\R^3} |\nabla\psi_{\boldsymbol{m}}|^2\,\d \boldsymbol{\xi}.
\end{equation}
The first term in \eqref{eqn:energy-G_h} represents the {elastic energy} and it is rescaled according to the linearized von K\'{a}rm\'{a}n regime \cite{friesecke.james.mueller2}. Precisely, we assume 
\begin{equation*}
    \beta>6 \vee p.
\end{equation*}
Note that, by the Lusin property (N${}^{-1}$), the composition $\boldsymbol{m}\circ \boldsymbol{\phi}$ is measurable and its equivalence class does not depend on the choice of the representative of $\boldsymbol{m}$.

{\MMM
The elastic energy density $W_h \colon \mathbb{X} \to [0,+\infty)$, where we set
\begin{equation*}
    \mathbb{X}\coloneqq \left \{(\boldsymbol{F},\boldsymbol{\lambda})\in \rtt_+ \times \rt:\:\: (\det\boldsymbol{F})|\boldsymbol{\lambda}|=1  \right \},
\end{equation*}
is continuous and takes the form
\begin{equation}
    \label{eqn:density-W_h}
    W_h(\boldsymbol{F},\boldsymbol{\lambda})\coloneqq \Phi \left(\boldsymbol{L}_h((\det\boldsymbol{F})\boldsymbol{\lambda})^{-1}\boldsymbol{F}\right)
\end{equation}
for some function $\Phi\colon \rtt_+ \to [0,+\infty)$.
In \eqref{eqn:density-W_h},  we define $\boldsymbol{L}_h\colon \mathbb{S}^2\to\rtt_+$ by setting 
\begin{equation}
    \label{eqn:prestrain}
    \boldsymbol{L}_h(\boldsymbol{z})\coloneqq \boldsymbol{I}+c_h \boldsymbol{z}\otimes \boldsymbol{z},
\end{equation}
where $c_h>0$. Regarding the asymptotic behavior of $(c_h)$, we assume the existence of the limit
\begin{equation}
\label{eqn:ch-limit}
    c_0\coloneqq \lim_{h \to 0^+} \frac{c_h}{h^{\beta/2}}>0.
\end{equation}
A direct computation shows that
\begin{equation}
\label{eqn:prestrain-determinant}
    \det \boldsymbol{L}_h(\boldsymbol{z})=1+c_h
\end{equation}
for every $\boldsymbol{z}\in \S^2$. In particular, the matrix $\boldsymbol{L}_h(\boldsymbol{z})$ is invertible and its inverse is given by
\begin{equation}
    \label{eqn:prestrain-inverse}
    \boldsymbol{L}_h(\boldsymbol{z})^{-1}= \boldsymbol{I}-\frac{c_h}{1+c_h} \boldsymbol{z}\otimes \boldsymbol{z}.
\end{equation}
}

The function $\Phi$ is assumed to have the following properties:
\begin{enumerate}[\bf (i)]
	\item \textbf{Normalization:} 
	\begin{equation}
		\label{eqn:normalization-Phi}
		\Phi(\boldsymbol{I})=0=\min\,\Phi,
	\end{equation}
	\item \textbf{Frame indifference:}
	\begin{equation}
		\label{eqn:frame-indifference-Phi}
		\forall\,\boldsymbol{R}\in SO(3),\:\forall\,\boldsymbol{\Xi}\in \rtt_+, \quad \Phi(\boldsymbol{R}\,\boldsymbol{\Xi})=\Phi(\boldsymbol{\Xi}),
	\end{equation}
	\item \textbf{Growth:} there exist $C_0,C_1,C_2>0$ and {\MMM $a> 1$} such that
	\begin{equation}
	\label{eqn:coercivity-Phi}
	\forall\,\boldsymbol{\Xi}\in \rtt_+, \quad \Phi(\boldsymbol{\Xi})\geq C_0\, \dist^2(\boldsymbol{\Xi};SO(3))\vee \dist^p(\boldsymbol{\Xi};SO(3)),
	\end{equation}
	and
	\begin{equation}
	\label{eqn:coercivity-Phi-det}
		{\MMM \forall\,\boldsymbol{\Xi}\in \rtt_+, \quad \Phi(\boldsymbol{\Xi})\geq \frac{C_1}{(\det\boldsymbol{\Xi})^a}-C_2,}
	\end{equation}
	\item \textbf{Regularity:}
	\begin{equation}
	\label{eqn:regularity-Phi}
	\text{$\Phi$ is continuous and of class $C^2$ in a neighborhood  of $SO(3)$},
	\end{equation}
\end{enumerate}

Given \eqref{eqn:prestrain-inverse}, we have
\begin{equation}
	\label{eqn:frame-indifference-prestrain}
	\forall\,\boldsymbol{R}\in SO(3),\:\forall \,\boldsymbol{z}\in\S^2,\quad \boldsymbol{L}_h(\boldsymbol{R}\,\boldsymbol{z})^{-1}=\boldsymbol{R}\,\boldsymbol{L}_h(\boldsymbol{z})^{-1}\boldsymbol{R}^\top.
\end{equation}
This, together with \eqref{eqn:frame-indifference-Phi}, yields the frame-indifference of $W_h$, namely
\begin{equation}
\label{eqn:frame-indifference}
	\forall\,\boldsymbol{R}\in SO(3),\:\forall\,(\boldsymbol{F},\boldsymbol{\lambda})\in\mathbb{X},\quad W_h(\boldsymbol{R}\boldsymbol{F},\boldsymbol{R}\boldsymbol{\lambda})=W_h(\boldsymbol{F},\boldsymbol{\lambda}).
\end{equation}
From \eqref{eqn:normalization-Phi} and \eqref{eqn:frame-indifference}, we realize that the function  $W_h$ is minimized on the set
\begin{equation*}
\left \{(\boldsymbol{R}\boldsymbol{F},\boldsymbol{R}\boldsymbol{\lambda}):\:\boldsymbol{R}\in SO(3),\:(\boldsymbol{F},\boldsymbol{\lambda})\in \mathbb{X},\: \boldsymbol{F}=\boldsymbol{L}_h((\det \boldsymbol{F})\boldsymbol{\lambda})\right \}  .
\end{equation*}
By \eqref{eqn:coercivity-Phi}, the map $\Phi$ has global $p$-growth and quadratic growth close to $SO(3)$. In particular, there exist $C_1,C_2>0$ such that, for $q\in \{2,p\}$, there holds
\begin{equation*}
\forall\,\boldsymbol{\Xi}\in \rtt_+,\quad \Phi(\boldsymbol{\Xi})\geq C_1\,|\boldsymbol{\Xi}|^q-C_2.
\end{equation*}
The specific form of the growth  condition of $\Phi$ with respect to the determinant in \eqref{eqn:coercivity-Phi-det} is assumed just for simplicity. Actually, because of the rescaling of the elastic energy, we could also assume  $a=1$. More generally, assumption \eqref{eqn:coercivity-Phi-det} can be replaced by the requirement
\begin{equation*}
	\forall\,\boldsymbol{\Xi}\in\rtt_+,\quad \Phi(\boldsymbol{\Xi})\geq \gamma(\det \boldsymbol{\Xi}),
\end{equation*} 
where $\gamma \colon (0,+\infty)\to [0,+\infty]$ is continuous and satisfies 
\begin{equation*}
    \gamma(1)=0=\min\gamma, \qquad \lim_{h \to 0^+}h\gamma(h)=+\infty.
\end{equation*}
Assumptions \eqref{eqn:normalization-Phi} and \eqref{eqn:regularity-Phi} justify  the  second-order Taylor expansion of $\Phi$ close to the identity. Precisely, we have the following:
\begin{equation}
    \label{eqn:Taylor-Phi}
    \forall\, \boldsymbol{\Xi}\in \rtt:\:|\boldsymbol{\Xi}|\ll 1,\quad \Phi(\boldsymbol{I}+\boldsymbol{\Xi})=\frac{1}{2} Q(\boldsymbol{\Xi})+\omega(\boldsymbol{\Xi}).
\end{equation}
The quadratic form $Q$ is defined by $Q(\boldsymbol{\Xi})\coloneqq D^2\Phi(\boldsymbol{I})(\boldsymbol{\Xi},\boldsymbol{\Xi})$, while  $\omega(\boldsymbol{\Xi})=o(|\boldsymbol{\Xi}|^2)$, as $|\boldsymbol{\Xi}|\to 0^+$. 
Note that $Q$ is positive semidefinite and, in turn, convex by \eqref{eqn:normalization-Phi}. Additionally, exploiting \eqref{eqn:coercivity-Phi} and \eqref{eqn:Taylor-Phi}, by arguing as in \cite[p. 927]{mielke.stefanelli} one can show that $Q$ is positive definite on symmetric matrices, that is, there exists $C>0$ such that
\begin{equation}
\label{eqn:Q-coe}
    \forall\,\boldsymbol{\Xi}\in \rtt,\quad Q(\boldsymbol{\Xi})=Q(\sym \boldsymbol{\Xi})\geq C |\sym\,\boldsymbol{\Xi}|^2.
\end{equation}
The last two terms in \eqref{eqn:energy-G_h} are of Eulerian type. The second one is the {exchange energy}, while the third one is the {magnetostatic energy}. This last term involves the stray-field potential $\psi_{\boldsymbol{m}}\colon \R^3 \to \R^3$ which is a weak solution of the {magnetostatic Maxwell equation}:
\begin{equation}
    \label{eqn:Maxwell}
    \Delta \psi_{\boldsymbol{m}}= \mathrm{div}\,(\chi_{\Omega_h^{\boldsymbol{\phi}}}\boldsymbol{m})\quad \text{in $\R^3$.}
\end{equation}
It is proved that weak solutions of \eqref{eqn:Maxwell} exist and are unique up to additive constants \cite[Proposition 8.8]{barchiesi.henao.moracorral}. Therefore, the magnetostatic energy is well defined.

We mention that the magnetostatic term usually comprises other terms such as the anisotropy energy \cite{kruzik.prohl} that here, for simplicity, we are neglecting.

\subsection{Change of variables and rescaling}
\label{subsec:cov-rescaling}
For $h>0$, we introduce the rescaling map $\boldsymbol{\pi}_h$ defined by $\boldsymbol{\pi}_h(\boldsymbol{x})\coloneqq ((\boldsymbol{x}')^\top,h x_3)^\top$ for every $\boldsymbol{x}\in \R^3$. Set $\Omega \coloneqq S \times I$. Given any  $\boldsymbol{\phi}\in W^{1,p}(\Omega_h; \R^3)$, we consider the map $\boldsymbol{y}\coloneqq \boldsymbol{\phi}\circ \boldsymbol{\pi}_h\restr{\Omega} \in W^{1,p}(\Omega;\R^3)$. In view of \eqref{eqn:clamped-unscaled}, this map satisfies
\begin{equation}
\label{eqn:y-cl}
    \text{$\boldsymbol{y}=\boldsymbol{\pi}_h$ on $\partial S \times I$.}
\end{equation}
Also, given the set  $\Omega^{\boldsymbol{y}}\coloneqq \boldsymbol{y}(\Omega) \setminus \boldsymbol{y}(\partial \Omega)$, there holds $\Omega^{\boldsymbol{y}}=\Omega_h^{\boldsymbol{\phi}}$ and, in view of \eqref{eqn:saturation-unscaled}, each magnetization $\boldsymbol{m}\in W^{1,2}(\Omega^{\boldsymbol{y}};\R^3)$ satisfies
\begin{equation}
    \label{eqn:saturation}
    \text{$|\boldsymbol{m}\circ \boldsymbol{y}|\det \nabla \boldsymbol{y}=h$ a.e. in $\Omega$.}
\end{equation}
Therefore, we define the class of admissible states as
\begin{equation}
    \label{eqn:class-Qh}
    \begin{split}
        \mathcal{Q}_h:= \Big \{(\boldsymbol{y},\boldsymbol{m}):\:\boldsymbol{y}\in W^{1,p}(\Omega;\R^3),\:\text{$\det \nabla \boldsymbol{y}>0$ a.e.},\:\text{$\boldsymbol{y}$ a.e. injective,}\:\text{$\boldsymbol{y}=\boldsymbol{\pi}_h$ on $\partial S \times I$,}&\\
    \boldsymbol{m}\in W^{1,2}(\Omega^{\boldsymbol{y}};\R^3),\:\text{$|\boldsymbol{m}\circ \boldsymbol{y}|\det \nabla \boldsymbol{y}=h$ a.e. in $\Omega$} & \Big \}.
    \end{split}
\end{equation}

Recalling \eqref{eqn:energy-G_h} and applying the change-of-variable formula, we obtain
\begin{equation*}
    \label{eqn:energy-rescaling}
    \frac{1}{h}G_h(\boldsymbol{\phi},\boldsymbol{m})=\frac{1}{h^\beta} \int_\Omega W_h(\nabla_h\boldsymbol{y},\boldsymbol{m}\circ\boldsymbol{y})\,\d \boldsymbol{x}+\frac{1}{h}\int_{\Omega^{\boldsymbol{y}}} |\nabla \boldsymbol{m}|^2\,\d \boldsymbol{\xi}+\frac{1}{2h}\int_{\R^3}|\nabla \psi_{\boldsymbol{m}}|^2\,\d \boldsymbol{\xi},
\end{equation*}
where the scaled gradient is defined as $\nabla_h\coloneqq(\nabla',h^{-1}\partial_3)$. 
Hence, we define the energy functional $E_h \colon \mathcal{Q}_h\to [0,+\infty)$  by setting
\begin{equation}
    \label{eqn:energy-E_h}
    \begin{split}
    E_h(\boldsymbol{q}) &\coloneqq \frac{1}{h^\beta} \int_\Omega W_h(\nabla_h\boldsymbol{y},\boldsymbol{m}\circ\boldsymbol{y})\,\d \boldsymbol{x}+\frac{1}{h}\int_{\Omega^{\boldsymbol{y}}} |\nabla \boldsymbol{m}|^2\,\d \boldsymbol{\xi}+\frac{1}{2h}\int_{\R^3}|\nabla \psi_{\boldsymbol{m}}|^2\,\d \boldsymbol{\xi},
    \end{split}
\end{equation}
where $\boldsymbol{q}=(\boldsymbol{y},\boldsymbol{m})$. 
We denote the three terms on the right-hand side  by $E_h^{\mathrm{el}}(\boldsymbol{q})$, $E_h^{\mathrm{exc}}(\boldsymbol{q})$ and $E_h^{\mathrm{mag}}(\boldsymbol{q})$, respectively.
Given \eqref{eqn:Maxwell}, the function $\psi_{\boldsymbol{m}}$ in \eqref{eqn:energy-E_h} is a weak solution of the equation
\begin{equation*}
    \Delta \psi_{\boldsymbol{m}}= \mathrm{div}\,(\chi_{\Omega^{\boldsymbol{y}}}\boldsymbol{m})\quad \text{in $\R^3$.}
\end{equation*}
More explicitly, this means that $\psi_{\boldsymbol{m}}\in V^{1,2}(\R^3)$ and satisfies
\begin{equation}
    \label{eqn:Maxwell-y-weak}
    \forall\,\varphi\in V^{1,2}(\R^3),\quad \int_{\R^3} \nabla \psi_{\boldsymbol{m}}\cdot\nabla\varphi\,\d\boldsymbol{\xi}=\int_{\R^3} \chi_{\Omega^{\boldsymbol{y}}}\boldsymbol{m} \cdot\nabla\varphi\,\d\boldsymbol{\xi}.
\end{equation}
Here, we adopt the same notation in \cite[Subsection 2.7.3]{necas} and we set
\begin{equation}
    \label{eqn:necas}
    V^{1,2}(\R^3)\coloneqq \left \{ \varphi\in L^2_{\mathrm{loc}}(\R^3):\:\nabla \varphi\in L^2(\R^3;\R^3)\right \}.
\end{equation}
In particular, testing \eqref{eqn:Maxwell-y-weak} with $\varphi=\psi_{\boldsymbol{m}}$ and applying the H\"{o}lder inequality, we obtain
\begin{equation}
    \label{eqn:maxwell-stability}
    \|\nabla \psi_{\boldsymbol{m}}\|_{L^2(\R^3;\R^3)}\leq \|\chi_{\Omega^{\boldsymbol{y}}}\boldsymbol{m}\|_{L^2(\R^3;\R^3)}.
\end{equation}

\begin{remark}[Invariance with respect to rigid motions]
\label{rem:invariance-rigid-motions}
The functional $E_h$ is invariant with respect to rigid motions. Let $\boldsymbol{q}=(\boldsymbol{y},\boldsymbol{m})\in\mathcal{Q}_h$ and let $\boldsymbol{T}$ be a rigid motion of the form $\boldsymbol{T}(\boldsymbol{\xi})\coloneqq \boldsymbol{Q}\boldsymbol{\xi}+\boldsymbol{c}$ for every $\boldsymbol{\xi}\in\R^3$, where $\boldsymbol{Q}\in SO(3)$ and $\boldsymbol{c}\in\R^3$. If we set $\widetilde{\boldsymbol{q}}=(\widetilde{\boldsymbol{y}},\widetilde{\boldsymbol{m}})\in\mathcal{Q}_h$ with $\widetilde{\boldsymbol{y}}\coloneqq \boldsymbol{T}\circ\boldsymbol{y}$ and $\widetilde{\boldsymbol{m}}\coloneqq \boldsymbol{Q}\boldsymbol{m}\circ \boldsymbol{T}^{-1}$, then there holds $E_h(\widetilde{\boldsymbol{q}})=E_h(\boldsymbol{q})$. Indeed, by \eqref{eqn:frame-indifference}, $E_h^{\mathrm{el}}(\widetilde{\boldsymbol{q}})=E_h^{\mathrm{el}}(\boldsymbol{q})$ and, by the change-of-variable formula, $E_h^{\mathrm{exc}}(\widetilde{\boldsymbol{q}})=E_h^{\mathrm{exc}}(\boldsymbol{q})$. Moreover, if $\psi_{\boldsymbol{m}}$ is a stray field potential corresponding to $\boldsymbol{q}$, then we check that $\psi_{\boldsymbol{m}}\circ \boldsymbol{T}^{-1}$ is a stray field potential corresponding to $\widetilde{\boldsymbol{q}}$. Clearly, this yields $E_h^{\mathrm{mag}}(\widetilde{\boldsymbol{q}})=E_h^{\mathrm{mag}}(\boldsymbol{q})$.
\end{remark}

\begin{remark}[Existence of minimizers for the bulk model]
\label{rem:existence-min-bulk}
In the present work, we do not deal with the problem of the existence of minimizers and we do not even specify the topology on $\mathcal{Q}_h$.
We just mention that, without further assumptions on the elastic energy density $W_h$, the functional $E_h$ in \eqref{eqn:energy-E_h} does not necessarily attain its infimum.
However, {\MMM if the function $\Phi$ in \eqref{eqn:density-W_h} satisfies feasible polyconvexity assumptions, then the existence of minimizers of $E_h$  can be proved \cite{bresciani.quasistatic,bresciani.davoli.kruzik}.}
\end{remark}

\section{Static setting}
\label{sec:static-setting}
In this section we study the asymptotic behavior of the energy $E_h$ in \eqref{eqn:energy-E_h}, as $h \to 0^+$, in the static case. First, in Subsection \ref{subsec:gamma-convergence}, we compute the $\Gamma$-limit the sequence $(E_h)$. Then, in Subsection \ref{subsec:cam}, we consider applied loads and we prove that sequences of almost minimizers of the total energy converge to minimizers of the corresponding energy in the reduced model. 

\subsection{\texorpdfstring{Static $\boldsymbol{\Gamma}$}{gamma}-convergence}
\label{subsec:gamma-convergence}
We introduce some notation that is going to be employed in the rest of the paper. For $h>0$ and $\boldsymbol{q}=(\boldsymbol{y},\boldsymbol{m})\in \mathcal{Q}_h$,  we define the {(scaled) horizontal} and {vertical averaged displacements} and the {(scaled) first moment}, respectively
\begin{equation*}
    \mathcal{U}_h(\boldsymbol{q})\colon S \to \R^2,\qquad \mathcal{V}_h(\boldsymbol{q})\colon S \to \R, \qquad \mathcal{W}_h(\boldsymbol{q})\colon S \to \R^3, 
\end{equation*}
by setting
\begin{align*}
    \mathcal{U}_h(\boldsymbol{q})(\boldsymbol{x}')&\coloneqq \frac{1}{h^{\beta/2}}\int_I \big ( \boldsymbol{y}'(\boldsymbol{x}',x_3)-\boldsymbol{x}'\big ) \,\d x_3,\\
    \mathcal{V}_h(\boldsymbol{q})(\boldsymbol{x}')&\coloneqq \frac{1}{h^{\beta/2-1}}\int_I y^3(\boldsymbol{x}',x_3)\,\d x_3,\\
    \mathcal{W}_h(\boldsymbol{q})(\boldsymbol{x}')&\coloneqq\frac{1}{h^{\beta/2}} \int_I x_3 \big ( \boldsymbol{y}(\boldsymbol{x}',x_3)-\boldsymbol{\pi}_h(\boldsymbol{x}',x_3)\big )\,\d x_3,
\end{align*}
for every $\boldsymbol{x}'\in S$. 
Furthermore, {\MMM we define
\begin{equation*}
	\mathcal{M}_h(\boldsymbol{q})\colon \Omega \to \R^3, \qquad \mathcal{N}_h(\boldsymbol{q})\colon \Omega \to \rtt, \qquad \mathcal{Z}_h(\boldsymbol{q})\colon \Omega \to \S^2,
\end{equation*}
by setting
\begin{align}
	\label{eqn:notation-mu}
	\mathcal{M}_h(\boldsymbol{q})&\coloneqq \left(\chi_{\Omega^{\boldsymbol{y}}}\boldsymbol{m}  \right)\circ \boldsymbol{\pi}_h,\\
	\label{eqn:notation-nu}
	\mathcal{N}_h(\boldsymbol{q})&\coloneqq \left(\chi_{\Omega^{\boldsymbol{y}}}\nabla\boldsymbol{m}  \right)\circ \boldsymbol{\pi}_h,\\
	\label{eqn:lagrangian-magnetization}
	\mathcal{Z}_h(\boldsymbol{q})&\coloneqq \boldsymbol{m}\circ \boldsymbol{y}\det \nabla_h\boldsymbol{y}.
\end{align}
We stress that the map $\mathcal{Z}_h(\boldsymbol{q})$ is sphere-valued as a consequence of \eqref{eqn:saturation}.
}

Recall the quadratic form $Q$ in \eqref{eqn:Taylor-Phi}. As in \cite{friesecke.james.mueller2}, the reduced quadratic form is defined  by
\begin{equation}
    \label{eqn:Q-red}
    Q_{\mathrm{red}}(\boldsymbol{\Sigma})\coloneqq \min \Bigg \{  Q \left( \left(\renewcommand\arraystretch{1.2} 
    \begin{array}{@{}c|c@{}}   \boldsymbol{\Sigma}  & \boldsymbol{0}'\\ \hline  (\boldsymbol{0}')^{\top} & 0 \end{array} \right)+\boldsymbol{c}\otimes \boldsymbol{e}_3+\boldsymbol{e}_3 \otimes \boldsymbol{c} \right): \:\:\boldsymbol{c} \in \R^3 \Bigg \}
\end{equation}
for every $\boldsymbol{\Sigma}\in \rtwtw$. The  positive definiteness and the convexity of $Q_{\mathrm{red}}$ follow from that of $Q$. Moreover, from \eqref{eqn:Q-coe}, we deduce that $Q_{\mathrm{red}}$ is also positive definite on symmetric matrices, namely, there exists $C>0$ such that
\begin{equation}
\label{eqn:Q-red-coe}
    \forall\,\boldsymbol{\Sigma}\in \rtwtw,\quad Q_{\mathrm{red}}(\boldsymbol{\Sigma})=Q_{\rm red}(\sym \boldsymbol{\Sigma})\geq C |\sym\,\boldsymbol{\Sigma}|^2.
\end{equation}

Set 
\begin{equation}
\label{eqn:class-Q0}
    \mathcal{Q}_0\coloneqq W^{1,2}_0(S;\R^2)\times W^{2,2}_0(S)\times W^{1,2}(S;\S^2).
\end{equation}
The $\Gamma$-limit of the functionals $(E_h)$ in \eqref{eqn:energy-E_h}, as $h \to 0^+$, is given by $E_0\colon \mathcal{Q}_0\to [0,+\infty)$
defined as
\begin{equation}
    \label{eqn:energy-E_0}
        \begin{split}
        E_0(\boldsymbol{q}_0) &\coloneqq \frac{1}{2}\int_S Q_{\mathrm{red}}(\sym\nabla'\boldsymbol{u}-c_0\boldsymbol{\zeta}'\otimes\boldsymbol{\zeta}')\,\d\boldsymbol{x}'+
        \frac{1}{24}\int_S Q_{\mathrm{red}}((\nabla')^2 v)\,\d\boldsymbol{x}'\\
        &\;+\int_S |\nabla'\boldsymbol{\zeta}|^2\,\d\boldsymbol{x}'+\frac{1}{2}\int_S |\zeta^3|^2\,\d \boldsymbol{x}',
        \end{split}
\end{equation}
where $\boldsymbol{q}_0=(\boldsymbol{u},v,\boldsymbol{\zeta})$.
We denote the sum of the first two terms on the right-hand side by $E_0^{\mathrm{el}}(\boldsymbol{q}_0)$ and the last two terms on the right-hand side  by $E_0^{\mathrm{exc}}(\boldsymbol{q}_0)$ and $E_0^{\mathrm{mag}}(\boldsymbol{q}_0)$, respectively.
Note that the limiting functional $E_0$ is purely Lagrangian and that it trivially admits minimizers.

Our first main result claims the $\Gamma$-convergence of $(E_h)$ to $E_0$, as $h \to 0^+$, and the equi-coercivity of the sequence $(E_h)$. 

\begin{theorem}[Static $\boldsymbol{\Gamma}$-convergence]
\label{thm:gamma-conv}
Assume $p>3$ and $\beta>6\vee p$. Suppose that the elastic energy density $W_h$ has the form in \eqref{eqn:density-W_h}, where the function $\Phi$ satisfies \eqref{eqn:normalization-Phi}--\eqref{eqn:regularity-Phi}. 
\begin{enumerate}[(i)]
    \item \textbf{(Compactness and lower bound).} Let $(\boldsymbol{q}_h)$ with $\boldsymbol{q}_h=(\boldsymbol{y}_h,\boldsymbol{m}_h)\in \mathcal{Q}_h$ be such that 
    \begin{equation}
    \label{eqn:be}
        \sup_{h>0} E_h(\boldsymbol{q}_h)\leq C.
    \end{equation}
    Then, there exists $\boldsymbol{q}_0=(\boldsymbol{u},v,\boldsymbol{\zeta})\in \mathcal{Q}_0$ such that, up to subsequences, the following convergences hold,  as $h \to 0^+$:
    \begin{align}
        \label{eqn:gamma-horizontal}
        \text{${\boldsymbol{u}}_h\coloneqq \mathcal{U}_h({\boldsymbol{q}}_h)$}&\text{$\wk {\boldsymbol{u}}$ in $W^{1,2}(S;\R^2)$,}\\
        \label{eqn:gamma-vertical}
        \text{${v}_h\coloneqq \mathcal{V}_h({\boldsymbol{q}}_h)$}&\text{$\to {v}$ in $W^{1,2}(S)$,}\\
        \label{eqn:gamma-lagrangian}
        \text{${\boldsymbol{z}}_h\coloneqq \mathcal{Z}_h({\boldsymbol{q}}_h)$}&\text{$\to {\boldsymbol{\zeta}}$ in $L^1(\Omega;\R^3)$.}
    \end{align}
    Moreover, the following inequality holds:
    \begin{equation}
    \label{eqn:gamma-liminf}
    E_0({\boldsymbol{q}}_0) \leq \liminf\limits_{h \to 0^+}E_h(\boldsymbol{q}_h).
    \end{equation}
    \item \textbf{(Optimality of the lower bound).} For every $\widehat{\boldsymbol{q}}_0=(\widehat{\boldsymbol{u}}, \widehat{v},\widehat{\boldsymbol{\zeta}})\in \mathcal{Q}_0$, there exists  $(\widehat{\boldsymbol{q}}_h)$ with $\widehat{\boldsymbol{q}}_h\in \mathcal{Q}_h$ such that the following convergences hold, as $h \to 0^+$:
    \begin{align*}
        \text{$\widehat{\boldsymbol{u}}_h\coloneqq \mathcal{U}_h(\widehat{\boldsymbol{q}}_h)$}&\text{$\to \widehat{\boldsymbol{u}}$ in $W^{1,2}(S;\R^2)$,}\\
        \text{$\widehat{v}_h\coloneqq \mathcal{V}_h(\widehat{\boldsymbol{q}}_h)$}&\text{$\to \widehat{v}$ in $W^{1,2}(S)$,}\\
        \text{$\widehat{\boldsymbol{z}}_h\coloneqq \mathcal{Z}_h(\widehat{\boldsymbol{q}}_h)$}&\text{$\to \widehat{\boldsymbol{\zeta}}$ in $L^1(\Omega;\R^3)$.}
    \end{align*}
    Moreover,  the following equality holds:
    \begin{equation}
        \label{eqn:opt-lb-limit}
         E_0(\widehat{\boldsymbol{u}},\widehat{v},\widehat{\boldsymbol{\zeta}}) = \lim_{h \to 0^+} E_h(\widehat{\boldsymbol{q}}_h).
    \end{equation} 
\end{enumerate}
\end{theorem}

Note that Theorem \ref{thm:gamma-conv} is not a proper $\Gamma$-convergence statement in the sense of the abstract definition \cite{braides} since the functionals $E_h$ and $E_0$ are defined on different spaces. However,  Theorem \ref{thm:gamma-conv} can be reformulated as a rigorous $\Gamma$-convergence statement similarly to \cite[Corollary 3.4]{bresciani}.

\begin{remark}[More general boundary conditions]
\label{rem:dirichlet-bc}
More general Dirichlet boundary conditions, like the ones in \cite{lecumberry.mueller}, can be considered in Theorem \ref{thm:gamma-conv}. Precisely, let $\overline{\boldsymbol{u}}\in W^{2,\infty}(S;\R^2)$ and $\overline{v}\in W^{3,\infty}(S)$. For $h>0$, let the deformation $\overline{\boldsymbol{y}}_h\in W^{1,\infty}(\Omega;\R^3)$ be defined as
\begin{equation*}
    \overline{\boldsymbol{y}}_h\coloneqq \boldsymbol{\pi}_h + h^{\beta/2} \left ( \begin{matrix}  \overline{\boldsymbol{u}}  \\ 0\end{matrix}   \right )+ h^{\beta/2-1} \left ( \begin{matrix}  \boldsymbol{0}'  \\ \overline{v}\end{matrix}   \right )-h^{\beta/2} x_3 \left ( \begin{matrix}  \nabla'\overline{v}  \\ 0\end{matrix}   \right ).
\end{equation*}
If $\Gamma\subset \partial S$ is given by a finite union of closed connected subsets of $\partial S$ with nonempty interior in the relative topology, then Theorem \ref{thm:gamma-conv} still holds true if we replace the boundary condition in \eqref{eqn:y-cl} with 
\begin{equation*}
    \text{$\boldsymbol{y}=\overline{\boldsymbol{y}}_h$ on $\Gamma\times I$}.
\end{equation*}
Accordingly, the limiting class $\mathcal{Q}_0$ in \eqref{eqn:class-Q0} needs to be replaced by the set
\begin{equation*}
    \left \{(\boldsymbol{u},v,\boldsymbol{\zeta})\in W^{1,2}(S;\R^2)\times W^{2,2}(S)\times W^{1,2}(S;\S^2):\:\text{$\boldsymbol{u}=\overline{\boldsymbol{u}}$ on $\Gamma$, $v=\overline{v}$ on $\Gamma$, $\nabla'v=\nabla'\overline{v}$ on $\Gamma$} \right  \}.
\end{equation*}
The main changes concern the construction of recovery sequences, as we need to approximate the limiting averaged displacements $\boldsymbol{u}$ and $v$ with regular maps satisfying the boundary conditions above. This is achieved by employing \cite[Proposition A.2]{friesecke.james.mueller1}, which requires the above mentioned regularity of $\Gamma$. For more details, we refer to \cite{bresciani.thesis}.
\end{remark}

The remainder of the subsection is devoted to the proof of Theorem \ref{thm:gamma-conv}. 

\subsubsection{Compactness}

For future reference, we start by collecting some preliminary compactness results which we present in a more self-contained form.
The  compactness of deformations is proved by adapting the techniques in \cite{friesecke.james.mueller2} to our setting. A fundamental tool in these arguments is the celebrated rigidity estimate \cite[Theorem 3.1]{friesecke.james.mueller1}. For convenience, given  $h>0$ and $\boldsymbol{y}\in W^{1,p}(\Omega;\R^3)$, we set
\begin{equation}
\label{eqn:rig}
    \mathcal{R}_h(\boldsymbol{y})\coloneqq \int_\Omega \dist^2(\nabla_h \boldsymbol{y};SO(3))\vee \dist^p(\nabla_h \boldsymbol{y};SO(3))\,\d \boldsymbol{x}.
\end{equation}
We will use the following slight modification of \cite[Theorem 6]{friesecke.james.mueller2} which was given in \cite[Lemma 4.1]{bresciani}.

\begin{lemma}[Approximation by rotations]
\label{lem:ar}
Let $\boldsymbol{y}\in W^{1,p}(\Omega;\R^3)$.
For every $h>0$, set $r_h\coloneqq \mathcal{R}_h(\boldsymbol{y})$ and $\boldsymbol{F}_h \coloneqq \nabla_h \boldsymbol{y}$. Suppose that $r_h/h^2\to 0$, as $h \to 0^+$.
Then, for $h\ll 1$, there exist $\boldsymbol{R}_h \in W^{1,p}(S;SO(3))$ and  $\boldsymbol{Q}_h \in SO(3)$ such that, for $q \in \{2,p\}$, the following estimates hold:
\begin{align*}
\|\boldsymbol{F}_h-\boldsymbol{R}_h\|_{L^q(\Omega;\rtt)}&\leq C r_h^{1/q}, \hspace{11mm} \|\nabla'\boldsymbol{R}_h\|_{L^q(S;\R^{3 \times 3\times 3})}\leq C h^{-1}r_h^{1/q},\\
\|\boldsymbol{R}_h-\boldsymbol{Q}_h\|_{L^q(S;\rtt)}&\leq C h^{-1}r_h^{1/q}, \quad \|\boldsymbol{F}_h-\boldsymbol{Q}_h\|_{L^q(\Omega;\rtt)}\leq C h^{-1}r_h^{1/q}.
\end{align*}
\end{lemma}

The next proposition provides a simple reformulation of the compactness results in \cite{friesecke.james.mueller2}. Henceforth, $\boldsymbol{\pi}_0$ denotes projection map defined  by $\boldsymbol{\pi}_0(\boldsymbol{x})\coloneqq((\boldsymbol{x}')^\top,0)^\top$ for every $\boldsymbol{x}\in \R^3$.

\begin{proposition}[Compactness of deformations]
\label{prop:cd}
Let $(\widehat{\boldsymbol{y}}_h)\subset W^{1,2}(\Omega;\rt)$ and let $(e_h)\subset \R$ with $e_h>0$ be such that $e_h/h^2\to0$, as $h\to 0^+$. Set $r_h\coloneqq\mathcal{R}_h(\widehat{\boldsymbol{y}}_h)$ and suppose that $r_h\leq Ce_h$ for every $h>0$. Also, set $\widehat{\boldsymbol{F}}_h\coloneqq \nabla_h\widehat{\boldsymbol{y}}_h$ and suppose that there exists $(\widehat{\boldsymbol{R}}_h)\subset W^{1,2}(S;SO(3))$ such that, for every $h>0$, the following estimates hold:
	\begin{align*}
	\|\widehat{\boldsymbol{F}}_h-\widehat{\boldsymbol{R}}_h\|_{L^2(\Omega;\rtt)}&\leq C \sqrt{r_h},\qquad \hspace{6mm} \|\nabla'\widehat{\boldsymbol{R}}_h\|_{L^2(S;\R^{3\times 3 \times 3})}\leq C h^{-1}\sqrt{r_h}\\
	\|\widehat{\boldsymbol{R}}_h-\boldsymbol{I}\|_{L^2(S;\rtt)}&\leq C h^{-1}\sqrt{r_h},\qquad \|\widehat{\boldsymbol{F}}_h-\boldsymbol{I}\|_{L^2(S;\rtt)}\leq C h^{-1}\sqrt{r_h}.
	\end{align*}
Also, for every $h>0$, assume the following:
\begin{equation}
\label{eqn:cd-assumption}
    \text{either $\widehat{\boldsymbol{y}}_h-\boldsymbol{\pi}_h$ has null average over $\Omega$ or   $\widehat{\boldsymbol{y}}_h=\boldsymbol{\pi}_h$ on $\partial S \times I$.}
\end{equation}
Define $\widehat{\boldsymbol{u}}_h\colon S \to \R^2$, $\widehat{v}_h \colon S \to \R$ and $\widehat{\boldsymbol{w}}_h\colon S \to \R^3$ by setting
\begin{align*}
    \widehat{\boldsymbol{u}}_h(\boldsymbol{x}')&\coloneqq \frac{h^2}{e_h}\wedge \frac{1}{\sqrt{e_h}}\int_I \big (\widehat{\boldsymbol{y}}_h'(\boldsymbol{x}',x_3)-\boldsymbol{x}'\big )\,\d x_3,\\
    \widehat{v}_h(\boldsymbol{x}')&\coloneqq \frac{h}{\sqrt{e_h}}\int_I \widehat{y}_h^{\,3}(\boldsymbol{x}',x_3)\,\d x_3,\\
     \widehat{\boldsymbol{w}}_h(\boldsymbol{x}')&\coloneqq \frac{1}{\sqrt{e_h}}\int_I x_3\big (\widehat{\boldsymbol{y}}_h(\boldsymbol{x}',x_3)-\boldsymbol{\pi}_h(\boldsymbol{x}',x_3)\big )\,\d x_3,
\end{align*}
for every $\boldsymbol{x}'\in S$. Then, the following estimates hold:
	\begin{align}
			\label{eqn:u-bdd}
			\|\widehat{\boldsymbol{u}}_h\|_{W^{1,2}(S;\R^2)}&\leq C \left(  \sqrt{\frac{r_h}{e_h}}+\frac{r_h}{e_h} \wedge \frac{r_h}{h^2\sqrt{e_h}}  \right),\\
		\label{eqn:v-bdd} 
		\|\widehat{v}_h\|_{L^2(S;\rtt)}&\leq C \sqrt{\frac{r_h}{e_h}},\\
			\label{eqn:w-bdd}
			\|\widehat{\boldsymbol{w}}_h\|_{L^2(S;\rtt)}&\leq C    \sqrt{\frac{r_h}{e_h}}.
	\end{align}
Moreover,  there exist $\widehat{\boldsymbol{u}}\in W^{1,2}(S;\R^2)$ and $\widehat{v}\in W^{2,2}(S)$ such that, up to subsequences, the following convergences hold, as $h \to 0^+$:
\begin{align}
    \label{eqn:cd-horizontal}
    \text{$\widehat{\boldsymbol{u}}_h$}&\text{$\wk\widehat{\boldsymbol{u}}$ in $W^{1,2}(S;\R^2)$,}\\
    \label{eqn:cd-vertical}
    \text{$\widehat{v}_h$}&\text{$\to \widehat{v}$ in $W^{1,2}(S)$,}\\
    \label{eqn:cd-momentum}
    \text{$\widehat{\boldsymbol{w}}_h$}&\text{$\wk -\frac{1}{12}\left ( \begin{array}{cc}
         \nabla' \widehat{v}  \\
         0
    \end{array}\right)$ in $W^{1,2}(S;\R^3)$.}
\end{align}
\end{proposition}
\begin{proof}
{\MMM
The convergences in \eqref{eqn:cd-horizontal}--\eqref{eqn:cd-momentum}  have been proved in \cite[Lemma 1]{friesecke.james.mueller2} and \cite[Corollary 1]{friesecke.james.mueller2}. Also the estimates in \eqref{eqn:u-bdd}--\eqref{eqn:w-bdd} are implicitly established in these results.
}
Note that assumption \eqref{eqn:cd-assumption} is needed in order to apply Poincar\'{e} and Korn inequalities. Indeed, if $\widehat{\boldsymbol{y}}_h-\boldsymbol{\pi}_h$ has null average over $\Omega$, then the same property holds for $\widehat{\boldsymbol{u}}_h$ and $\widehat{v}_h$, while, if $\widehat{\boldsymbol{y}}_h=\boldsymbol{\pi}_h$ on $\partial S$, then $\widehat{\boldsymbol{u}}_h$ and $\widehat{v}_h$ satisfy homogeneous Dirichlet boundary conditions.
\end{proof}

{\MMM 
In the next result, we show how the clamped boundary conditions can be exploited to establish the convergence of the sequence of constant rotations provided by Lemma \ref{lem:ar} towards the identity matrix. The arguments are adapted from \cite{lecumberry.mueller}.
}

\begin{lemma}[Clamped boundary conditions]
	\label{lem:clamped}
	Let $(\boldsymbol{y}_h)\subset W^{1,2}(\Omega;\rt)$ and let $(e_h)\subset \R$ with $e_h>0$ be such that $e_h/h^2\to 0$, as $h \to 0^+$. 
	Set $r_h\coloneqq \mathcal{R}_h(\boldsymbol{y}_h)$ and suppose that $r_h\leq C e_h$ for every $h>0$. Also, set $\boldsymbol{F}_h\coloneqq \nabla_h\boldsymbol{y}_h$ and suppose that there exist $(\boldsymbol{R}_h)\subset W^{1,2}(S;SO(3))$ and $(\boldsymbol{Q}_h)\subset SO(3)$ such that, for every $h>0$, the following estimates hold:
	\begin{align}
	\label{eqn:clamped-estimates}
	\|\boldsymbol{F}_h-\boldsymbol{R}_h\|_{L^2(\Omega;\rtt)}&\leq C \sqrt{r_h},\qquad \hspace{9mm}\|\nabla'\boldsymbol{R}_h\|_{L^2(S;\R^{3\times 3 \times 3})}\leq C h^{-1}\sqrt{r_h},\\
	\label{eqn:clamped-estimates-bis}
	\|\boldsymbol{R}_h-\boldsymbol{Q}_h\|_{L^2(S;\rtt)}&\leq C h^{-1}\sqrt{r_h},\qquad \|\boldsymbol{F}_h-\boldsymbol{Q}_h\|_{L^2(S;\rtt)}\leq C h^{-1}\sqrt{r_h}.    
	\end{align}
	Additionally, suppose that $\boldsymbol{y}_h=\boldsymbol{\pi}_h$ on $\partial S \times I$ for every $h>0$. 
	Then, for $h \ll 1$, there holds
	\begin{equation}
	\label{eqn:clamped-rotation}
	|\boldsymbol{Q}_h-\boldsymbol{I}|\leq C  \frac{\sqrt{r_h}}{h}.
	\end{equation}
	Moreover, denoting by $\boldsymbol{c}_h\in \rt$ the average of $\boldsymbol{Q}_h^\top \boldsymbol{y}_h-\boldsymbol{\pi}_h$ over $\Omega$, there holds
	\begin{equation}
	\label{eqn:clamped-translation}
	|\boldsymbol{c}_h|\leq C  \frac{\sqrt{r_h}}{h}.
	\end{equation}
\end{lemma}
\begin{proof}
	We first prove \eqref{eqn:clamped-rotation}. 
	Define  $\widetilde{\boldsymbol{y}}_h\coloneqq \boldsymbol{Q}_h^\top \boldsymbol{y}_h-\boldsymbol{c}_h$ with $\boldsymbol{c}_h \in \rt$ chosen so that $\widetilde{\boldsymbol{y}}_h-\boldsymbol{\pi}_h$ has null average over $\Omega$. Set $\widetilde{\boldsymbol{F}}_h\coloneqq \nabla_h \widetilde{\boldsymbol{y}}_h$ and $\widetilde{\boldsymbol{R}}_h\coloneqq \boldsymbol{Q}_h^\top \boldsymbol{R}_h$. Assumptions \eqref{eqn:clamped-estimates}--\eqref{eqn:clamped-estimates-bis}  immediately yield
	\begin{align*}
	\|\widetilde{\boldsymbol{F}}_h-\widetilde{\boldsymbol{R}}_h\|_{L^2(\Omega;\rtt)}&\leq C \sqrt{r_h},\qquad \hspace*{6mm} \|\nabla'\widetilde{\boldsymbol{R}}_h\|_{L^2(S;\R^{3\times 3 \times 3})}\leq C h^{-1}\sqrt{r_h},\\
	\|\widetilde{\boldsymbol{R}}_h-\boldsymbol{I}\|_{L^2(S;\rtt)}&\leq C h^{-1}\sqrt{r_h},\qquad \|\widetilde{\boldsymbol{F}}_h-\boldsymbol{I}\|_{L^2(S;\rtt)}\leq C h^{-1}\sqrt{r_h}.
	\end{align*}
	Define $\widetilde{\boldsymbol{u}}_h\colon S \to \R^2$, $\widetilde{v}_h:S \to \R$ and $\widetilde{\boldsymbol{w}}_h\colon S \to \R^3$  by setting
	\begin{align*}
	\widetilde{\boldsymbol{u}}_h(\boldsymbol{x}')&\coloneqq  \left(\frac{h^2}{e_h}\wedge \frac{1}{\sqrt{e_h}}\right)\int_I \left (\widetilde{\boldsymbol{y}}_h'(\boldsymbol{x}',x_3)-\boldsymbol{x}'\right )\,\d x_3,\\
	\widetilde{v}_h(\boldsymbol{x}')&\coloneqq \frac{h}{\sqrt{e_h}}\int_I \widetilde{y}_h^{\,3}(\boldsymbol{x}',x_3)\,\d x_3,\\
	\widetilde{\boldsymbol{w}}_h(\boldsymbol{x}')&\coloneqq \frac{1}{\sqrt{e_h}}\int_I x_3\left (\widetilde{\boldsymbol{y}}_h-\boldsymbol{\pi}_h\right )\,\d x_3.
	\end{align*}
	Applying Proposition \ref{prop:cd} to $\widehat{\boldsymbol{y}}_h=\widetilde{\boldsymbol{y}}_h$ and then the trace inequality, we obtain the estimates 
	\begin{align}
		\label{eqn:trace-u}
		\|\widetilde{ \boldsymbol{u}}_h\|_{L^2(\partial S;\R^2)}&\leq C \left(\sqrt{\frac{r_h}{e_h}}+ \frac{r_h}{e_h} \wedge \frac{r_h}{h^2\sqrt{e_h}} \right),\\
		\label{eqn:trace-v}
		\|\widetilde{ v}_h\|_{L^2(\partial S)}&\leq C  \sqrt{\frac{r_h}{e_h}},\\
		\label{eqn:trace-w}
		\|\widetilde{ \boldsymbol{w}}_h\|_{L^2(\partial S;\rt)}&\leq C   \sqrt{\frac{r_h}{e_h}}.
	\end{align}
	Analogously, define ${\boldsymbol{u}}_h\colon S \to \R^2$, ${v}_h:S \to \R$ and ${\boldsymbol{w}}_h\colon S \to \rt$  by setting
	\begin{align*}
	{\boldsymbol{u}}_h(\boldsymbol{x}')&\coloneqq  \left(\frac{h^2}{e_h}\wedge \frac{1}{\sqrt{e_h}}\right)\int_I \left ({\boldsymbol{y}}_h'(\boldsymbol{x}',x_3)-\boldsymbol{x}'\right)\,\d x_3,\\
	{v}_h(\boldsymbol{x}')&\coloneqq \frac{h}{\sqrt{e_h}}\int_I {y}_h^{\,3}(\boldsymbol{x}',x_3)\,\d x_3,\\
	{\boldsymbol{w}}_h(\boldsymbol{x}')&\coloneqq \frac{1}{\sqrt{e_h}}\int_I x_3\left ({\boldsymbol{y}}_h-\boldsymbol{\pi}_h\right )\,\d x_3.
	\end{align*}
	In view of the clamped boundary condition satisfied by $\boldsymbol{y}_h$, there hold
	\begin{equation}
	    \label{eqn:trace-clamped}
	    \text{$\boldsymbol{u}_h=\boldsymbol{0}'$ on $\partial S$, \qquad $v_h=0$ on $\partial S$, \qquad $\boldsymbol{w}_h=\boldsymbol{0}$ on $\partial S$.}
	\end{equation}
	For simplicity, set $\boldsymbol{d}_h\coloneqq \boldsymbol{Q}_h \boldsymbol{c}_h$.
	Then
	\begin{align}
	\label{eqn:clamped1}
	\left ( \begin{array}{cc}
	h^{-2}e_h\vee \sqrt{e_h}\,\boldsymbol{u}_h \\
	h^{-1}\sqrt{e_h} \,v_h
	\end{array}\right)&=(\boldsymbol{Q}_h-\boldsymbol{I})\, \boldsymbol{\pi}_0+\boldsymbol{Q}_h\left ( \begin{array}{cc}
	h^{-2}e_h\vee \sqrt{e_h}\,\widetilde{\boldsymbol{u}}_h \\
	h^{-1}\sqrt{e_h} \,\widetilde{v}_h
	\end{array}\right)+\boldsymbol{d}_h,\\
	\label{eqn:clamped2}
	\sqrt{e_h}\boldsymbol{w}_h&=\frac{h}{12}(\boldsymbol{Q}_h-\boldsymbol{I})\boldsymbol{e}_3+\sqrt{e_h}\boldsymbol{Q}_h\,\widetilde{\boldsymbol{w}}_h.
	\end{align}
	From \eqref{eqn:clamped2}, given \eqref{eqn:trace-w}--\eqref{eqn:trace-clamped}, we obtain
	\begin{equation}
	\label{eqn:clamped-rot}
		|(\boldsymbol{Q}_h-\boldsymbol{I})\boldsymbol{e}_3|\leq C \frac{\sqrt{e_h}}{h} \|\widetilde{ \boldsymbol{w}}_h\|_{L^2(\partial S;\rt)}
		\leq C \frac{\sqrt{r_h}}{h}.
	\end{equation}
	In turn, we also have
	\begin{equation}
	\label{eqn:clamped-rot-bis}
	|(\boldsymbol{Q}_h^\top-\boldsymbol{I})\boldsymbol{e}_3|\leq C \frac{\sqrt{r_h}}{h}.
	\end{equation}
	Looking at the first two components of \eqref{eqn:clamped1}, we estimate
	\begin{equation}
	\label{eqn:clamped-norm}
	\begin{split}
		\|(\boldsymbol{Q}_h''-\boldsymbol{I}'')\boldsymbol{x}'+\boldsymbol{d}_h'\|_{L^2(\partial S;\R^2)} &\leq C  \left( \frac{e_h}{h^2}\vee  \sqrt{e_h} \right)  	\|\widetilde{ \boldsymbol{u}}_h\|_{L^2(\partial S;\R^2)}
		+  C \frac{\sqrt{e_h}}{h} \|\widetilde{ v}_h\|_{L^2(\partial S)}\\
		&\leq C \left (\frac{\sqrt{r_h e_h}}{h^2}+\sqrt{r_h}+\frac{r_h}{h^2} \right )\leq C \frac{\sqrt{r_h}}{h}.
	\end{split}
	\end{equation}
	{\MMM Here, in the last inequality, we exploited the two conditions $r_h\leq C e_h$ and $e_h/h^2\to 0$, as $h\to 0^+$.}
	
	Up to translations, we can assume that $$\int_{\partial S} \boldsymbol{x}'\,\d \boldsymbol{l}=\boldsymbol{0}'.$$ In this case, \eqref{eqn:clamped-norm} yields
	\begin{equation}
	\label{eqn:clamped-est}
		|\boldsymbol{Q}_h''-\boldsymbol{I}''|+|\boldsymbol{d}_h'|\leq C  \frac{\sqrt{r_h}}{h}.
	\end{equation}
	Combining \eqref{eqn:clamped-rot}--\eqref{eqn:clamped-rot-bis} with \eqref{eqn:clamped-est}, we establish \eqref{eqn:clamped-rotation}.
	
	Similarly, looking at the third component of \eqref{eqn:clamped1} and exploiting \eqref{eqn:trace-u}--\eqref{eqn:trace-v}, \eqref{eqn:trace-clamped} and \eqref{eqn:clamped-rot-bis}, we  estimate
	\begin{equation}
	\label{eqn:clamped-trasl}
		|d_h^{\,3}|\leq C \left( \frac{e_h}{h^2}\vee \sqrt{e_h} \right) \|\widetilde{ \boldsymbol{u}}_h\|_{L^2(\partial S;\R^2)}+C \frac{\sqrt{e_h}}{h}\|\widetilde{v}_h\|_{L^2(\partial S)}+C |(\boldsymbol{Q}_h^\top-\boldsymbol{I})\boldsymbol{e}_3|\leq C \frac{\sqrt{r_h}}{h}.
	\end{equation}
	Thus, \eqref{eqn:clamped-translation} follows from   \eqref{eqn:clamped-est}--\eqref{eqn:clamped-trasl}.
\end{proof}

{\MMM To identify the limiting strain,} we employ the following result given by \cite[Lemma 2]{friesecke.james.mueller2}.

\begin{lemma}[Identification of the limiting strain]
\label{lem:limiting-strain}
Let $(\widehat{\boldsymbol{y}}_h)\subset W^{1,2}(\Omega;\rt)$ and let $(e_h) \subset \R$ with $e_h>0$ be such that  $e_h/h^2\to 0$, as $h \to 0^+$. Set $r_h\coloneqq \mathcal{R}_h(\widehat{\boldsymbol{y}}_h)$ and suppose that $r_h\leq C e_h$ for every $h>0$. Also, set  $\widehat{\boldsymbol{F}}_h\coloneqq \nabla_h \widehat{\boldsymbol{y}}_h$ and suppose that  there exists $(\widehat{\boldsymbol{R}}_h)\subset W^{1,2}(S;SO(3))$  such that, for every $h>0$, the following estimate holds:
	\begin{equation*}
	\|\widehat{\boldsymbol{F}}_h-\widehat{\boldsymbol{R}}_h\|_{L^2(\Omega;\rtt)}\leq C \sqrt{r_h}.
	\end{equation*}
	Define $\widehat{\boldsymbol{u}}_h\colon S \to \R^2$ and $\widehat{v}_h \colon S \to \R$ by setting
	\begin{align*}
	\widehat{\boldsymbol{u}}_h(\boldsymbol{x}')&\coloneqq \frac{h^2}{e_h}\wedge \frac{1}{\sqrt{e_h}}\int_I \left (\widehat{\boldsymbol{y}}_h'(\boldsymbol{x}',x_3)-\boldsymbol{x}'\right )\,\d x_3,\\
	\widehat{v}_h(\boldsymbol{x}')&\coloneqq \frac{h}{\sqrt{e_h}}\int_I \widehat{y}_h^{\,3}(\boldsymbol{x}',x_3)\,\d x_3,
	\end{align*}
	for every $\boldsymbol{x}'\in S$, and assume that there exist $\widehat{\boldsymbol{u}}\in W^{1,2}(S;\R^2)$ and $\widehat{v}\in W^{2,2}(S)$ such that the following convergences hold:
	\begin{equation*}
	\text{$\widehat{\boldsymbol{u}}_h\wk\widehat{\boldsymbol{u}}$ in $W^{1,2}(S;\R^2)$,}\qquad 
	\text{$\widehat{v}_h\to \widehat{v}$ in $W^{1,2}(S)$.}
	\end{equation*}
	Then, there exists $\widehat{\boldsymbol{G}}\in L^2(\Omega;\rtt)$ such that
	\begin{equation*}
	\text{$\widehat{\boldsymbol{G}}_h\coloneqq \frac{1}{\sqrt{e_h}}(\widehat{\boldsymbol{R}}_h^\top \widehat{\boldsymbol{F}}_h-\boldsymbol{I})\wk \widehat{\boldsymbol{G}}$ in $L^2(\Omega;\rtt)$.}
	\end{equation*}
	Furthermore, there exists $\widehat{\boldsymbol{\Sigma}}\in L^2(\Omega;\rtwtw)$ such that, for almost every $\boldsymbol{x}\in \Omega$, there holds
	\begin{equation*}
	\label{eqn:cd-G-structure}
	\text{$\widehat{\boldsymbol{G}}''(\boldsymbol{x}',x_3)=\widehat{\boldsymbol{\Sigma}}(\boldsymbol{x}')-(\nabla')^2\widehat{v}(\boldsymbol{x}')x_3$.}
	\end{equation*}
	Eventually, if $e_h/h^4\to 0$, as $h \to 0^+$, then $\sym\, \widehat{\boldsymbol{\Sigma}}=\sym \nabla'\widehat{\boldsymbol{u}}$.
\end{lemma}

The compactness of magnetizations is established by refining the techniques introduced in \cite[Proposition 4.3]{bresciani}. We will use the following notation. Given $\eta>0$, we set $S^\eta \coloneqq \{\boldsymbol{x}'\in S:\:\dist(\boldsymbol{x}';\partial S)<\eta\}$ and $S^{-\eta}\coloneqq \{\boldsymbol{x}'\in \R^3:\:\dist(x;S)<\eta\}$. Moreover, for $l>0$, we set $\Omega^\eta_l\coloneqq S^\eta\times lI$ and $\Omega^{-\eta}_l\coloneqq S^{-\eta}\times lI$. Recall the notation in \eqref{eqn:notation-mu}--\eqref{eqn:lagrangian-magnetization}.

\begin{proposition}[Compactness of magnetizations]
\label{prop:cm}
	Let $(\widehat{\boldsymbol{q}}_h)$ with $\widehat{\boldsymbol{q}}_h=(\widehat{\boldsymbol{y}}_h,\widehat{\boldsymbol{m}}_h)\in\mathcal{Q}_h$ be such that
	\begin{equation}
	\label{eqn:exc-bdd}
	\sup_{h>0} \left \{ E^{\mathrm{el}}_h(\widehat{\boldsymbol{q}}_h)+ E^{\mathrm{exc}}_h(\widehat{\boldsymbol{q}}_h) \right\}\leq C.
	\end{equation}
	Suppose that there exists $\alpha>1$ such that, for every $h>0$,  there holds
	\begin{align}
		\label{eqn:cm-estimate}
		\|\widehat{\boldsymbol{y}}_h-\boldsymbol{\pi}_h\|_{C^0(\closure{\Omega};\R^3)}&\leq Ch^\alpha,\\
		\label{eqn:cm-gradient-estimate}
		\|\widehat{\boldsymbol{F}}_h-\boldsymbol{I}\|_{L^p(\Omega;\R^3)}&\leq Ch^{\beta/p-1}.
	\end{align}
	Then, there exist $\widehat{\boldsymbol{\zeta}}\in W^{1,2}(S;\S^2)$ and $\widehat{\boldsymbol{\chi}}\in L^2(\R^3;\R^3)$ such that, up to subsequences, the following convergences hold, as $h \to 0^+$:
	\begin{align}
	\label{eqn:cm-eta}
	\text{$\widehat{\boldsymbol{\mu}}_h\coloneqq \mathcal{M}_h(\widehat{\boldsymbol{q}}_h)$}&\text{$\to \chi_\Omega \widehat{\boldsymbol{\zeta}}$ in $L^2(\R^3;\rtt)$,}\\
	\label{eqn:cm-H}
	\text{$\widehat{\boldsymbol{N}}_h\coloneqq \mathcal{N}_h(\widehat{\boldsymbol{q}}_h)$}&\text{$\wk \chi_\Omega (\nabla'\boldsymbol{\widehat{\zeta}}\vert \widehat{\boldsymbol{\chi}})$ in $L^2(\R^3;\rtt)$, }\\
	\label{eqn:cm-composition}
	\text{$\widehat{\boldsymbol{m}}_h\circ\widehat{\boldsymbol{y}}_h$}&\text{$\to \widehat{\boldsymbol{\zeta}}$ in $L^1(\Omega;\rt)$,}\\
	\label{eqn:cm-lagrangian}
	\text{$\widehat{\boldsymbol{z}}_h\coloneqq \mathcal{Z}_h(\widehat{\boldsymbol{q}}_h)$}&\text{$\to \widehat{\boldsymbol{\zeta}}$ in $L^1(\Omega;\rtt)$.}
	\end{align}
\end{proposition}
\begin{proof}
For convenience of the reader, the proof is subdivided into three steps.

\textbf{Step 1 (Approximation of the deformed configuration).} The estimate \eqref{eqn:cm-estimate} entails the following two statements:
\begin{equation}
    \label{eqn:cm-subcylinder}
    \forall\, \eta>0,\:\forall\, 0<\tht<1,\:\:\exists\, \bar{h}(\eta,\tht)>0:\:\:\forall\, 0<h\leq \bar{h}(\eta,\tht), \:\:\Omega^\eta_{\tht h} \subset \Omega^{\widehat{\boldsymbol{y}}_h},
\end{equation}
\begin{equation}
    \label{eqn:cm-supercylinder}
    \hspace*{7mm}\forall\, \eta>0,\:\forall\, \ell>1,\:\:\exists\, \underline{h}(\eta,\ell)>0:\:\:\forall\, 0<h\leq \underline{h}(\eta,\ell), \:\: \Omega^{\widehat{\boldsymbol{y}}_h}\subset \Omega^{-\eta}_{\ell h}.
\end{equation}
To see \eqref{eqn:cm-subcylinder}, fix $\eta>0$ and $0<\vartheta<1$. Let $\boldsymbol{\xi}\in \Omega^{\eta}_{\vartheta h}$. As $\alpha>1$, there exists $\bar{h}(\eta,\vartheta)>0$ such that for every $h\leq \bar{h}(\eta,\vartheta)$ there holds
\begin{equation*}
    \dist(\boldsymbol{\xi};\partial \Omega_h)\geq \eta \wedge (1-\tht)h/2 > C h^\alpha
\end{equation*}
so that, by \eqref{eqn:cm-estimate}, we obtain
\begin{equation*}
    \|\widehat{\boldsymbol{y}}_h-\boldsymbol{\pi}_h\|_{C^0(\closure{\Omega};\R^3)}< \dist(\boldsymbol{\xi};\partial \Omega_h)=\dist(\boldsymbol{\xi};\boldsymbol{\pi}_h(\partial \Omega)).
\end{equation*}
By the stability property of the topological degree \cite[Theorem 2.3, Claim (1)]{fonseca.gangbo}, this entails $\boldsymbol{\xi}\notin \widehat{\boldsymbol{y}}_h(\partial \Omega)$ and $\deg(\widehat{\boldsymbol{y}}_h,\Omega,\boldsymbol{\xi})=\deg(\boldsymbol{\pi}_h,\Omega,\boldsymbol{\xi})=1$. Then, by the solvability property of the topological degree \cite[Theorem 2.1]{fonseca.gangbo}, we deduce $\boldsymbol{\xi}\in \Omega^{\widehat{\boldsymbol{y}}_h}$. As $\boldsymbol{\xi}\in \Omega^\eta_{\vartheta h}$ was arbitrary, this proves \eqref{eqn:cm-subcylinder}. 

To see \eqref{eqn:cm-supercylinder}, fix $\eta>0$ and $\ell>1$. 
Again, as $\alpha>1$, there exists $\underline{h}(\eta,\ell)>0$ such that for every $h \leq \underline{h}(\eta,\ell)$ there holds
\begin{equation*}
    \dist(\closure{\Omega}_h; \partial \Omega^{-\eta}_{\ell h})\geq \eta \wedge (\ell-1)h/2 > C h^\alpha.
\end{equation*}
Thus, by \eqref{eqn:cm-estimate}, we have
\begin{equation*}
    \Omega^{\widehat{\boldsymbol{y}}_h}\subset \Omega_h + B(\boldsymbol{0},Ch^\alpha)\subset \Omega^{-\eta}_{\ell h}.
\end{equation*}

{\MMM
	Similarly to \eqref{eqn:cm-subcylinder}, we also have the following:
	\begin{equation}
		\label{eqn:cm-subcylinder-variant}
		\begin{split}
			&\forall\,0<\eta_0<\eta,\:\forall\,0<\vartheta<\vartheta_0<1, \: \exists\,\widetilde{h}(\eta,\eta_0,\vartheta,\vartheta_0)>0:\\
			&\hspace*{8mm}\forall\,0<h<\widetilde{h}(\eta,\eta_0,\vartheta,\vartheta_0),\quad \Omega^\eta_{\vartheta h} \subset \widehat{\boldsymbol{y}}_h(\Omega^{\eta_0}_{\vartheta_0}).
		\end{split}
	\end{equation}
}

\textbf{Step 2 {\MMM (Identification of the limiting magnetization)}.}
{\MMM
Employing the notation in \eqref{eqn:prestrain-inverse} and \eqref{eqn:lagrangian-magnetization}, we set $\boldsymbol{F}_h\coloneqq \nabla_h \widehat{\boldsymbol{y}}_h$ and $\widehat{\boldsymbol{\Xi}}_h\coloneqq \boldsymbol{L}_h(\mathcal{Z}_h(\widehat{\boldsymbol{q}}_h))^{-1}\boldsymbol{F}_h$. Thanks to \eqref{eqn:prestrain-determinant} and \eqref{eqn:coercivity-Phi-det}, we have
	\begin{equation}
	\label{eqn:a2}
	\begin{split}
	\int_\Omega  \frac{1}{(\det \nabla_h\widehat{\boldsymbol{y}}_h)^a}\,\d\boldsymbol{x}&=(1+c_h)^{-a} \int_\Omega  \frac{1}{(\det \widehat{\boldsymbol{\Xi}}_h)^a}\,\d\boldsymbol{x}\\
	&\leq (1+c_h)^{-a} \left \{\int_\Omega \Phi(\widehat{\boldsymbol{\Xi}}_h)\,\d\boldsymbol{x} +C\right\}\\
	&\leq (1+c_h)^{-a} \left \{h^\beta E_h^{\rm el}(\widehat{\boldsymbol{q}}_h) +C\right\}\leq C,
	\end{split}
	\end{equation} 
	where in the last line we used \eqref{eqn:ch-limit},  \eqref{eqn:coercivity-Phi-det} and \eqref{eqn:exc-bdd}. 
	As $a>1$, we infer that $(1/\det D_h\widehat{\boldsymbol{y}}_h)$ is equi-integrable by the de la Vall\'{e}e-Poussin Criterion \cite[Theorem 2.29]{fonseca.leoni}. 

Let $\eta>0$ and $0<\tht<1$. By \eqref{eqn:cm-subcylinder}, for $h\leq \bar{h}(\eta,\tht)$, the composition  $\widehat{\boldsymbol{\zeta}}_h\coloneqq \widehat{\boldsymbol{m}}_h\circ \boldsymbol{\pi}_h\restr{\Omega^\eta_\vartheta}$ is well defined  as a map in $W^{1,2}(\Omega^\eta_\vartheta;\rt)$. Then, employing the change-of-variable formula,  we estimate
	\begin{equation}
	\label{eqn:a1}
	\begin{split}
	\int_{\Omega^\eta_\vartheta} |\widehat{\boldsymbol{\zeta}}_h|^2 \,\d\boldsymbol{x}&=\int_{\Omega^\eta_\vartheta}|\widehat{\boldsymbol{m}}_h\circ \boldsymbol{\pi}_h|^2\,\d\boldsymbol{x}=\frac{1}{h}\int_{\Omega^\eta_{\vartheta h}}|\widehat{\boldsymbol{m}}_h|^2\,\d\boldsymbol{\xi}\\
	&\leq \frac{1}{h}\int_{\Omega^{\widehat{{\boldsymbol{y}}}_h}} |\widehat{\boldsymbol{m}}_h|^2\,\d\boldsymbol{\xi}\leq \frac{1}{h}\int_\Omega |\widehat{\boldsymbol{m}}\circ \widehat{\boldsymbol{y}}_h|^2\det \nabla \widehat{\boldsymbol{y}}_h\,\d\boldsymbol{x}\\
	&=\int_\Omega |\widehat{\boldsymbol{m}}\circ \widehat{\boldsymbol{y}}_h|^2\det \nabla_h \widehat{\boldsymbol{y}}_h\,\d\boldsymbol{x}=\int_\Omega  \frac{1}{\det \nabla_h\widehat{\boldsymbol{y}}_h}\,\d\boldsymbol{x},
	\end{split}
	\end{equation}
	where in the last line we exploited the magnetic saturation constraint
	\begin{equation}
	    \label{eqn:saturation-sequence}
	    \text{$|\widehat{\boldsymbol{m}}_h\circ \widehat{\boldsymbol{y}}_h|\det \nabla_h\widehat{\boldsymbol{y}}_h=1$ a.e. in $\Omega$.}
	\end{equation}
    Thus, by \eqref{eqn:a2}, the sequence $(\widehat{\boldsymbol{\zeta}}_h)$ is bounded in $L^2(\Omega^\eta_\vartheta;\rt)$.
	
	From \eqref{eqn:exc-bdd} and \eqref{eqn:cm-subcylinder}, employing again
    the change-of-variable formula, we obtain
	\begin{equation}
	\label{eqn:exc-bdd2}
	\begin{split}
	\int_{\Omega^\eta_\vartheta} \left| \nabla_h \widehat{\boldsymbol{\zeta}}_h \right|^2\,\d \boldsymbol{x}&=\int_{\Omega^\eta_\vartheta} |\nabla\widehat{\boldsymbol{m}}_h\circ \boldsymbol{\pi}_h|^2\,\d\boldsymbol{x}=\frac{1}{h}\int_{\Omega^\eta_{\vartheta h}}|\nabla \widehat{\boldsymbol{m}}_h|^2\,\d \boldsymbol{\xi}\\
	 &\leq \frac{1}{h}\int_{\Omega^{\widehat{\boldsymbol{y}}_h}}|D \widehat{\boldsymbol{m}}_h|^2\,\d \boldsymbol{\xi}=E^{\text{exc}}_h(\widehat{\boldsymbol{q}}_h)\leq C.
	\end{split}
	\end{equation}
	
	Therefore $(\widehat{\boldsymbol{\zeta}}_h)$ is bounded in $W^{1,2}(\Omega^\eta_\vartheta;\rt)$, so that
	there exists $\widehat{\boldsymbol{\zeta}}\in W^{1,2}(\Omega^\eta_\vartheta;\rt)$ such that, up to subsequences,  $\widehat{\boldsymbol{\zeta}}_h\wk\widehat{\boldsymbol{\zeta}}$ in $W^{1,2}(\Omega^\eta_\vartheta;\rt)$. Also, by \eqref{eqn:exc-bdd2}, there exists a map $\widehat{\boldsymbol{\chi}}\in L^2(\Omega^\eta_\vartheta;\rt)$ such that,
	up to subsequences,  $h^{-1}\partial_3\widehat{\boldsymbol{\zeta}}_h \wk \widehat{\boldsymbol{\chi}}$ in $L^2(\Omega^\eta_\vartheta;\rt)$, as $h \to 0^+$. This yields $\partial_3\widehat{\boldsymbol{\zeta}}=\boldsymbol{0}$, so that $\widehat{\boldsymbol{\zeta}}\in W^{1,2}(S^\eta;\rt)$.
	
	In principle, both the subsequences and the weak limits $\widehat{\boldsymbol{\zeta}}$ and $\widehat{\boldsymbol{\chi}}$ depend on the parameters $\eta$ and $\vartheta$. However, by means of a diagonal argument, we can assume that $\widehat{\boldsymbol{\zeta}}\in W^{1,2}_\loc(S;\rt)$ and $\widehat{\boldsymbol{\chi}}\in L^2_\loc(\Omega;\rt)$, and that, for a not relabeled subsequence, the following holds:
	\begin{equation}
	\label{eqn:cm-diagonal}
		\forall\,\eta>0, \forall\,0<\tht<1, \quad \text{$\widehat{\boldsymbol{\zeta}}_h\wk \widehat{\boldsymbol{\zeta}}$ in $W^{1,2}(\Omega^\eta_\vartheta;\rt)$ and a.e. in $\Omega^\eta_\vartheta$,}\quad
		\text{$h^{-1}\partial_3 \widehat{\boldsymbol{\zeta}}_h\wk\widehat{\boldsymbol{\chi}}$ in $L^2(\Omega^\eta_\vartheta;\rt)$.}
	\end{equation}
	Here, we exploited the Sobolev embedding, which is applicable since $\Omega^\eta_\vartheta$ is a Lipschitz domain, at least for $\eta\ll 1$ and $1-\vartheta \ll 1$.
	Note that the sequences in  \eqref{eqn:cm-diagonal} are defined only for $h \ll 1$ depending on $\eta$ and $\vartheta$.
	In view of \eqref{eqn:a2}--\eqref{eqn:a1} and \eqref{eqn:cm-diagonal}, by lower semicontinuity, for every $\eta>0$ and $0<\vartheta<1$, there holds 
	\begin{equation}
		\label{eqn:cm-b1}
		\vartheta\int_{S^{\eta}} |\widehat{\boldsymbol{\zeta}}|^2\,\d\boldsymbol{x}'=  \int_{\Omega^\eta_\vartheta} |\widehat{\boldsymbol{\zeta}}|^2\,\d\boldsymbol{x}\leq \liminf_{h \to 0^+} \int_{\Omega^\eta_\vartheta} |\widehat{\boldsymbol{\zeta}}_h|^2\,\d\boldsymbol{x}\leq C.
	\end{equation}
	Similarly, by
	\eqref{eqn:exc-bdd2} and \eqref{eqn:cm-diagonal}, for every $\eta>0$ and $0<\vartheta<1$, there holds 
	\begin{equation}
		\label{eqn:cm-b2}
	\begin{split}
	\vartheta \int_{S^{\eta}} |\nabla'\widehat{\boldsymbol{\zeta}}|^2\,\d \boldsymbol{x}+\int_{\Omega^{\eta}_{\vartheta}} |\widehat{\boldsymbol{\chi}}|^2\,\d \boldsymbol{x}&=\int_{\Omega^\eta_\vartheta} |\nabla\widehat{\boldsymbol{\zeta}}|^2\,\d\boldsymbol{x}+\int_{\Omega^{\eta}_{\vartheta}} |\widehat{\boldsymbol{\chi}}|^2\,\d \boldsymbol{x}\\
	&\leq \liminf_{h \to 0^+} \int_{\Omega^{\eta}_{\vartheta}}\left|\nabla_h \widehat{\boldsymbol{\zeta}}_h\right|^2\,\d\boldsymbol{x} \leq C.
	\end{split}
	\end{equation}
	Letting $\eta\to0^+$ and $\vartheta \to 1^-$ in \eqref{eqn:cm-b1}--\eqref{eqn:cm-b2}, we deduce that $\widehat{\boldsymbol{\zeta}}\in W^{1,2}(S;\rt)$ and $\widehat{\boldsymbol{\chi}}\in L^2(\Omega;\rt)$.
	
		Set $\widehat{\boldsymbol{\mu}}_h\coloneqq \mathcal{M}_h(\widehat{\boldsymbol{q}}_h)$ and $\widehat{\boldsymbol{N}}_h\coloneqq \mathcal{N}_h(\widehat{\boldsymbol{q}}_h)$. Note that these two maps are defined on the whole space for every $h>0$. We claim that $(\widehat{\boldsymbol{\mu}}_h)$ is bounded in $L^2(\rt;\rt)$. To see this, exploiting \eqref{eqn:saturation-sequence} and applying the change-of-variable formula, we compute
	\begin{equation}
	\label{eqn:a3}
		\begin{split}
		\int_{\rt} |\widehat{\boldsymbol{\mu}}_h|^2\,\d\boldsymbol{x}&=\int_{\boldsymbol{\pi}_h^{-1}(\Omega^{\widehat{\boldsymbol{y}}_h})} |\widehat{\boldsymbol{m}}_h\circ \boldsymbol{\pi}_h|^2\,\d\boldsymbol{x}=\frac{1}{h}\int_{\Omega^{\widehat{\boldsymbol{y}}_h}} |\widehat{\boldsymbol{m}}_h|^2\,\d\boldsymbol{\xi}\\
		&=\frac{1}{h}\int_\Omega |\widehat{\boldsymbol{m}}_h\circ\widehat{\boldsymbol{y}}_h|^2\,\det \nabla\widehat{\boldsymbol{y}}_h\,\d\boldsymbol{x}=\int_\Omega |\widehat{\boldsymbol{m}}_h\circ\widehat{\boldsymbol{y}}_h|^2\,\det \nabla_h\widehat{\boldsymbol{y}}_h\,\d\boldsymbol{x}\\
		&=\int_\Omega |\widehat{\boldsymbol{m}}_h\circ\widehat{\boldsymbol{y}}_h|\,\d\boldsymbol{x}=\int_\Omega \frac{1}{\det \nabla_h\widehat{\boldsymbol{y}}_h}\,\d\boldsymbol{x}.
		\end{split}
	\end{equation}
	Then, the claim follows from \eqref{eqn:a2}. We deduce the existence of $\widehat{\boldsymbol{\mu}}\in L^2(\rt;\rt)$ such that, up to subsequences,  $\widehat{\boldsymbol{\mu}}_h\wk \widehat{\boldsymbol{\mu}}$ in $L^2(\rt;\rt)$. However, 
	by \eqref{eqn:cm-subcylinder}--\eqref{eqn:cm-supercylinder}, there holds $\chi_{\boldsymbol{\pi}_h^{-1}(\Omega^{\widehat{\boldsymbol{y}}_h})}\to \chi_\Omega$ almost everywhere which, together with \eqref{eqn:cm-diagonal}, yields $\widehat{\boldsymbol{\mu}}_h\to \chi_\Omega \widehat{\boldsymbol{\zeta}}$ almost everywhere in $\rt$, as $h \to 0^+$. Also, by \eqref{eqn:cm-supercylinder}, the maps $\widehat{\boldsymbol{\mu}}_h$ are supported in a common compact set containing $\Omega$ for $h \ll 1$, so that $\widehat{\boldsymbol{\mu}}_h\to \chi_\Omega \widehat{\boldsymbol{\zeta}}$ in $L^1(\rt;\rt)$ by the Vitali Convergence Theorem. Thus, $\widehat{\boldsymbol{\mu}}=\chi_\Omega \widehat{\boldsymbol{\zeta}}$ and this establishes the weak convergence in 
	\eqref{eqn:cm-eta}.

	To prove \eqref{eqn:cm-H}, we observe that
	\begin{equation}
	\label{eqn:cm-H-norm}
	E_h^{\mathrm{exc}}(\widehat{\boldsymbol{q}}_h)=\frac{1}{h}\int_{\Omega^{\widehat{\boldsymbol{y}}_h}} |\nabla \widehat{\boldsymbol{m}}_h|^2\,\d\boldsymbol{\xi}=\int_{\boldsymbol{\pi}_h^{-1}(\Omega^{\widehat{\boldsymbol{y}}_h})} |\nabla \widehat{\boldsymbol{m}}_h|^2\circ \boldsymbol{\pi}_h\,\d\boldsymbol{x}=\int_{\rt} |\widehat{\boldsymbol{N}}_h|^2\,\d\boldsymbol{x},
	\end{equation}
	where we employed the change-of-variable formula. In view of \eqref{eqn:exc-bdd}, this shows the boundedness of $(\widehat{\boldsymbol{N}}_h)$ in $L^2(\rt;\rtt)$. 
	To check \eqref{eqn:cm-H}, let $\boldsymbol{\Phi}\in L^2(\rt;\rtt)$ and set $\widehat{\boldsymbol{N}}\coloneqq \chi_\Omega (\nabla'\widehat{\boldsymbol{\zeta}}\vert \widehat{\boldsymbol{\chi}})$. Given $\eta>0$ and $0<\vartheta<1$, we write
	\begin{equation}
	\label{eqn:cm-H-weak}
	\int_{\rt} (\widehat{\boldsymbol{N}}_h-\widehat{\boldsymbol{N}}):\boldsymbol{\Phi}\,\d \boldsymbol{x}=\int_{\Omega^\eta_\vartheta} (\widehat{\boldsymbol{N}}_h-\widehat{\boldsymbol{N}}):\boldsymbol{\Phi}\,\d \boldsymbol{x}+\int_{\rt \setminus \Omega^\eta_\vartheta} (\widehat{\boldsymbol{N}}_h-\widehat{\boldsymbol{N}}):\boldsymbol{\Phi}\,\d \boldsymbol{x}.
	\end{equation}
	The first integral on the right-hand side of \eqref{eqn:cm-H-weak} goes to zero, as $h \to 0^+$. Indeed, by \eqref{eqn:cm-subcylinder}, for every  $h \leq \bar{h}(\eta,\vartheta)$, we have
	\begin{equation*}
	\begin{split}
			\int_{\Omega^\eta_\vartheta} \left (\widehat{\boldsymbol{N}}_h-\widehat{\boldsymbol{N}} \right):\boldsymbol{\Phi}\,\d\boldsymbol{x}&=\int_{\Omega^\eta_\vartheta} \left (\nabla \widehat{\boldsymbol{m}}_h \circ \boldsymbol{\pi}_h-(\nabla'\widehat{\boldsymbol{\zeta}}|\widehat{\boldsymbol{\chi}})\right):\boldsymbol{\Phi}\,\d\boldsymbol{x}
			=\int_{\Omega^\eta_\vartheta} \left (\nabla_h\widehat{\boldsymbol{\zeta}}_h-(\nabla'\widehat{\boldsymbol{\zeta}}|\widehat{\boldsymbol{\chi}})\right):\boldsymbol{\Phi}\,\d\boldsymbol{x},
	\end{split}
	\end{equation*}
	where the right-hand side goes to zero, as $h \to 0^+$, by \eqref{eqn:cm-diagonal}. The second integral on the right-hand side of \eqref{eqn:cm-H-weak} goes as well to zero, as $h \to 0^+$. Indeed, by \eqref{eqn:cm-supercylinder} and by the boundedness of $(\widehat{\boldsymbol{N}}_h)$ in $L^2(\rt;\rtt)$, for every $\ell>1$ and for every  $h\leq \underline{h}(\eta,\ell)$, there holds
	\begin{equation*}
	\begin{split}
	\left| \int_{\rt\setminus \Omega^\eta_\vartheta}\left (\widehat{\boldsymbol{N}}_h-\widehat{\boldsymbol{N}}\right):\boldsymbol{\Phi}\,\d\boldsymbol{x}\right |&=\left| \int_{\Omega^{-\eta}_\ell\setminus\Omega^\eta_\vartheta} \left (\widehat{\boldsymbol{N}}_h- \widehat{\boldsymbol{N}}\right):\boldsymbol{\Phi}\,\d\boldsymbol{x}\right |\leq C\,||\boldsymbol{\Phi}||_{L^2(\Omega^{-\eta}_\ell\setminus\Omega^\eta_\vartheta;\rtt)}.
	\end{split}
	\end{equation*}
	As the right-hand side can be made arbitrarily small by properly choosing $\eta$, $\vartheta$ and $\ell$ according to $\boldsymbol{\Phi}$ only, this establishes \eqref{eqn:cm-H}.
}

\textbf{Step 3 (Convergence of compositions).} 
{\MMM
We now prove  \eqref{eqn:cm-composition}. Recall that $(1/\det \nabla_h\widehat{\boldsymbol{y}}_h)$ is equi-integrable thanks to \eqref{eqn:a2}. Thus, in view of \eqref{eqn:saturation-sequence}, so is $(\widehat{{\boldsymbol{m}}}_h\circ \widehat{\boldsymbol{y}}_h)$. For this reason, in order to prove the claim, it is sufficient to show that, for every $\eta>0$ and $0<\vartheta<1$, we have $\widehat{\boldsymbol{m}}_h \circ \widehat{\boldsymbol{y}}_h\to\widehat{\boldsymbol{\zeta}}$ in $L^1(\Omega^\eta_\vartheta;\rt)$.

}

Fix $\eta>0$ and $0<\vartheta<1$. Recall \eqref{eqn:cm-subcylinder} and consider  $h\leq \bar{h}(\eta,\vartheta)$. We have
\begin{equation}
\label{eqn:cm-split}
    \int_{\Omega^\eta_\vartheta} |\widehat{\boldsymbol{m}}_h\circ \widehat{\boldsymbol{y}}_h-\widehat{\boldsymbol{\zeta}}|^2\,\d\boldsymbol{x}\leq \int_{\Omega^\eta_\vartheta} |\widehat{\boldsymbol{m}}_h\circ \widehat{\boldsymbol{y}}_h-\widehat{\boldsymbol{\zeta}}_h|^2\,\d\boldsymbol{x}+\int_{\Omega^\eta_\vartheta} |\widehat{\boldsymbol{\zeta}}_h-\widehat{\boldsymbol{\zeta}}|^2\,\d\boldsymbol{x}.
\end{equation}
By \eqref{eqn:cm-diagonal}, up to subsequences, the second integral on the right-hand side of \eqref{eqn:cm-split} goes to zero, as $h \to 0^+$. Thus, we focus on the first one and we show that, up to subsequences, it goes as well to zero, as $h \to 0^+$.

{\MMM 
Let $\varepsilon>0$ be arbitrary. Let also $0<\eta_1<\eta$ and $0<\vartheta<\vartheta_1<1$. Observe that, for  $h\leq \overline{h}(\eta_1,\vartheta_1)$, the sequence $(\widehat{\boldsymbol{\zeta}}_h)$  is equi-integrable on $\Omega^{\eta_1}_{\vartheta_1}$ by \eqref{eqn:cm-diagonal}. 
	Therefore, there exists $\delta(\eta_1,\vartheta_1,\varepsilon)>0$ such that the following property holds:
	\begin{equation}
		\label{eqn:cylinder-equi-integrability}
			\forall\, A \subset \Omega^{\eta_1}_{\vartheta_1}\:\text{measurable}:\,\leb(A)<\delta(\eta_1,\vartheta_1,\varepsilon),
			\quad \sup_{h\leq \overline{h}(\eta_1,\vartheta_1)} \int_A |\widehat{\boldsymbol{m}}_h\circ \widehat{\boldsymbol{y}}_h-\widehat{\boldsymbol{\zeta}}_h|\,\d\boldsymbol{x}<\varepsilon.
	\end{equation}
	 Set $A^\eta_{\vartheta,h}\coloneqq \widehat{\boldsymbol{y}}_h^{-1}(\Omega^\eta_{\vartheta h})$. Let $\eta_1<\eta_0<\eta$ and $\vartheta<\vartheta_0<\vartheta_1$ be such that
	 \begin{equation}
	 	\label{eqn:cylinder-choice}
	 	\lebt(\Omega^{\eta_0}_{\vartheta_0}\setminus \Omega^\eta_\vartheta)<\delta(\eta_1,\vartheta_1,\varepsilon).
	 \end{equation}	 
	 By \eqref{eqn:cm-subcylinder-variant}, for every $h \leq \widetilde{h}(\eta,\eta_0,\vartheta,\vartheta_0)$, we have $A^\eta_{\vartheta,h}\subset \Omega^{\eta_0}_{\vartheta_0}$ and 
	we write
	\begin{equation}
		\label{eqn:cm-A}
		\begin{split}
			\int_{\Omega^\eta_\vartheta} |\widehat{\boldsymbol{m}}_h\circ \widehat{\boldsymbol{y}}_h-\widehat{\boldsymbol{\zeta}}_h|\,\d\boldsymbol{x}=\int_{\Omega^\eta_\vartheta \cap A^\eta_{\vartheta,h}} |\widehat{\boldsymbol{m}}_h\circ \widehat{\boldsymbol{y}}_h-\widehat{\boldsymbol{\zeta}}_h|\,\d\boldsymbol{x}
			+\int_{\Omega^\eta_\vartheta \setminus A^\eta_{\vartheta,h}} |\widehat{\boldsymbol{m}}_h\circ \widehat{\boldsymbol{y}}_h-\widehat{\boldsymbol{\zeta}}_h|\,\d\boldsymbol{x}.
		\end{split}
	\end{equation}
	We focus on the second integral on the right-hand side of the previous equation. For convenience, set $\sigma_h\coloneqq \det \nabla_h \widehat{\boldsymbol{y}}_h-1$. Thus, $\sigma_h\to 0$ in $L^1(\Omega)$ by \eqref{eqn:cm-gradient-estimate} since $\beta>p>3$. Using the change-of-variable formula, we compute
	\begin{equation*}
		\lebt(\Omega^\eta_{\vartheta h})=\lebt(\widehat{\boldsymbol{y}}_h(A^\eta_{\vartheta,h}))=\int_{A^\eta_{\vartheta,h}} \det \nabla \widehat{\boldsymbol{y}}_h\,\d\boldsymbol{x}=h\lebt(A^\eta_{\vartheta,h})+h \int_{A^\eta_{\vartheta,h}} \sigma_h\,\d\boldsymbol{x}.
	\end{equation*}
	Then
	\begin{equation*}
		\lebt(A^\eta_{\vartheta,h})=\lebt(\Omega^\eta_\vartheta)-\int_{A^\eta_{\vartheta,h}} \sigma_h\,\d\boldsymbol{x},
	\end{equation*}
	so that
	\begin{equation*}
		\begin{split}
			\lebt(\Omega^\eta_\vartheta \setminus A^\eta_{\vartheta,h})&\leq \lebt(\Omega^{\eta_0}_{\vartheta_0}\setminus A^\eta_{\vartheta,h})=\lebt(\Omega^{\eta_0}_{\vartheta_0})-\lebt(A^\eta_{\vartheta,h})
			=\lebt(\Omega^{\eta_0}_{\vartheta_0})-\lebt(\Omega^\eta_\vartheta)+\int_{A^\eta_{\vartheta,h}} \sigma_h\,\d\boldsymbol{x}.
		\end{split}
	\end{equation*}
	Here, we exploited the inclusion $A^\eta_{\vartheta,h}\subset \Omega^{\eta_0}_{\vartheta_0}$.  
	This yields
	\begin{equation*}
		\limsup_{h \to 0^+} \lebt(\Omega^\eta_\vartheta \setminus A^\eta_{\vartheta,h})\leq \lebt(\Omega^{\eta_0}_{\vartheta_0}\setminus \Omega^\eta_\vartheta)<\delta(\eta_1,\vartheta_1,\varepsilon),
	\end{equation*}
	where we used \eqref{eqn:cylinder-choice}.
	Therefore, from \eqref{eqn:cylinder-equi-integrability}, we obtain
	\begin{equation}
		\label{eqn:bb1}
		\limsup_{h \to 0^+} \int_{\Omega^\eta_\vartheta \setminus A^\eta_{\vartheta,h}} |\widehat{\boldsymbol{m}}_{h}\circ \widehat{\boldsymbol{y}}_{h}-\widehat{\boldsymbol{\zeta}}_{h}|\leq \varepsilon.
	\end{equation}
	
	To estimate the first integral on the right-hand side of \eqref{eqn:cm-A}, we proceed as follows.
	Without loss of generality, we can assume that $\Omega^{\eta}_{\vartheta}$ is a Lipschitz domain. In this case, the map
}
$\widehat{\boldsymbol{\zeta}}_h\in W^{1,2}(\Omega^\varepsilon_\vartheta;\S^2)$ admits an extension $\widehat{\boldsymbol{Z}}_h\in W^{1,2}(\R^3;\R^3)$, possibly dependent on $\varepsilon$ and $\vartheta$, which satisfies
\begin{equation*}
    \|\widehat{\boldsymbol{Z}}_h\|_{W^{1,2}(\R^3;\R^3)}\leq C(\eta,\vartheta)\,\|\widehat{\boldsymbol{\zeta}}_h\|_{W^{1,2}(\Omega^\eta_\vartheta;\R^3)}.
\end{equation*}
In particular, recalling \eqref{eqn:exc-bdd2}, we have
\begin{equation}
    \label{eqn:cm-nabla-Z}
    \|\nabla \widehat{\boldsymbol{Z}}_h\|_{L^2(\R^3;\rtt)}\leq C(\eta,\tht) \left ( \int_{\Omega^\eta\vartheta} |\widehat{\boldsymbol{\zeta}}_h|^2\,\d\boldsymbol{x}+\int_{\Omega^\eta_\vartheta} |\nabla\widehat{\boldsymbol{\zeta}}_h|^2\,\d\boldsymbol{x}\right)\leq C(\eta,\vartheta).
\end{equation}
Define $\widehat{\boldsymbol{M}}_h\coloneqq \widehat{\boldsymbol{Z}}_h\circ \boldsymbol{\pi}_h^{-1}$. By construction, $\widehat{\boldsymbol{M}}_h\restr{\Omega^\eta_{\vartheta h}}=\widehat{\boldsymbol{m}}_h\restr{\Omega^\eta_{\vartheta h}}$ and, by \eqref{eqn:cm-nabla-Z} and the change-of-variable formula, there holds
\begin{equation}
    \label{eqn:cm-nabla-M}
    \begin{split}
        \int_{\R^3} |\nabla \widehat{\boldsymbol{M}}_h|^2\,\d\boldsymbol{\xi}&=\int_{\R^3} |\nabla_h \widehat{\boldsymbol{Z}}_h\circ \boldsymbol{\pi}_h^{-1}|^2\,\d\boldsymbol{\xi}\leq \frac{1}{h^2} \int_{\R^3} |\nabla \widehat{\boldsymbol{Z}}_h\circ \boldsymbol{\pi}_h^{-1}|^2\,\d\boldsymbol{\xi}\\
        &=\frac{1}{h} \int_{\R^3} |\nabla \widehat{\boldsymbol{Z}}_h|^2\,\d \boldsymbol{x}\leq \frac{C(\eta,\vartheta)}{h}.
    \end{split}
\end{equation}
Let $\lambda>0$. By a Lusin-type property of Sobolev maps \cite{acerbi.fusco}, there exists a measurable set $F_{\lambda,h}\subset \R^3$ such that $\widehat{\boldsymbol{M}}_h\restr{F_{\lambda,h}}$ is Lipschitz continuous with constant $C\lambda>0$, that is
\begin{equation}
    \label{eqn:cm-lipschitz}
    \forall\,\boldsymbol{\xi},\widehat{\boldsymbol{\xi}}\in F_{\lambda,h}, \quad |\widehat{\boldsymbol{M}}_h(\boldsymbol{\xi})-\widehat{\boldsymbol{M}}_h(\widehat{\boldsymbol{\xi}})|\leq C\lambda\,|\boldsymbol{\xi}-\widehat{\boldsymbol{\xi}}|.
\end{equation}
Moreover, thanks to \eqref{eqn:cm-nabla-M}, the measure of the complement of the set $F_{\lambda,h}$ is controlled as follows
\begin{equation}
    \label{eqn:cm-complement}
    \leb(\R^3\setminus F_{\lambda,h})\leq \frac{C}{\lambda^2}\int_{\{|\nabla \widehat{\boldsymbol{M}}_h|\geq \lambda/2\}} |\nabla \widehat{\boldsymbol{M}}_h|^2\,\d \boldsymbol{\xi}\leq \frac{C(\eta,\vartheta)}{\lambda^2h},
\end{equation}
where we used \eqref{eqn:cm-nabla-M}. 

Going back to the first integral on the right-hand side of \eqref{eqn:cm-split}, {\MMM
setting
	\begin{equation*}
		X_{\lambda,h}\coloneqq \widehat{\boldsymbol{y}}_h^{-1}(F_{\lambda,h}), \qquad Y_{\lambda,h}\coloneqq \boldsymbol{\pi}_h^{-1}(F_{\lambda,h}),
	\end{equation*}
	we split it as
	\begin{equation}
	\label{eqn:cm-split-bis}
	\begin{split}
	\int_{\Omega^\eta_\vartheta \cap A^\eta_{\vartheta,h}} |\widehat{\boldsymbol{m}}_h\circ \widehat{\boldsymbol{y}}_h-\widehat{\boldsymbol{\zeta}}_h|\,\d\boldsymbol{x}&=\int_{(\Omega^\eta_\vartheta \cap A^\eta_{\vartheta,h})\cap (X_{\lambda,h}\cap Y_{\lambda,h})} |\widehat{\boldsymbol{m}}_h\circ \widehat{\boldsymbol{y}}_h-\widehat{\boldsymbol{\zeta}}_h|\,\d\boldsymbol{x}\\
	&+\int_{(\Omega^\eta_\vartheta \cap A^\eta_{\vartheta,h})\setminus (X_{\lambda,h}\cap Y_{\lambda,h})} |\widehat{\boldsymbol{m}}_h\circ \widehat{\boldsymbol{y}}_h-\widehat{\boldsymbol{\zeta}}_h|\,\d\boldsymbol{x}.
	\end{split}
	\end{equation}
	For the first integral on the right-hand side of \eqref{eqn:cm-split-bis}, note that, for every $\boldsymbol{x}\in (\Omega^\eta_\vartheta \cap A^\eta_{\vartheta,h})\cap (X_{\lambda,h}\cap Y_{\lambda,h}),$
	both $\widehat{\boldsymbol{y}}_h(\boldsymbol{x})$ and $\boldsymbol{\pi}_h(\boldsymbol{x})$ belong to $\Omega^\eta_{\vartheta h}\cap F_{\lambda,h}$. Thus, by \eqref{eqn:cm-estimate} and \eqref{eqn:cm-lipschitz}, we have
	\begin{equation}
		\label{eqn:bb2}
		\begin{split}
			\int_{(\Omega^\eta_\vartheta \cap A^\eta_{\vartheta,h})\cap (X_{\lambda,h}\cap Y_{\lambda,h})} |\widehat{\boldsymbol{m}}_h\circ \widehat{\boldsymbol{y}}_h-\widehat{\boldsymbol{\zeta}}_h|\,\d\boldsymbol{x}&=\int_{(\Omega^\eta_\vartheta \cap A^\eta_{\vartheta,h})\cap (X_{\lambda,h}\cap Y_{\lambda,h})} |\widehat{\boldsymbol{M}}_h\circ \widehat{\boldsymbol{y}}_h-\widehat{\boldsymbol{M}}_h\circ \boldsymbol{\pi}_h|\,\d\boldsymbol{x}\\
			&\leq C\lambda \|\widehat{\boldsymbol{y}}_h-\boldsymbol{\pi}_h\|_{C^0(\closure{\Omega};\rt)}\leq C\lambda h^\alpha.
		\end{split}
	\end{equation}
	For the second integral on the right-hand side of \eqref{eqn:cm-split-bis}, observe that 
	\begin{equation}
		\label{eqn:b3}
		(\Omega^\eta_\vartheta \cap A^\eta_{\vartheta,h}) \setminus (X_{\lambda,h} \cap Y_{\lambda,h})\subset  \left (\Omega^\eta_{\vartheta }  \setminus X_{\lambda,h} \right ) \cup \left (\Omega^\eta_{\vartheta}  \setminus Y_{\lambda,h} \right ).
	\end{equation}
	By the change-of-variable formula,  there holds
	\begin{equation*}
		\begin{split}
			\lebt(\widehat{\boldsymbol{y}}_h(\Omega^\eta_\vartheta)\setminus F_{\lambda,h})&=\lebt(\widehat{\boldsymbol{y}}_h(\Omega^\eta_\vartheta \setminus X_{\lambda,h}))=\int_{\Omega^\eta_\vartheta \setminus X_{\lambda,h}} \det \nabla \widehat{\boldsymbol{y}}_h\,\d\boldsymbol{x}\\
			&=h \lebt(\Omega^\eta_\vartheta \setminus X_{\lambda,h}) + h \int_{\Omega^\eta_\vartheta \setminus X_{\lambda,h}} \sigma_h\,\d\boldsymbol{x}.
		\end{split}
	\end{equation*}
	From this, recalling \eqref{eqn:cm-complement}, we obtain
	\begin{equation}
		\label{eqn:b4}
		\begin{split}
			\lebt(\Omega^\eta_\vartheta \setminus X_{\lambda,h})=\frac{1}{h}\lebt(\widehat{\boldsymbol{y}}_h(\Omega^\eta_\vartheta)\setminus F_{\lambda,h})+\int_{\Omega^\eta_\vartheta \setminus X_{\lambda,h}} \sigma_h\,\d\boldsymbol{x}
			\leq \frac{C(\eta,\vartheta)}{\lambda^2 h^2}+\int_{\Omega^\eta_\vartheta \setminus X_{\lambda,h}} \sigma_h\,\d\boldsymbol{x}. 
		\end{split}
	\end{equation}
	Instead, applying the change-of-variable formula and exploiting \eqref{eqn:cm-complement}, we estimate
	\begin{equation}
		\label{eqn:b5}
		\lebt(\Omega^\eta_\vartheta \setminus Y_{\lambda,h})=\lebt(\boldsymbol{\pi}_h^{-1}(\Omega^\eta_{\vartheta h}\setminus F_{\lambda,h}))=\frac{1}{h}\lebt(\Omega^\eta_{\vartheta h} \setminus F_{\lambda,h})\leq \frac{C(\eta,\vartheta)}{\lambda^2 h^2}.
	\end{equation}
	Now, recalling that $\alpha>1$, we choose $\lambda=h^{-b}$ for some $1<b<\alpha$. In this case, the bound in \eqref{eqn:bb2} yields
	\begin{equation}
		\label{eqn:b6}
		\lim_{h \to 0^+} \int_{(\Omega^\eta_\vartheta \cap A^\eta_{\vartheta,h})\cap (X_{\lambda,h}\cap Y_{\lambda,h})} |\widehat{\boldsymbol{m}}_h\circ \widehat{\boldsymbol{y}}_h-\widehat{\boldsymbol{\zeta}}_h|\,\d\boldsymbol{x}=0.
	\end{equation}
	Also, from \eqref{eqn:b3}--\eqref{eqn:b5}, we obtain
	\begin{equation*}
		\lim_{h \to 0^+} \leb((\Omega^\eta_\vartheta \cap A^\eta_{\vartheta,h}) \setminus (X_{\lambda,h} \cap Y_{\lambda,h}))=0,
	\end{equation*}
	which, by \eqref{eqn:cylinder-equi-integrability}, entails
	\begin{equation}
		\label{eqn:b7}
		\limsup_{h \to 0^+} \int_{(\Omega^\eta_\vartheta \cap A^\eta_{\vartheta,h})\setminus (X_{\lambda,h}\cap Y_{\lambda,h})} |\widehat{\boldsymbol{m}}_h\circ \widehat{\boldsymbol{y}}_h-\widehat{\boldsymbol{\zeta}}_h|\,\d\boldsymbol{x}\leq \varepsilon.	
	\end{equation}
	Therefore, combining  \eqref{eqn:bb1} and \eqref{eqn:b6}--\eqref{eqn:b7}, we obtain
	\begin{equation*}
		\limsup_{h \to 0^+} \int_{\Omega^\eta_\vartheta} |\widehat{\boldsymbol{m}}_h \circ \widehat{\boldsymbol{y}}_h-\widehat{\boldsymbol{\zeta}}_h|\,\d\boldsymbol{x}\leq 2\varepsilon.
	\end{equation*}
	Since $\varepsilon>0$ is arbitrary, this shows that the first integral on the right-hand side of \eqref{eqn:cm-split} goes to zero, as $h \to 0^+$.

	\textbf{Step 4 (Improved convergences).} We are left to prove the strong convergences in \eqref{eqn:cm-eta} and \eqref{eqn:cm-lagrangian}, and that $|\widehat{\boldsymbol{\zeta}}|=1$ almost everywhere in $S$. By \eqref{eqn:cm-gradient-estimate}, since $\beta>p$, there holds $\widehat{\boldsymbol{F}}_h\to \boldsymbol{I}$ in $L^p(\Omega;\rtt)$. Thus, up to subsequences, we have $\det \widehat{\boldsymbol{F}}_h \to 1$ almost everywhere.  By \eqref{eqn:cm-composition}, up to subsequences, we also have $\widehat{\boldsymbol{m}}_h\circ\widehat{\boldsymbol{y}}_h \to \widehat{\boldsymbol{\zeta}}$ almost everywhere. Taking into account \eqref{eqn:saturation-sequence}, this gives \eqref{eqn:cm-lagrangian} by the dominated convergence theorem. From \eqref{eqn:saturation-sequence}, passing to the limit, as $h\to 0^+$, we deduce that $|\widehat{\boldsymbol{\zeta}}|=1$ almost everywhere in $\Omega$. Equivalently, $\widehat{\boldsymbol{\zeta}}\in W^{1,2}(S;\S^2)$. Finally, passing to the limit, as $h\to 0^+$, in \eqref{eqn:a3} , we obtain
	\begin{equation*}
		\lim_{h \to 0^+} \int_{\rt} |\widehat{\boldsymbol{\mu}}_h|^2\,\d\boldsymbol{\xi}=\lebt(\Omega)=\int_{\rt} |\chi_\Omega \widehat{\boldsymbol{\zeta}}|^2\,\d\boldsymbol{\xi}.
	\end{equation*}
	Since we already proved the weak convergence, this proves the strong convergence in \eqref{eqn:cm-eta}.
}
\end{proof}

{\MMM The compactness properties of sequences of admissible states with equi-bounded energy deduced from the previous results are summarized in the next proposition.}

\begin{proposition}[Compactness]
\label{prop:compactness}
Let $(\boldsymbol{q}_h)$ with $\boldsymbol{q}_h=(\boldsymbol{y}_h,\boldsymbol{m}_h)\in {\mathcal{Q}}_h$ be such that
\begin{equation}
	\label{eqn:bddc}
	\sup_{h>0} \left \{ E_h^{\mathrm{el}}(\boldsymbol{q}_h)+E_h^{\mathrm{mag}}(\boldsymbol{q}_h)\right\}\leq C.
\end{equation}
Then,  there exist $\boldsymbol{q}_0=(\boldsymbol{u},v,\boldsymbol{\zeta})\in{\mathcal{Q}}_0$ and $\boldsymbol{\chi}\in L^2(\rt;\rt)$ such that, up to subsequences, the following convergences hold, as $h \to 0^+$:
\begin{align}
\label{eqn:compactness-clamped-horizontal}
	\text{$\boldsymbol{u}_h\coloneqq \mathcal{U}_h(\boldsymbol{q}_h)$}&\text{$\wk \boldsymbol{u}$ in $W^{1,2}(S;\R^2)$,}\\
	\label{eqn:compactness-clamped-vertical}
	\text{$v_h\coloneqq \mathcal{V}_h(\boldsymbol{q}_h)$}&\text{$\to v$ in $W^{1,2}(S)$,}\\
	\label{eqn:compactness-clamped-moment}
	\text{$\boldsymbol{w}_h\coloneqq \mathcal{W}_h(\boldsymbol{q}_h)$}&\text{$\to -\frac{1}{12}\left ( \begin{array}{cc}
		\nabla' {v}  \\
		0
		\end{array}\right)$ in $W^{1,2}(S;\rt)$,}\\
	\label{eqn:compactness-clamped-eta}
	\text{$\boldsymbol{\mu}_h\coloneqq \mathcal{M}_h(\boldsymbol{q}_h)$}&\text{$\to\chi_\Omega \boldsymbol{\zeta}$ in $L^2(\rt;\rt)$,}\\
	\label{eqn:compactness-clamped-H}
	\text{$\boldsymbol{N}_h\coloneqq \mathcal{N}_h(\boldsymbol{q}_h)$}&\text{$\wk \chi_\Omega \left ( 
		\nabla' \boldsymbol{\zeta}  \vert 
		\boldsymbol{\chi}\right)$ in $L^2(\rt;\rtt)$,}\\
	\label{eqn:compactness-clamped-composition}
	\text{$\boldsymbol{m}_h \circ \boldsymbol{y}_h$}&\text{$\to\boldsymbol{\zeta}$ in $L^1(\Omega;\rt)$,}   \\
	\label{eqn:compactness-clamped-lagrangian}
	\text{$\boldsymbol{z}_h\coloneqq \mathcal{Z}_h(\boldsymbol{q}_h)$}&\text{$\to\boldsymbol{\zeta}$ in $L^1(\Omega;\rt)$.}    
	\end{align}
{\MMM Additionally, there exist $(\boldsymbol{R}_h)\subset W^{1,2}(S;SO(3))$ and $\boldsymbol{G}\in L^2(\Omega;\rtt)$ such that, setting $\boldsymbol{F}_h\coloneqq \nabla_h\boldsymbol{y}_h$, we have
\begin{align}
    \label{eqn:compactness-clamped-R}
    \text{$\boldsymbol{R}_h$}&\text{$\to \boldsymbol{I}$ in $L^2(\Omega;\rtt)$,}\\
    \label{eqn:compactness-clamped-G}
    \text{$\boldsymbol{G}_h\coloneqq h^{-\beta/2}(\boldsymbol{R}_h^\top \boldsymbol{F}_h-\boldsymbol{I})$}&\text{$\wk \boldsymbol{G}$ in $L^2(\Omega;\rtt)$,}
\end{align}
and, for almost every $\boldsymbol{x}\in\Omega$, there holds
\begin{equation}
    \label{eqn:compactness-clamped-G-structure}
    \boldsymbol{G}''(\boldsymbol{x})=\sym \nabla'\boldsymbol{u}(\boldsymbol{x}')-((\nabla')^2\boldsymbol{v}(\boldsymbol{x}'))x_3.
\end{equation}
}
\end{proposition}
\begin{proof}
{\MMM
Recall \eqref{eqn:density-W_h}--\eqref{eqn:prestrain-inverse}. For simplicity, we set 
	\begin{equation*}
	\boldsymbol{F}_h \coloneqq \nabla_h \boldsymbol{y}_h, \quad  \quad    \boldsymbol{L}_h\coloneqq \boldsymbol{L}_h(\boldsymbol{\lambda}_h\,(\det\boldsymbol{F}_h)), \quad \boldsymbol{\Xi}_h\coloneqq \boldsymbol{L}_h^{-1}\boldsymbol{F}_h.
	\end{equation*} 
First, recalling \eqref{eqn:prestrain}--\eqref{eqn:ch-limit}, we write	\begin{equation*}
	\begin{split}
	\dist(\boldsymbol{F}_h;SO(3))&\leq |\boldsymbol{F}_h-\boldsymbol{\Xi}_h|+\dist(\boldsymbol{\Xi}_h;SO(3))
	\leq |\boldsymbol{L}_h-\boldsymbol{I}|\,|\boldsymbol{\Xi}_h|+\dist(\boldsymbol{\Xi}_h;SO(3))\\
	&\leq Ch^{\beta/2}|\boldsymbol{\Xi}_h|+\dist(\boldsymbol{\Xi}_h;SO(3))
	\leq C\left(h^{\beta/2}+ \dist(\boldsymbol{\Xi}_h;SO(3)) \right).
	\end{split}
	\end{equation*}
	Thus, for $q \in \{2,p\}$, we have
	\begin{equation*}
		\begin{split}
		\int_\Omega \dist^q(\boldsymbol{F}_h;SO(3))\,\d\boldsymbol{x}&\leq C \left(h^{(\beta/2)q} +\int_\Omega \dist^q(\boldsymbol{\Xi}_h;SO(3))\,\d\boldsymbol{x}  \right).
		\end{split}
	\end{equation*}
	Recalling \eqref{eqn:coercivity-Phi} and adopting the notation in \eqref{eqn:rig}, this shows that
	\begin{equation}
	\label{eqn:rigidity}
		\mathcal{R}_h(\boldsymbol{y}_h)\leq C h^\beta \left( E_h^{\rm el}(\boldsymbol{q}_h)+1  \right).
	\end{equation}
	}
	
Thus, setting $r_h\coloneqq \mathcal{R}_h(\boldsymbol{y}_h)$, from \eqref{eqn:bddc} we obtain the bound $r_h\leq Ch^\beta$. 

By applying Lemma \ref{lem:ar}, we find $(\boldsymbol{R}_h)\subset W^{1,p}(S;SO(3))$ and $(\boldsymbol{Q}_h)\subset SO(3)$ such that the following estimates hold for $q \in \{2,p\}$:
	{\MMM
	\begin{align}
    \label{eqn:cmpt-est1}
    \|\boldsymbol{F}_h-\boldsymbol{R}_h\|_{L^q(\Omega;\rtt)}&\leq C r_h^{1/q}, \qquad \hspace*{8mm} \|\nabla'\boldsymbol{R}_h\|_{L^q(S;\R^{3 \times 3 \times 3})}\leq C h^{-1} r_h^{1/q},\\
    \label{eqn:cmpt-est2}
    \|\boldsymbol{R}_h-\boldsymbol{Q}_h\|_{L^q(S;\rtt)}&\leq C h^{-1}r_h^{1/q}, \qquad \|\boldsymbol{F}_h-\boldsymbol{Q}_h\|_{L^q(S;\R^{3 \times 3})}\leq C h^{-1}r_h^{1/q}.
\end{align}
	}
	
Then, thanks to Lemma \ref{lem:clamped}, we obtain the estimate
\begin{equation*}
		|\boldsymbol{Q}_h-\boldsymbol{I}|\leq C \frac{\sqrt{r_h}}{h},
\end{equation*}
which, together with \eqref{eqn:cmpt-est2}, yields
\begin{equation}
			\label{eqn:cbcq}
			\|\boldsymbol{R}_h-\boldsymbol{I}\|_{L^q(S;\rtt)}\leq C h^{-1}r_h^{1/q}, \qquad \|\boldsymbol{F}_h-\boldsymbol{I}\|_{L^q(S;\R^{3 \times 3})}\leq C h^{-1}r_h^{1/q}.
\end{equation}
{\MMM At this point, we apply Proposition \ref{prop:cd} to $\widehat{\boldsymbol{y}}=\boldsymbol{y}_h$ with $e_h=h^\beta$, so that \eqref{eqn:compactness-clamped-horizontal}--\eqref{eqn:compactness-clamped-moment} are proved.}

{\MMM
Now, given the assumption $\beta>6 \vee p$, we have $\beta/2-1>0$ and $\beta/p-1>0$.  Let $0<s<1$ and let $2 \leq q_s \leq p$ be such that $1/{q_s}=s/2+(1-s)/p$.
	By the interpolation inequality \cite[Proposition 1.1.14]{grafakos} and  the second estimate in \eqref{eqn:cbcq}, there holds
	\begin{equation}
		\label{eqn:interpolation}
		\|\widetilde{\boldsymbol{F}}_h-\boldsymbol{I}\|_{L^{q_s}(\Omega;\rtt)}\leq \|\widetilde{\boldsymbol{F}}_h-\boldsymbol{I}\|^s_{L^2(\Omega;\rtt)}\|\widetilde{\boldsymbol{F}}_h-\boldsymbol{I}\|^{1-s}_{L^p(\Omega;\rtt)}\leq C h^{\alpha_s},
	\end{equation}
	where we set $\alpha_s\coloneqq s(\beta/2-1)+(1-s)(\beta/p-1)$. We choose $s$ in order to have
	\begin{equation}
		\label{eqn:interpolation-param}
		q_s>3,\qquad \alpha_s>1.
	\end{equation}
	Note that these two conditions are equivalent to
	\begin{equation*}
		s<\frac{2(p-3)}{3(p-2)}, \qquad s>\frac{2(2p-\beta)}{\beta(p-2)},
	\end{equation*}
	respectively. Therefore, as
	\begin{equation*}
		\frac{2(2p-\beta)}{\beta(p-2)}<\frac{2(p-3)}{3(p-2)}
	\end{equation*}
	if and only if $\beta>6$, such a value $0<s<1$ always exists. Hence, from \eqref{eqn:interpolation}, by applying the Poincar\'{e} inequality and the Morrey embedding, we obtain the following estimate:
	\begin{equation}
	    \label{eqn:compactness-uniform-estimate}
	    \|\boldsymbol{y}_h-\boldsymbol{\pi}_h\|_{C^0(\Omega;\R^3)}\leq C h^{\alpha_s}.
	\end{equation}
	}Note that, here, we implicitly exploited the first condition in \eqref{eqn:interpolation-param}. Therefore, given \eqref{eqn:cbcq} and \eqref{eqn:compactness-uniform-estimate}, claims \eqref{eqn:compactness-clamped-eta}--\eqref{eqn:compactness-clamped-lagrangian} follow at once by applying Proposition \ref{prop:cm} to $\widehat{\boldsymbol{q}}_h=\boldsymbol{q}_h$.
	
	{\MMM Eventually, we note that \eqref{eqn:compactness-clamped-R} immediately follows from the first estimate in \eqref{eqn:cbcq}, while \eqref{eqn:compactness-clamped-G}--\eqref{eqn:compactness-clamped-G-structure} are established by applying Lemma \ref{lem:limiting-strain} to $\widehat{\boldsymbol{y}}_h=\boldsymbol{y}_h$ with $e_h=h^\beta$ taking into account the first estimate in \eqref{eqn:cmpt-est1} and \eqref{eqn:compactness-clamped-horizontal}--\eqref{eqn:compactness-clamped-vertical}.
	}
\end{proof}

{\MMM
\begin{remark}[Norm of the averaged displacements]
\label{rem:norm}
With the notation employed in the proof of Proposition \ref{prop:compactness}, by \eqref{eqn:rigidity},  we  have 
\begin{equation*}
    \frac{r_h}{h^\beta}\leq C \left (E_h^{\rm el}(\boldsymbol{q}_h)+1 \right ).
\end{equation*}
Hence,  by applying Proposition \ref{prop:cd} to $\widehat{\boldsymbol{y}}_h=\boldsymbol{y}_h$ with $e_h=h^\beta$ taking into account \eqref{eqn:cmpt-est1} and \eqref{eqn:cbcq},  we obtain the following estimates for $h\ll 1$: 
\begin{align*}
    \|\boldsymbol{u}_h\|_{W^{1,2}(S;\R^2)}&\leq C \left (\sqrt{E_h^{\rm el}(\boldsymbol{q}_h)}+1 \right ),\\
    \|v_h\|_{W^{1,2}(S)}&\leq C \left (\sqrt{E_h^{\rm el}(\boldsymbol{q}_h)}+1 \right ),\\
    \|\boldsymbol{w}_h\|_{W^{1,2}(S;\rt)}&\leq C \left (\sqrt{E_h^{\rm el}(\boldsymbol{q}_h)}+1 \right ).
\end{align*}
The first estimate is justified by the fact that $r_h/h^4\to 0$, as $h\to 0^+$, so that, for $h\ll 1$, there holds
\begin{equation*}
    \frac{r_h}{h^{\beta/2+2}}=\frac{\sqrt{r_h}}{h^2}\,\sqrt{\frac{r_h}{h^\beta}}\leq \sqrt{\frac{r_h}{h^\beta}}.
\end{equation*}
\end{remark}
}

\subsubsection{Lower bound}

We now move to the the proof of the lower bound. For future reference, we highlight the result regarding the convergence of the magnetostatic energy. This has already been proved in \cite[Proposition 4.7]{bresciani} by adapting the results in \cite{carbou,gioia.james}. For convenience of the reader, we briefly sketch the proof and we refer to the first paper for details. Recall the notation in \eqref{eqn:necas} and \eqref{eqn:lagrangian-magnetization}.

\begin{proposition}[Convergence of the magnetostatic energy]
	\label{prop:magnetostatic}
	Let $(\widehat{\boldsymbol{q}}_h)$ with $\widehat{\boldsymbol{q}}_h=(\widehat{\boldsymbol{y}}_h,\widehat{\boldsymbol{m}}_h)\in\mathcal{Q}_h$. Suppose that there exists $\widehat{\boldsymbol{\zeta}}\in W^{1,2}(S;\S^2)$ such that the following convergence holds, as $h \to 0^+$:
	\begin{equation}
	\label{eqn:magnetostatic-eta}
	\text{$\widehat{\boldsymbol{\mu}}_h\coloneqq \mathcal{M}_h(\widehat{\boldsymbol{q}}_h) \to \chi_\Omega \widehat{\boldsymbol{\zeta}}$ in $L^2(\rt;\rt)$.}
	\end{equation}
	Then, the following equality holds:
	\begin{equation}
	\label{eqn:magnetostatic-limit}
	\lim_{h \to 0^+} E_h^{\mathrm{mag}}(\widehat{\boldsymbol{q}}_h)=\frac{1}{2}\int_S |\zeta^3|^2\,\d\boldsymbol{x}'.
	\end{equation}
\end{proposition}
\begin{proof}
	Denote by $\widehat{\psi}_h \in V^{1,2}(\rt)$ a stray field potential corresponding to $\widehat{\boldsymbol{q}}_h$. Thus, we have the following:
	\begin{equation}
	\label{eqn:wk-max}
	\forall\,\varphi \in V^{1,2}(\rt),\quad \int_{\rt} \nabla \widehat{\psi}_h \cdot \nabla \varphi\,\d\boldsymbol{\xi}=\int_{\rt} \chi_{\Omega^{\widehat{\boldsymbol{y}}_h}} \widehat{\boldsymbol{m}}_h \cdot \nabla \varphi\,\d\boldsymbol{\xi}.
	\end{equation}
	By \eqref{eqn:maxwell-stability}, there holds
	\begin{equation*}
	\|\nabla \widehat{\psi}_h\|_{L^2(\rt;\rt)}\leq \|\chi_{\Omega^{\widehat{\boldsymbol{y}}_h}} \widehat{\boldsymbol{m}}_h\|_{L^2(\rt;\rt)}.
	\end{equation*}
	Taking the square at both sides and applying the change-of-variable formula, we estimate
	\begin{equation}
	\label{eqn:lb-magnetostatic-1}
	\int_{\rt} |\nabla \widehat{\psi}_h|^2\,\d\boldsymbol{\xi}\leq \int_{\rt} |\chi_{\Omega^{\widehat{\boldsymbol{y}}_h}} \widehat{\boldsymbol{m}}_h|^2\,\d\boldsymbol{\xi}=\int_{\rt} |\widehat{\boldsymbol{\mu}}_h|^2 \circ \boldsymbol{\pi}_h^{-1}\,\d\boldsymbol{\xi}=h \int_{\rt} |\widehat{\boldsymbol{\mu}}_h|^2\,\d\boldsymbol{\xi}.
	\end{equation}
	Define $\widehat{\rho}_h\coloneqq \widehat{\psi}_h \circ \boldsymbol{\pi}_h \in V^{1,2}(\rt)$. From \eqref{eqn:lb-magnetostatic-1}, using again the change-of-variable formula, we estimate
	\begin{equation*}
	\int_{\rt} |\nabla_h \widehat{\rho}_h|^2\,\d\boldsymbol{x}=\int_{\rt} |\nabla \widehat{\psi}_h|^2\circ \boldsymbol{\pi}_h\,\d\boldsymbol{x}=\frac{1}{h}\int_{\rt} |\nabla \widehat{\psi}_h|^2\,\d\boldsymbol{\xi}\leq \int_{\rt} |\widehat{\boldsymbol{\mu}}_h|^2\,\d\boldsymbol{\xi}.
	\end{equation*}
	As the right-hand side is uniformly bounded in view \eqref{eqn:magnetostatic-eta}, we deduce that $(\nabla_h \widehat{\rho}_h)$ is bounded in $L^2(\rt;\rt)$. From this, we deduce two facts. First, there exists $\widehat{\chi}\in L^2(\rt)$ such that, up to subsequences, $\partial_3 \widehat{\rho}_h/h \wk \widehat{\chi}$ in $L^2(\rt)$ and, in turn, $\partial_3 \widehat{\rho}_h\to 0$ in $L^2(\rt)$, as $h \to 0^+$. Second, exploiting the Hilbert space structure of the quotient $V^{1,2}(\rt)/\R$, we infer the existence of $\widehat{\rho}\in V^{1,2}(\rt)$ such that, up to subsequences, there holds $\nabla \widehat{\rho}_h \wk \nabla \widehat{\rho}$ in $L^2(\rt;\rt)$, as $h \to 0^+$. These two facts together imply that $\partial_3 \widehat{\rho}=0$ almost everywhere which, as $\nabla \widehat{\rho}\in L^2(\rt;\rt)$, yields $\nabla' \widehat{\rho}=\boldsymbol{0}'$ almost everywhere.  
	
	Now, testing  \eqref{eqn:wk-max} with $\varphi=\widehat{\psi}_h$ and applying the change-of-variable formula, we write
	\begin{equation*}
	\begin{split}
	E_h^{\mathrm{mag}}(\widehat{\boldsymbol{q}}_h)&=\frac{1}{2h} \int_{\rt} \chi_{\Omega^{\widehat{\boldsymbol{y}}_h}} \widehat{\boldsymbol{m}}_h \cdot \nabla \widehat{\psi}_h\,\d\boldsymbol{\xi}
	=\frac{1}{2}\int_{\rt} \widehat{\boldsymbol{\mu}}_h \cdot \nabla_h \widehat{\rho}_h\,\d\boldsymbol{x}\\
	&=\frac{1}{2}\int_{\rt} \widehat{\boldsymbol{\mu}}_h' \cdot \nabla' \widehat{\rho}_h\,\d\boldsymbol{x}+\frac{1}{2}\int_{\rt} \widehat{\boldsymbol{\mu}}_h^3 \, \frac{\partial_3 \widehat{\rho}_h}{h}\,\d\boldsymbol{x}.
	\end{split}
	\end{equation*}
	Then, recalling \eqref{eqn:magnetostatic-eta} and passing to the limit, we obtain
	\begin{equation}
	\label{eqn:magnetostatic-l}
	\lim_{h \to 0^+} E_h^{\mathrm{mag}}(\widehat{\boldsymbol{q}}_h)=\frac{1}{2}\int_\Omega \widehat{\zeta}^{\,3}\,\widehat{\chi}\,\d\boldsymbol{x}.
	\end{equation}
	Thus, if we show that $\widehat{\chi}=\chi_\Omega \widehat{\zeta}^{\,3}$ almost everywhere in $\R^3$, then \eqref{eqn:magnetostatic-limit} follows from \eqref{eqn:magnetostatic-l}. To check this, we go back to \eqref{eqn:wk-max}. Using once more the change-of-variable formula, we deduce that $\widehat{\rho}_h$ satisfies the following:
	\begin{equation*}
	\forall\,\varphi\in V^{1,2}(\rt), \qquad \int_{\rt} \nabla_h \widehat{\rho}_h \cdot \nabla_h \varphi\,\d\boldsymbol{x}=\int_{\rt} \widehat{\boldsymbol{\mu}}_h \cdot \nabla_h \varphi\,\d\boldsymbol{x}.
	\end{equation*}
	Multiplying by $h$ and then passing to the limit, as $h \to 0^+$, we obtain
	\begin{equation*}
	\forall\, \varphi \in V^{1,2}(\rt), \qquad \int_{\rt} \left (\widehat{\chi}-\chi_\Omega \widehat{\zeta}^{\, 3}\right )\,\partial_3 \varphi\,\d\boldsymbol{x}=0.
	\end{equation*}
	Given the arbitrariness of $\varphi$, this entails that the function $\widehat{\chi}-\chi_\Omega \widehat{\zeta}^{\,3}$ does not depend on the third variable. However, as this function belongs to $L^2(\rt)$, we necessarily have $\widehat{\chi}-\chi_\Omega \widehat{\zeta}^{\,3}=0$ almost everywhere. 
\end{proof}

The next result asserts the existence of a lower bound and, for future reference, it is presented in a more self-contained form.

\begin{proposition}[Lower bound]
	\label{prop:lb}
	Let $(\widehat{\boldsymbol{q}}_h)$ with $\widehat{\boldsymbol{q}}_h=(\widehat{\boldsymbol{y}}_h,\widehat{\boldsymbol{m}}_h)\in\mathcal{Q}_h$. Set $\widehat{\boldsymbol{F}}_h\coloneqq \nabla_h\widehat{\boldsymbol{y}}_h$. Suppose that there exist $(\widehat{\boldsymbol{R}}_h)\subset W^{1,2}(S;SO(3))$ and also $\widehat{\boldsymbol{G}}\in L^2(\Omega;\rtt)$ and $\widehat{\boldsymbol{q}}_0=(\widehat{\boldsymbol{u}},\widehat{v},\widehat{\boldsymbol{\zeta}})\in\mathcal{Q}_0$ such that, as $h \to 0^+$, there hold
	\begin{equation}
		\label{eqn:lb-R}
		\text{$\widehat{\boldsymbol{R}}_h \to \boldsymbol{I}$ in $L^1(S;\rtt)$,}
	\end{equation}
	\begin{equation}
	\label{eqn:lb-G}
	\text{$\widehat{\boldsymbol{G}}_h\coloneqq h^{-\beta/2} (\widehat{\boldsymbol{R}}_h^\top \widehat{\boldsymbol{F}}_h -\boldsymbol{I})\wk \widehat{\boldsymbol{G}}$ in $L^2(\Omega;\rtt)$,}
	\end{equation}
	and, for almost every $\boldsymbol{x}\in \Omega$, we have
	\begin{equation}
	\label{eqn:lb-G-structure}
	\text{$\widehat{\boldsymbol{G}}''(\boldsymbol{x})=\sym \nabla' \widehat{\boldsymbol{u}}(\boldsymbol{x}')+((\nabla')^2\widehat{v}(\boldsymbol{x}'))\,x_3$.}
	\end{equation}
	Moreover, suppose that there exists  $\widehat{\boldsymbol{\chi}}\in L^2(\rt;\rt)$ such that the following convergences hold, as $h \to 0^+$:
	\begin{align}
	\label{eqn:lb-eta}
	\text{$\widehat{\boldsymbol{\mu}}_h\coloneqq \mathcal{M}_h(\widehat{\boldsymbol{q}}_h)$}&\text{$\to \chi_\Omega \widehat{\boldsymbol{\zeta}}$ in $L^2(\rt;\rt)$,}\\
	\label{eqn:lb-H}
	\text{$\widehat{\boldsymbol{N}}_h\coloneqq \mathcal{N}_h(\widehat{\boldsymbol{q}}_h)$}&\text{$\wk \chi_\Omega (\nabla'\widehat{\boldsymbol{\zeta}}|\widehat{\boldsymbol{\chi}})$ in $L^2(\rt;\rtt)$,}\\
	\label{eqn:lb-lagrangian}
	\text{$\widehat{\boldsymbol{z}}_h\coloneqq \mathcal{Z}_h(\widehat{\boldsymbol{q}}_h)$}&\text{$\to \widehat{\boldsymbol{\zeta}}$ in $L^1(\Omega;\rt)$.}
	\end{align}
	Then, the following inequality holds:
	\begin{equation}
	\label{eqn:lb-liminf-inequality}
	E_0(\widehat{\boldsymbol{q}}_0)\leq \liminf_{h \to 0^+} E_h(\widehat{\boldsymbol{q}}_h).
	\end{equation}
\end{proposition}
\begin{proof}
	It is sufficient to prove the following:
	\begin{align}
	\label{eqn:lb-elastic}
	E_0^{\mathrm{el}}(\widehat{\boldsymbol{q}}_0)&\leq \liminf_{h \to 0^+} E_h^{\mathrm{el}}(\widehat{\boldsymbol{q}}_h),\\
	\label{eqn:lb-exchange}
	E_0^{\mathrm{exc}}(\widehat{\boldsymbol{q}}_0)&\leq \liminf_{h \to 0^+} E_h^{\mathrm{exc}}(\widehat{\boldsymbol{q}}_h).
	\end{align}
	Indeed, thanks to \eqref{eqn:lb-eta}, the limit in \eqref{eqn:magnetostatic-limit} holds by Proposition \ref{prop:magnetostatic}. Thus, combining \eqref{eqn:lb-elastic}--\eqref{eqn:lb-exchange} with \eqref{eqn:magnetostatic-limit}, we obtain \eqref{eqn:lb-liminf-inequality}.
	
	Recalling  \eqref{eqn:cm-H-norm}, claim \eqref{eqn:lb-exchange}  follows immediately from \eqref{eqn:lb-H}. Indeed, 
	by lower semicontinuity, we have
	\begin{equation*}
	\liminf_{h \to 0^+} \int_{\rt} |\widehat{\boldsymbol{N}}_h|^2\,\d\boldsymbol{x}
		\geq \int_\Omega |\nabla'\widehat{\boldsymbol{\zeta}}|^2\,\d\boldsymbol{x}+\int_\Omega |\widehat{\boldsymbol{\chi}}|^2\,\d\boldsymbol{x}
		\geq \int_S |\nabla'\widehat{\boldsymbol{\zeta}}|^2\,\d\boldsymbol{x}'.
	\end{equation*} 
	
	We thus focus on \eqref{eqn:lb-elastic}.
	This is proved similarly to \cite[Corollary 2]{friesecke.james.mueller2}. 
	Recall \eqref{eqn:density-W_h}--\eqref{eqn:prestrain-inverse}.  
	For simplicity, set $\widehat{\boldsymbol{\lambda}}_h\coloneqq \widehat{\boldsymbol{m}}_h \circ \widehat{\boldsymbol{y}}_h$and $\widehat{\boldsymbol{L}}_h\coloneqq \boldsymbol{L}_h(\widehat{\boldsymbol{R}}_h \hspace*{-4pt}{}^\top\widehat{\boldsymbol{z}}_h)$.
	By \eqref{eqn:lb-R} and \eqref{eqn:lb-lagrangian}, recalling \eqref{eqn:ch-limit} and \eqref{eqn:saturation}, we have
	\begin{equation}
	\label{eqn:lb-L}
	\text{$\widehat{\boldsymbol{K}}_h\coloneqq h^{-\beta/2}(\boldsymbol{I}-\widehat{\boldsymbol{L}}_h^{-1})\to c_0\widehat{\boldsymbol{\zeta}}\otimes \widehat{\boldsymbol{\zeta}}$ in $L^1(\Omega;\rtt)$.}
	\end{equation}
	Define $A_h \coloneqq \{|\widehat{\boldsymbol{G}}_h|\leq h^{-\beta/4}\}$, so that, by \eqref{eqn:lb-G}, $\chi_{A_h} \to 1$ in $L^1(\Omega)$.
	Note that, on $A_h$, there holds
	\begin{equation}
	\label{eqn:lb-nonlinear}
	\widehat{\boldsymbol{L}}_h^{-1}\widehat{\boldsymbol{R}}_h^\top\widehat{\boldsymbol{F}}_h=\boldsymbol{I}+h^{\beta/2}(\widehat{\boldsymbol{G}}_h-\widehat{\boldsymbol{K}}_h)+O(h^{3/4\beta}).
	\end{equation}
	Recalling \eqref{eqn:Taylor-Phi}, for $h \ll1$ we have
	$|\widehat{\boldsymbol{L}}_h^{-1}\widehat{\boldsymbol{R}}_h^\top\widehat{\boldsymbol{F}}_h-\boldsymbol{I}|\ll 1$ on $A_h$. Then, exploiting \eqref{eqn:frame-indifference} and \eqref{eqn:lb-nonlinear}, we write
	\begin{equation*}
	\begin{split}
	\int_\Omega W_h(\widehat{\boldsymbol{F}}_h,\widehat{\boldsymbol{\lambda}}_h)\,\d\boldsymbol{x} &=\int_\Omega W_h(\widehat{\boldsymbol{R}}_h^\top \widehat{\boldsymbol{F}}_h,\widehat{\boldsymbol{R}}_h^\top \widehat{\boldsymbol{\lambda}}_h)\,\d\boldsymbol{x}=\int_\Omega \Phi(\widehat{\boldsymbol{L}}_h^{-1}\widehat{\boldsymbol{R}}_h^\top \widehat{\boldsymbol{F}}_h)\,\d\boldsymbol{x}
	 \geq \int_\Omega \chi_{A_h} \Phi(\widehat{\boldsymbol{L}}_h^{-1}\widehat{\boldsymbol{R}}_h^\top \widehat{\boldsymbol{F}}_h)\,\d\boldsymbol{x}\\
	&= \int_\Omega \chi_{A_h}  \Phi \left ( \boldsymbol{I}+h^{\beta/2}(\widehat{\boldsymbol{G}}_h-\widehat{\boldsymbol{K}}_h)+O(h^{3\beta/4}) \right)\,\d\boldsymbol{x}\\
	&=\frac{h^\beta}{2}\int_\Omega Q \left(\chi_{A_h} (\widehat{\boldsymbol{G}}_h-\widehat{\boldsymbol{K}}_h) +O(h^{3\beta/4})\right)\,\d\boldsymbol{x}\\
	&+\int_\Omega \chi_{A_h} \omega\left(h^{\beta/2}(\widehat{\boldsymbol{G}}_h-\widehat{\boldsymbol{K}}_h)+ O(h^{3\beta/4}) \right)\,\d\boldsymbol{x}. 
	\end{split}
	\end{equation*}
	Thus
	\begin{equation}
	\begin{split}
	\label{eqn:lb-el}
	E_h^{\mathrm{el}}(\widehat{\boldsymbol{q}}_h)&\geq \frac{1}{2}\int_\Omega Q \left(\chi_{A_h} (\sym\,\widehat{\boldsymbol{G}}_h-\widehat{\boldsymbol{K}}_h) +O(h^{3\beta/4})\right)\,\d\boldsymbol{x}\\
	&+\frac{1}{h^\beta}\int_\Omega \chi_{A_h} \omega\left(h^{\beta/2}(\sym\,\widehat{\boldsymbol{G}}_h-\widehat{\boldsymbol{K}}_h)+ O(h^{3\beta/4}) \right)\,\d\boldsymbol{x}. 
	\end{split}
	\end{equation}
	To provide a lower bound for the first integral on the right-hand side of \eqref{eqn:lb-el}, we exploit the convexity of $Q$. By \eqref{eqn:lb-G} and \eqref{eqn:lb-L}, we have $\chi_{A_h}\widehat{\boldsymbol{G}}_h \wk \widehat{\boldsymbol{G}}$ in $L^2(\Omega;\rtt)$ and $\chi_{A_h}\widehat{\boldsymbol{K}}_h \to c_0 \widehat{\boldsymbol{\zeta}}\otimes\widehat{\boldsymbol{\zeta}}$ in $L^1(\Omega;\rtt)$, as $h \to 0^+$. Thus, by lower semicontinuity, we get
	\begin{equation*}
	\begin{split}
	\liminf_{h \to 0^+}\int_\Omega Q \left(\chi_{A_h} (\widehat{\boldsymbol{G}}_h-\widehat{\boldsymbol{K}}_h) +O(h^{3\beta/4})\right)\,\d\boldsymbol{x} &\geq \int_\Omega Q(\widehat{\boldsymbol{G}}-c_0 \widehat{\boldsymbol{\zeta}}\otimes\widehat{\boldsymbol{\zeta}})\,\d\boldsymbol{x}\\
	&\geq \int_\Omega Q_{\mathrm{red}}(\widehat{\boldsymbol{G}}''-c_0 \widehat{\boldsymbol{\zeta}}'\otimes\widehat{\boldsymbol{\zeta}}')\,\d\boldsymbol{x},
	\end{split}
	\end{equation*}
	where the last inequality is justified by the definition of $Q_{\rm red}$ in \eqref{eqn:Q-red}.
	Given  \eqref{eqn:lb-G-structure}, this proves \eqref{eqn:lb-elastic} once we show that the second integral on the right-hand side of \eqref{eqn:lb-el} goes to zero, as $h \to 0^+$. To prove this, for every $s>0$, we set
	\begin{equation*}
		\overline{\omega}(s)\coloneqq \sup \left \{  \frac{|\omega(\boldsymbol{F})|}{|\boldsymbol{F}|^2}:\:\boldsymbol{F}\in\rtt, \:\:|\boldsymbol{F}|\leq s \right \}.
	\end{equation*}
	By definition, the function $\overline{\omega}$ is decreasing and satisfies $\overline{\omega}(s)\to0$, as $s\to 0^+$. Then, recalling the definition of $A_h$, we have
	\begin{equation*}
	\begin{split}
	\frac{1}{h^\beta}&\int_\Omega \chi_{A_h} \left | \omega\left(h^{\beta/2}(\widehat{\boldsymbol{G}}_h-\widehat{\boldsymbol{K}}_h)+ O(h^{3\beta/4}) \right)\right |\,\d\boldsymbol{x} \\
	&\leq  \int_\Omega \chi_{A_h}\overline{\omega} \left( h^{\beta/2}|\widehat{\boldsymbol{G}}_h-\widehat{\boldsymbol{K}}_h|+ O(h^{3\beta/4}) \right)\frac{\left | h^{\beta/2}(\widehat{\boldsymbol{G}}_h-\widehat{\boldsymbol{K}}_h)+ O(h^{3\beta/4}) \right|^2}{h^\beta}\,\d\boldsymbol{x}\\
	&\leq \overline{\omega}(O(h^{\beta/4})) \int_\Omega \left | \widehat{\boldsymbol{G}}_h-\widehat{\boldsymbol{K}}_h+O(h^{\beta/4})\right|^2\,\d\boldsymbol{x},
	\end{split}
	\end{equation*}
	where the right-hand side goes to zero, as $h \to 0^+$. 
\end{proof}

\subsubsection{Optimality of the lower bound}

The following result ensures the existence of recovery sequences.

\begin{proposition}[Recovery sequences]
\label{prop:recovery}
Let $\widehat{\boldsymbol{q}}_0=(\widehat{\boldsymbol{u}},\widehat{v},\widehat{\boldsymbol{\zeta}})\in\mathcal{Q}_0$. Then, there exists $(\widehat{\boldsymbol{q}}_h)$ with $\widehat{\boldsymbol{q}}_h=(\widehat{\boldsymbol{y}}_h,\widehat{\boldsymbol{m}}_h)\in\mathcal{Q}_h$  such that the following convergences hold, as $h \to 0^+$:
\begin{align}
    \label{eqn:recovery-horizontal}
    \text{$\widehat{\boldsymbol{u}}_h\coloneqq \mathcal{U}_h(\widehat{\boldsymbol{q}}_h)$}&\text{$\to\widehat{\boldsymbol{u}}$ in $W^{1,2}(S;\R^3)$,}\\
    \label{eqn:recovery-vertical}
    \text{$\widehat{v}_h\coloneqq \mathcal{V}_h(\widehat{\boldsymbol{q}}_h)$}&\text{$\to \widehat{v}$ in $W^{2,2}(S)$,}\\
    \label{eqn:recovery-moment}
    \text{$\widehat{\boldsymbol{w}}_h\coloneqq \mathcal{W}_h(\widehat{\boldsymbol{q}}_h)$}&\text{$\to -\frac{1}{12}\left(\begin{array}{c}
         \nabla'\widehat{v}  \\
         0
    \end{array} \right)$ in $W^{1,2}(S;\rt)$,}\\
    \label{eqn:recovery-eta}
    \text{$\widehat{\boldsymbol{\mu}}_h\coloneqq \mathcal{M}_h(\widehat{\boldsymbol{q}}_h)$}&\text{$\to \chi_\Omega \widehat{\boldsymbol{\zeta}}$ in $L^2(\R^3;\R^3)$,}\\
    \label{eqn:recovery-H}
    \text{$\widehat{\boldsymbol{N}}_h\coloneqq\mathcal{N}_h(\widehat{\boldsymbol{q}}_h)$}&\text{$\to \chi_\Omega (\nabla'\widehat{\boldsymbol{\zeta}}\vert\boldsymbol{0})$ in $L^2(\R^3;\rtt)$,}\\
    \label{eqn:recovery-composition}
    \text{$\widehat{\boldsymbol{m}}_h \circ \widehat{\boldsymbol{y}}_h$}&\text{$\to\widehat{\boldsymbol{\zeta}}$ in $L^1(\Omega;\R^3)$,}\\
    \label{eqn:recovery-lagrangian-magnetization}
    \text{$\widehat{\boldsymbol{z}}_h\coloneqq \mathcal{Z}_h(\widehat{\boldsymbol{q}}_h)$}&\text{$\to \widehat{\boldsymbol{\zeta}}$ in $L^1(\Omega;\rtt)$.}
\end{align}
Moreover, the following equality holds:
\begin{equation}
    \label{eqn:recovery-limit}
    E_0(\widehat{\boldsymbol{q}}_0)=\lim_{h\to 0^+}E_h(\widehat{\boldsymbol{q}}_h).
\end{equation}
\end{proposition}
\begin{proof}
For convenience of the reader, the proof is subdivided into three steps.

\textbf{Step 1 (Construction of the recovery sequence).} First, we additionally assume that $\widehat{\boldsymbol{u}}\in C^2_{\rm c}(S;\R^2)$, $\widehat{v}\in C^3_{\rm c}(S)$ and $\widehat{\boldsymbol{\zeta}}\in C^1(\closure{S};\S^2)$. Let $\widehat{\boldsymbol{a}},\widehat{\boldsymbol{b}}\in C^2_{\rm c}(S;\rt)$.
Deformations of the recovery sequence are constructed according to the classical ansatz for the linearized von
K\'{a}rm\'{a}n regime \cite{friesecke.james.mueller2}. Thus, for every $h>0$, we define
\begin{equation*}
\begin{split}
	\widehat{\boldsymbol{y}}_h\coloneqq  \,&\boldsymbol{\pi}_h + h^{\beta/2} \left ( \begin{matrix}  \widehat{\boldsymbol{u}}  \\ 0\end{matrix}   \right )+ h^{\beta/2-1} \left ( \begin{matrix}  \boldsymbol{0}'  \\ \widehat{v}\end{matrix}   \right )-h^{\beta/2} x_3 \left ( \begin{matrix}  \nabla'\widehat{v} \\ 0\end{matrix}   \right )+ 2 h^{\beta/2+1} x_3 \widehat{\boldsymbol{a}}+h^{\beta/2+1} x_3^2 \widehat{\boldsymbol{b}}.
\end{split}
\end{equation*}
Observe that $\widehat{\boldsymbol{y}}_h\in C^2(\closure{\Omega};\rt)$ and there holds
\begin{equation}
\label{eqn:recovery-key}
\|\widehat{\boldsymbol{y}}_h-\boldsymbol{\pi}_h\|_{C^0(\closure{\Omega};\rt)}\leq C h^{\beta/2-1}.
\end{equation}
We compute
\begin{equation}
\label{eqn:recovery-gradient}
\begin{split}
\widehat{\boldsymbol{F}}_h&= \boldsymbol{I}+ h^{\beta/2}\left( \renewcommand\arraystretch{1.3} \begin{array}{@{}c|c@{}}   \nabla'\widehat{\boldsymbol{u}}  & \boldsymbol{0}'\\ \hline  (\boldsymbol{0}')^\top & 0 \end{array} \right)+  h^{\beta/2-1} \left( \renewcommand\arraystretch{1.3} \begin{array}{@{}c|c@{}}   \mathbf{O}''  & -\nabla' \widehat{v}\\ \hline  (\nabla'v)^\top & 0 \end{array} \right)-h^{\beta/2} x_3 \left( \renewcommand\arraystretch{1.3} \begin{array}{@{}c|c@{}}   (\nabla')^2 \widehat{v}  & \boldsymbol{0}'\\ \hline  (\boldsymbol{0}')^\top & 0 \end{array} \right) \\
&+  2h^{\beta/2}\, \widehat{\boldsymbol{a}} \otimes \boldsymbol{e}_3+2h^{\beta/2} \,x_3 \,\widehat{\boldsymbol{b}} \otimes \boldsymbol{e}_3 + O(h^{\beta/2+1}),
\end{split}   
\end{equation}
so that we immediately deduce
\begin{equation}
\label{eqn:recovery-gradient-uniform}
	\|\widehat{\boldsymbol{F}}_h-\boldsymbol{I}\|_{C^0(\closure{\Omega};\rtt)}\leq C h^{\beta/2-1}.
\end{equation}
Recall the identity $(\boldsymbol{I}+\boldsymbol{F})^\top (\boldsymbol{I}+\boldsymbol{F})=\boldsymbol{I}+2\,\sym\,\boldsymbol{F} + \boldsymbol{F}^\top \boldsymbol{F}$ for every $\boldsymbol{F}\in \rtt$. Thanks to the assumption $\beta>6$, we have
\begin{equation}
\label{eqn:recovery-gradient-2}
\sqrt{\Big (\widehat{\boldsymbol{F}}_h\Big)^\top \widehat{\boldsymbol{F}}_h}= \boldsymbol{I}+h^{\beta/2}(\widehat{\boldsymbol{A}}+x_3\,\widehat{\boldsymbol{B}})+O(h^{\beta/2+1}),
\end{equation}
where
\begin{equation*}
\begin{split}
	&\widehat{\boldsymbol{A}} \coloneqq  \left( \renewcommand\arraystretch{1.3} \begin{array}{@{}c|c@{}}    \sym \nabla'\widehat{\boldsymbol{u}}  & \boldsymbol{0}'\\ \hline  (\boldsymbol{0}')^\top & 0 \end{array} \right) + \widehat{\boldsymbol{a}} \otimes \boldsymbol{e}_3 + \boldsymbol{e}_3 \otimes \widehat{\boldsymbol{a}},\\
	&\widehat{\boldsymbol{B}} \coloneqq - \left( \renewcommand\arraystretch{1.3} \begin{array}{@{}c|c@{}}   (\nabla')^2 \widehat{v}  & \boldsymbol{0}'\\ \hline  (\boldsymbol{0}')^\top & 0 \end{array} \right) + \widehat{\boldsymbol{b}} \otimes \boldsymbol{e}_3 + \boldsymbol{e}_3 \otimes \widehat{\boldsymbol{b}}.
\end{split}
\end{equation*}
Using the identity $\det (\boldsymbol{I}+\boldsymbol{F})=1+\tr\boldsymbol{F}+\tr\,(\cof\boldsymbol{F})+\det \boldsymbol{F}$ for every $\boldsymbol{F}\in\rtt$, we obtain 
\begin{equation}
\label{eqn:recovery-determinant-bdd-bis}
\det \widehat{\boldsymbol{F}}_h=1+O(h^{\beta/2}), \qquad \nabla(\det \widehat{\boldsymbol{F}}_h)=O(h^{\beta/2}).
\end{equation}
This ensures that 
\begin{equation}
\label{eqn:recovery-determinant-bdd}
	\text{$1/2<\det \widehat{\boldsymbol{F}}_h<3/2$  in $\Omega$,}
\end{equation}
at least for $h \ll1$. In particular, {\MMM by the inverse function theorem,  $\widehat{\boldsymbol{y}}_h$ is a local diffeomorphism of class $C^2$  with $\nabla \widehat{\boldsymbol{y}}_h=(\nabla \widehat{\boldsymbol{y}}_h)^{-1}\circ \widehat{\boldsymbol{y}}_h^{-1}$ in $\Omega^{\widehat{\boldsymbol{y}}_h}$.} 

We claim that the map $\widehat{\boldsymbol{y}}_h$ is injective. To see this, we argue as in \cite[Theorem 5.5-1, Claim (b)]{ciarlet}. 
Fix $\boldsymbol{x},\widehat{\boldsymbol{x}}\in \Omega$ with $\boldsymbol{x}\neq \widehat{\boldsymbol{x}}$. By \cite[Exercise 1.9]{ciarlet}, there exists a finite number of distinct points $\widetilde{\boldsymbol{x}}_1,\dots,\widetilde{\boldsymbol{x}}_m\in\Omega$ with $\widetilde{\boldsymbol{x}}_1=\boldsymbol{x}$ and $\widetilde{\boldsymbol{x}}_m=\widehat{\boldsymbol{x}}$ such that each segment connecting $\widetilde{\boldsymbol{x}}_i$ to $\widetilde{\boldsymbol{x}}_{i+1}$ is entirely contained in $\Omega$ and $\sum_{i=1}^m |\widetilde{\boldsymbol{x}}_{i+1}-\widetilde{\boldsymbol{x}}_i| \leq C |\boldsymbol{x}-\widehat{\boldsymbol{x}}|$ for some constant $C>0$ depending on $\Omega$ only. Then, applying the mean value theorem in combination with \eqref{eqn:recovery-gradient-uniform}, we estimate
\begin{equation*}
\begin{split}
|\widehat{\boldsymbol{y}}_h(\boldsymbol{x})-\widehat{\boldsymbol{y}}_h(\widehat{\boldsymbol{x}})-(\boldsymbol{\pi}_h(\boldsymbol{x})-\boldsymbol{\pi}_h(\widehat{\boldsymbol{x}}))|&
\leq \sum_{i=1}^m|(\widehat{\boldsymbol{y}}_h(\widetilde{\boldsymbol{x}}_{i+1})-\boldsymbol{\pi}_h(\widetilde{\boldsymbol{x}}_{i+1}))-(\widehat{\boldsymbol{y}}_h(\widetilde{\boldsymbol{x}}_{i})-\boldsymbol{\pi}_h(\widetilde{\boldsymbol{x}}_{i}))|\\
&\leq 	\|\nabla\widehat{\boldsymbol{y}}_h-\nabla\boldsymbol{\pi}_h\|_{C^0(\closure{\Omega};\rtt)}\,\sum_{i=1}^m|\widetilde{\boldsymbol{x}}_{i+1}-\widetilde{\boldsymbol{x}}_i|\\
&\leq Ch^{\beta/2-1} |\boldsymbol{x}-\widehat{\boldsymbol{x}}|\leq Ch^{\beta/2-2} |\boldsymbol{\pi}_h(\boldsymbol{x})-\boldsymbol{\pi}_h(\widehat{\boldsymbol{x}})|.
\end{split}
\end{equation*}
Since $\beta>4$, for $h\ll 1$, we obtain
\begin{equation*}
|\widehat{\boldsymbol{y}}_h(\boldsymbol{x})-\widehat{\boldsymbol{y}}_h(\widehat{\boldsymbol{x}})-(\boldsymbol{\pi}_h(\boldsymbol{x})-\boldsymbol{\pi}_h(\widehat{\boldsymbol{x}}))|<|\boldsymbol{\pi}_h(\boldsymbol{x})-\boldsymbol{\pi}_h(\widehat{\boldsymbol{x}})|.
\end{equation*}
As $\boldsymbol{\pi}_h(\boldsymbol{x})\neq\boldsymbol{\pi}_h(\widehat{\boldsymbol{x}})$, we necessarily have $\widehat{\boldsymbol{y}}_h(\boldsymbol{x})\neq\widehat{\boldsymbol{y}}_h(\widehat{\boldsymbol{x}})$, which proves the claim. 

{\MMM
By the invariance of domain theorem, the map $\widehat{\boldsymbol{y}}_h$ is a homeomorphism. {\MMM Hence, we actually have $\widehat{\boldsymbol{y}}_h^{-1}\in  C^2(\Omega^{\widehat{\boldsymbol{y}}_h};\rt)$. Moreover,}
 setting $\boldsymbol{I}_h\coloneqq \nabla\boldsymbol{\pi}_h$, there holds
\begin{equation}
\label{eqn:recovery-gradient-inverse}
\begin{split}
\boldsymbol{I}_h\,\nabla\widehat{\boldsymbol{y}}_h^{-1}&=\boldsymbol{I}_h\,(\nabla\widehat{\boldsymbol{y}}_h)^{-1}\circ \widehat{\boldsymbol{y}}_h^{-1}=\left(\nabla\widehat{\boldsymbol{y}}_h\circ \widehat{\boldsymbol{y}}_h^{-1}\,\boldsymbol{I}_h^{-1} \right)^{-1}=\left(\widehat{\boldsymbol{F}}_h\circ \widehat{\boldsymbol{y}}_h^{-1}\right)^{-1}=\widehat{\boldsymbol{F}}_h^{-1}\circ \widehat{\boldsymbol{y}}_h^{-1}.
\end{split}
\end{equation} 
From \eqref{eqn:recovery-gradient}, by computing the  inverse of $\widehat{\boldsymbol{F}}_h$ using Neumann series, 
we obtain
\begin{equation}
	\label{eqn:recovery-neumann}
	\widehat{\boldsymbol{F}}_h^{-1}=\boldsymbol{I}+O(h^{\beta/2-1}),	
\end{equation}
which, in view of \eqref{eqn:recovery-gradient-inverse}, yields
\begin{equation}
    \label{eqn:gradient-inverse}
    \|\nabla \widehat{\boldsymbol{y}}_h^{-1}\|_{C^0(\closure{\Omega};\R^3)}\leq Ch^{-1}.
\end{equation}

Now,  define $\widehat{\boldsymbol{m}}_h\colon \Omega^{\widehat{\boldsymbol{y}}_h} \to \rt$ by setting
\begin{equation*}
	\widehat{\boldsymbol{m}}_h\coloneqq \frac{\mathstrut \widehat{\boldsymbol{\zeta}}\circ \widehat{\boldsymbol{y}}_h^{-1}}{\det \nabla_h\widehat{\boldsymbol{y}}_h\circ \widehat{\boldsymbol{y}}_h^{-1}}. 
\end{equation*}
Observe that $\widehat{\boldsymbol{m}}_h\in W^{1,2}(\Omega^{\widehat{\boldsymbol{y}}_h};\rt)$ thanks to the regularity of $\widehat{\boldsymbol{y}}_h^{-1}$ and to \eqref{eqn:recovery-determinant-bdd}. Moreover, there holds $\widehat{\boldsymbol{m}}_h\circ \widehat{\boldsymbol{y}}_h\,\det \nabla_h\widehat{\boldsymbol{y}}_h=\widehat{\boldsymbol{\zeta}}$ in $\Omega$, so that $\widehat{\boldsymbol{q}}_h\coloneqq (\widehat{\boldsymbol{y}}_h,\widehat{\boldsymbol{m}}_h)\in\mathcal{Q}_h$.

The convergences in \eqref{eqn:recovery-vertical}--\eqref{eqn:recovery-moment} follow by direct computation, while \eqref{eqn:recovery-lagrangian-magnetization} is trivial. From \eqref{eqn:recovery-determinant-bdd} and \eqref{eqn:recovery-neumann}, we also have
\begin{equation}
\label{eqn:recovery-uu}
	\text{$\det \widehat{\boldsymbol{F}}_h\to 1$ in $C^0(\closure{\Omega})$, \qquad $\frac{\mathstrut 1}{\det \widehat{\boldsymbol{F}}}_h\to 1$  in $C^0(\closure{\Omega})$,}
\end{equation}
and
\begin{equation}
	\label{eqn:recovery-uu2}
	\text{$\nabla(\det \widehat{\boldsymbol{F}}_h)\to \boldsymbol{0}$ in $C^0(\closure{\Omega};\rt)$, \qquad  $\widehat{\boldsymbol{F}}_h^{-1}\to\boldsymbol{I}$ in $C^0(\closure{\Omega};\rtt)$.}
\end{equation}
Thus \eqref{eqn:recovery-composition} follows. 

We prove \eqref{eqn:recovery-eta}. We have
\begin{equation*}
	\widehat{\boldsymbol{\mu}}_h=\chi_{\boldsymbol{\pi}_h^{-1}(\Omega^{\widehat{\boldsymbol{y}}_h})}\frac{\widehat{\boldsymbol{\zeta}}\circ \widehat{{\boldsymbol{y}}}_h^{-1}\circ \boldsymbol{\pi}_h}{\det \widehat{{\boldsymbol{F}}}_h\circ  \widehat{{\boldsymbol{y}}}_h^{-1}\circ \boldsymbol{\pi}_h}.
\end{equation*}
First, as $\beta>4$, from \eqref{eqn:recovery-key} we establish \eqref{eqn:cm-subcylinder}--\eqref{eqn:cm-supercylinder} by arguing as in Step 1 of the proof of Proposition \ref{prop:cm}. 
Let $\boldsymbol{x}\in\rt \setminus \closure{\Omega}$, so  that $\boldsymbol{x}\in \rt \setminus \Omega^{-\eta}_\ell$ for some $\eta>0$ and $\ell >1$. Then, \eqref{eqn:cm-supercylinder} entails that  $\boldsymbol{x}\in \rt \setminus \boldsymbol{\pi}_h^{-1}(\Omega^{\widehat{\boldsymbol{y}}_h})$ for $h\leq \underline{h}(\eta,\ell)$. Instead, let $\boldsymbol{x}\in\Omega$ and let $\eta>0$ and $0<\vartheta<1$ be such that  $\boldsymbol{x}\in\Omega^\eta_\vartheta$. By \eqref{eqn:cm-subcylinder}, we have $\boldsymbol{x}\in \boldsymbol{\pi}_h^{-1}(\Omega^{\widehat{\boldsymbol{y}}_h})$for $h\leq \overline{h}(\eta,\vartheta)$ and, in turn, $\boldsymbol{x}_h\coloneqq \widehat{\boldsymbol{y}}_h^{-1}(\boldsymbol{\pi}_h(\boldsymbol{x}))\in\Omega$. 
From \eqref{eqn:recovery-key}, we see that
\begin{equation*}
	\|\boldsymbol{\pi}_h^{-1}\circ \widehat{\boldsymbol{y}}_h-\boldsymbol{id}\,\|_{C^0(\closure{\Omega};\rt)}\leq \frac{1}{h}  \| \widehat{\boldsymbol{y}}_h-\boldsymbol{\pi}_h\|_{C^0(\closure{\Omega};\rt)}\leq  Ch^{\beta/2-2}.
\end{equation*}
Let $\varepsilon>0$ be arbitrary. For $h\ll 1$, the right-hand side of the previous equation is smaller than $\varepsilon$, so that
\begin{equation*}
	|\boldsymbol{x}-\widehat{\boldsymbol{y}}_h^{-1}(\boldsymbol{\pi}_h(\boldsymbol{x}))|=|\boldsymbol{\pi}_h^{-1}(\widehat{\boldsymbol{y}}_h(\boldsymbol{x}_h))-\boldsymbol{x}_h|\leq \|\boldsymbol{\pi}_h^{-1}\circ \widehat{\boldsymbol{y}}_h-\boldsymbol{id}\,\|_{C^0(\closure{\Omega};\rt)}<\varepsilon.
\end{equation*}
This shows that $\widehat{\boldsymbol{y}}_h^{-1}(\boldsymbol{\pi}_h(\boldsymbol{x}))\to\boldsymbol{x}$, as $h \to 0^+$. Since $\widehat{\boldsymbol{\zeta}}$ is continuous and \eqref{eqn:recovery-uu} holds, we deduce that $\widehat{\boldsymbol{\mu}}_h\to \widehat{\boldsymbol{\mu}}$ almost everywhere in $\rt$, where $\widehat{\boldsymbol{\mu}}\coloneqq \chi_\Omega \widehat{\boldsymbol{\zeta}}$. As $(\widehat{\boldsymbol{\mu}}_h)$ is uniformly bounded and, for $h \ll 1$, supported in a compact set containing $\Omega$  by \eqref{eqn:cm-supercylinder}, the convergence in \eqref{eqn:recovery-eta} follows by applying the dominated convergence theorem.

To prove \eqref{eqn:recovery-H}, first we compute
\begin{equation*}
	\nabla\widehat{\boldsymbol{m}}_h=\frac{\nabla(\widehat{\boldsymbol{\zeta}}\circ\widehat{\boldsymbol{y}}_h^{-1})}{\det \widehat{{\boldsymbol{F}}}_h\circ\widehat{\boldsymbol{y}}_h^{-1}}-\frac{\widehat{\boldsymbol{\zeta}}\circ\widehat{\boldsymbol{y}}_h^{-1}\,\nabla(\det \widehat{{\boldsymbol{F}}}_h\circ\widehat{\boldsymbol{y}}_h^{-1})}{(\det \widehat{{\boldsymbol{F}}}_h\circ\widehat{\boldsymbol{y}}_h^{-1})^2}.
\end{equation*}
Since $\widehat{\boldsymbol{\zeta}}$ is a function of the plane variable only, there holds $\widehat{\boldsymbol{\zeta}}=\widehat{\boldsymbol{\zeta}}\circ\boldsymbol{\pi}_h\restr{\Omega}$. Thus
\begin{equation*}
	\begin{split}
		\nabla(\widehat{\boldsymbol{\zeta}}\circ\widehat{\boldsymbol{y}}_h^{-1})&=\nabla(\widehat{\boldsymbol{\zeta}}\circ \boldsymbol{\pi}_h\circ\widehat{\boldsymbol{y}}_h^{-1})=\nabla\widehat{\boldsymbol{\zeta}}\circ \boldsymbol{\pi}_h\circ \widehat{\boldsymbol{y}}_h^{-1} \boldsymbol{I}_h\,\nabla\widehat{\boldsymbol{y}}_h^{-1}= \nabla\widehat{\boldsymbol{\zeta}}\circ \widehat{\boldsymbol{y}}_h^{-1}\widehat{\boldsymbol{F}}_h^{-1}\circ \widehat{\boldsymbol{y}}_h^{-1},
	\end{split}
\end{equation*}
where, in the last line, we used \eqref{eqn:recovery-gradient-inverse}. Also, we have
\begin{equation*}
\begin{split}
		\left(\nabla(\det \widehat{\boldsymbol{F}}_h\circ\widehat{\boldsymbol{y}}_h^{-1})\right)^\top&=\nabla(\det \widehat{\boldsymbol{F}}_h)\circ\widehat{\boldsymbol{y}}_h^{-1}\nabla\widehat{\boldsymbol{y}}_h^{-1}=\nabla(\det \widehat{\boldsymbol{F}}_h)\circ\widehat{\boldsymbol{y}}_h^{-1} \boldsymbol{I}_h^{-1}\widehat{\boldsymbol{F}}_h^{-1}\circ \widehat{\boldsymbol{y}}_h^{-1},
\end{split}
\end{equation*}
where we employed again \eqref{eqn:recovery-gradient-inverse}.
Thus
\begin{equation*}
	\nabla\widehat{\boldsymbol{m}}_h=\frac{\nabla\widehat{\boldsymbol{\zeta}}\circ\widehat{\boldsymbol{y}}_h^{-1}\widehat{\boldsymbol{F}}_h^{-1}\circ \widehat{\boldsymbol{y}}_h^{-1}}{\det \widehat{{\boldsymbol{F}}}_h\circ\widehat{\boldsymbol{y}}_h^{-1}}-\frac{\widehat{\boldsymbol{\zeta}}\circ\widehat{\boldsymbol{y}}_h^{-1}\,\nabla(\det \widehat{{\boldsymbol{F}}}_h)\circ\widehat{\boldsymbol{y}}_h^{-1}\boldsymbol{I}_h^{-1}\widehat{\boldsymbol{F}}_h^{-1}\circ \widehat{\boldsymbol{y}}_h^{-1}}{(\det \widehat{{\boldsymbol{F}}}_h\circ\widehat{\boldsymbol{y}}_h^{-1})^2},
\end{equation*}
so that
\begin{equation*}
	\begin{split}
		\widehat{\boldsymbol{N}}_h&=\chi_{\boldsymbol{\pi}_h^{-1}(\Omega^{\widehat{\boldsymbol{y}}_h})} \frac{\nabla\widehat{\boldsymbol{\zeta}}\circ\widehat{\boldsymbol{y}}_h^{-1}\circ\boldsymbol{\pi}_h\,\widehat{\boldsymbol{F}}_h^{-1}\circ \widehat{\boldsymbol{y}}_h^{-1}\circ\boldsymbol{\pi}_h}{\det \widehat{{\boldsymbol{F}}}_h\circ\widehat{\boldsymbol{y}}_h^{-1}\circ\boldsymbol{\pi}_h}\\
		&-\chi_{\boldsymbol{\pi}_h^{-1}(\Omega^{\widehat{\boldsymbol{y}}_h})} \frac{\widehat{\boldsymbol{\zeta}}\circ\widehat{\boldsymbol{y}}_h^{-1}\circ\boldsymbol{\pi}_h\,\nabla(\det \widehat{{\boldsymbol{F}}}_h)\circ\widehat{\boldsymbol{y}}_h^{-1}\circ\boldsymbol{\pi}_h\,\boldsymbol{I}_h^{-1}\,\widehat{\boldsymbol{F}}_h^{-1}\circ \widehat{\boldsymbol{y}}_h^{-1}\circ\boldsymbol{\pi}_h}{(\det \widehat{{\boldsymbol{F}}}_h\circ\widehat{\boldsymbol{y}}_h^{-1}\circ\boldsymbol{\pi}_h)^2}.
	\end{split}
\end{equation*}
With arguments similar to the ones previously employed, exploiting the continuity of $\nabla'\widehat{\boldsymbol{\zeta}}$ together with \eqref{eqn:cm-subcylinder}--\eqref{eqn:cm-supercylinder}, \eqref{eqn:recovery-determinant-bdd-bis}--\eqref{eqn:recovery-determinant-bdd} and \eqref{eqn:recovery-uu}--\eqref{eqn:recovery-uu2}, we show that $\widehat{\boldsymbol{N}}_h\to \widehat{\boldsymbol{N}}$ almost everywhere, where $\widehat{\boldsymbol{N}}\coloneqq \chi_\Omega(\nabla'\widehat{\boldsymbol{\zeta}}\vert\boldsymbol{0})$. As both maps are bounded and supported in a compact set containing $\Omega$, claim \eqref{eqn:recovery-H} follows by applying the dominated convergence theorem. }

\textbf{Step 2 (Attainment of the lower bound).}
We now compute the limit of the magnetoelastic energy along the recovery sequence. Recalling \eqref{eqn:cm-H-norm} and Proposition \ref{prop:magnetostatic}, from \eqref{eqn:recovery-eta}--\eqref{eqn:recovery-H} we obtain
\begin{equation}
\label{eqn:recovery-exc-mag}
	\lim_{h\to 0^+} E_h^{\rm exc}(\widehat{{\boldsymbol{q}}}_h)=E_0^{\rm exc}(\widehat{\boldsymbol{q}}_0), \qquad \lim_{h\to 0^+} E_h^{\rm mag}(\widehat{{\boldsymbol{q}}}_h)=E_0^{\rm mag}(\widehat{\boldsymbol{q}}_0).
\end{equation}
We show the convergence of the elastic energy. By the polar decomposition theorem, there exists $(\widehat{\boldsymbol{P}}_h)\subset C^1(\closure{S};SO(3))$ such that $\widehat{\boldsymbol{F}}_h=\widehat{\boldsymbol{P}}_h \left(\widehat{\boldsymbol{F}}_h^\top\widehat{\boldsymbol{F}}_h\right)^{1/2}$. Thanks to \eqref{eqn:recovery-gradient-uniform}--\eqref{eqn:recovery-gradient-2}, passing to the limit, as $h\to 0^+$, in the previous identity, we see that $\widehat{\boldsymbol{P}}_h\to\boldsymbol{I}$ uniformly in $\Omega$. For convenience, set $\widehat{\boldsymbol{L}}_h\coloneqq \boldsymbol{L}_h(\widehat{{\boldsymbol{P}}}_h^\top \widehat{\boldsymbol{z}}_h)$. By \eqref{eqn:ch-limit} and \eqref{eqn:prestrain-inverse}, there holds
\begin{equation}
	\label{eqn:recovery-L}
	\text{$\widehat{{\boldsymbol{K}}}_h\coloneqq h^{-\beta/2}(\boldsymbol{I}-\widehat{\boldsymbol{L}}_h^{-1})\to c_0\,\widehat{\boldsymbol{\zeta}}\otimes \widehat{\boldsymbol{\zeta}}$  in $L^1(\Omega;\rtt)$.}
\end{equation}
Recalling \eqref{eqn:frame-indifference-Phi} and \eqref{eqn:frame-indifference-prestrain}, we write
\begin{equation}
\label{eqn:fir}
	\begin{split}
			W_h(\widehat{\boldsymbol{F}}_h,\widehat{\boldsymbol{\lambda}}_h)&=\Phi \left(\boldsymbol{L}_h(\widehat{\boldsymbol{z}}_h)^{-1}\widehat{{\boldsymbol{F}}}_h  \right)=\Phi \left(\boldsymbol{L}_h(\widehat{\boldsymbol{z}}_h)^{-1}\widehat{{\boldsymbol{P}}}_h\left(\widehat{\boldsymbol{F}}_h^\top\widehat{\boldsymbol{F}}_h\right)^{1/2}  \right)\\
			&=\Phi \left(\widehat{{\boldsymbol{P}}}_h^\top\boldsymbol{L}_h(\widehat{\boldsymbol{z}}_h)^{-1}\widehat{{\boldsymbol{P}}}_h\left(\widehat{\boldsymbol{F}}_h^\top\widehat{\boldsymbol{F}}_h\right)^{1/2}  \right)=\Phi \left(\widehat{{\boldsymbol{L}}}_h^{-1}\left(\widehat{\boldsymbol{F}}_h^\top\widehat{\boldsymbol{F}}_h\right)^{1/2}  \right).
	\end{split}
\end{equation}
Thus, given \eqref{eqn:prestrain-inverse}, we have
\begin{equation}
\label{eqn:fir2}
	\widehat{{\boldsymbol{L}}}_h^{-1}\left(\widehat{\boldsymbol{F}}_h^\top\widehat{\boldsymbol{F}}_h\right)^{1/2}=\boldsymbol{I}+h^{\beta/2}\left(\widehat{\boldsymbol{A}}-\widehat{\boldsymbol{K}}_h+x_3\widehat{\boldsymbol{B}}\right)+O(h^{\beta/2+1}).
\end{equation}
Using \eqref{eqn:Taylor-Phi} and \eqref{eqn:fir}--\eqref{eqn:fir2}, we compute
\begin{equation*}
\begin{split}
E_h^{\mathrm{el}}(\widehat{\boldsymbol{q}}_h)&=\frac{1}{h^\beta}\int_\Omega \Phi \left (\widehat{{\boldsymbol{L}}}_h^{-1}\left(\widehat{\boldsymbol{F}}_h^\top\widehat{\boldsymbol{F}}_h\right)^{1/2} \right)\,\d\boldsymbol{x}\\
&=\frac{1}{2}\int_\Omega Q \left (\widehat{\boldsymbol{A}}-\widehat{\boldsymbol{K}}_h+x_3 \widehat{\boldsymbol{B}}+O(h)\right)\,\d\boldsymbol{x}\\
&+ \frac{1}{h^\beta} \int_\Omega \omega\left (h^{\beta/2}\left (\widehat{\boldsymbol{A}}-\widehat{\boldsymbol{K}}_h+x_3 \widehat{\boldsymbol{B}}\right)+O(h^{\beta/2+1})\right)\,\d\boldsymbol{x}.
\end{split}
\end{equation*}
Then, thanks to \eqref{eqn:recovery-L}, applying the dominated convergence theorem, we obtain
\begin{equation}
\label{eqn:recovery-el}
\lim_{h\to 0^+} E_h^{\mathrm{el}}(\widehat{\boldsymbol{q}}_h)=\frac{1}{2}\int_\Omega Q (\widehat{\boldsymbol{A}}-c_0\widehat{\boldsymbol{
\zeta}}\otimes\widehat{{\boldsymbol{\zeta}}}+x_3\widehat{\boldsymbol{B}})\,\d\boldsymbol{x}
\end{equation}
Therefore, combining \eqref{eqn:recovery-exc-mag} and \eqref{eqn:recovery-el}, we have
\begin{equation*}
	\lim_{h\to 0^+} E_h(\widehat{\boldsymbol{q}}_h)=\frac{1}{2}\int_S Q (\widehat{\boldsymbol{A}}-c_0\widehat{\boldsymbol{\zeta}}\otimes\widehat{\boldsymbol{\zeta}})\,\d\boldsymbol{x}+\frac{1}{24}\int_S Q (\widehat{\boldsymbol{B}})\,\d\boldsymbol{x}+\int_S |\nabla'\widehat{\boldsymbol{\zeta}}|^2\,\d\boldsymbol{x}'+\int_S |\widehat{\zeta}^{\,3}|^2\,\d\boldsymbol{x}'.
\end{equation*}

\textbf{Step 3 (Diagonal argument).} To conclude the proof, we employ a standard diagonal argument. 
By the definition of $Q_{\mathrm{red}}$ in \eqref{eqn:Q-red}, there exist $\widehat{\boldsymbol{a}},\widehat{\boldsymbol{b}}\colon S \to \R^3$ such that
\begin{equation}
\label{eqn:recovery-choice}
    Q_{\mathrm{red}}(\sym\,\nabla'\widehat{\boldsymbol{u}}-\widehat{\boldsymbol{\zeta}}\,'\otimes\widehat{\boldsymbol{\zeta}}\,')=Q(\widehat{\boldsymbol{A}}-\widehat{\boldsymbol{\zeta}}\otimes\widehat{\boldsymbol{\zeta}}),\qquad Q_{\mathrm{red}}((\nabla')^2\widehat{v})=Q(\widehat{\boldsymbol{B}}),
\end{equation}
where we set
\begin{equation*}
	\widehat{\boldsymbol{A}} \coloneqq  \left( \renewcommand\arraystretch{1.2} \begin{array}{@{}c|c@{}}    \sym \nabla'\widehat{\boldsymbol{u}}  & \boldsymbol{0}'\\ \hline  (\boldsymbol{0}')^\top & 0 \end{array} \right) + \widehat{\boldsymbol{a}} \otimes \boldsymbol{e}_3 + \boldsymbol{e}_3 \otimes \widehat{\boldsymbol{a}}, \qquad \widehat{\boldsymbol{B}} \coloneqq - \left( \renewcommand\arraystretch{1.2} \begin{array}{@{}c|c@{}}   (\nabla')^2 \widehat{v}  & \boldsymbol{0}'\\ \hline  (\boldsymbol{0}')^\top & 0 \end{array} \right) + \widehat{\boldsymbol{b}} \otimes \boldsymbol{e}_3 + \boldsymbol{e}_3 \otimes \widehat{\boldsymbol{b}}.
\end{equation*}
In particular, thanks to \eqref{eqn:Q-coe}, we have $\widehat{\boldsymbol{a}},\widehat{\boldsymbol{b}}\in L^2(S;\R^3)$. 
Let $(\widehat{\boldsymbol{u}}_j)\subset C^2_{\rm c}(S;\R^2)$, $(\widehat{v}_j)\subset C^3_{\rm c}(\closure{S})$ and $(\widehat{\boldsymbol{a}}_j),(\widehat{\boldsymbol{b}}_j)\subset C^2_{\rm c}(\closure{S};\R^3)$ be such that the following convergences hold, as $j \to \infty$:
\begin{align}
    \label{eqn:recovery-uv-convergence}
    &\text{$\widehat{\boldsymbol{u}}_j \to \widehat{\boldsymbol{u}}$ in $W^{1,2}(S;\R^2)$,}\qquad \text{$\widehat{v}_j \to \widehat{v}$ in $W^{2,2}(S)$,}\\
    \label{eqn:recovery-ab-convergence}
    &\text{$\widehat{\boldsymbol{a}}_j \to \widehat{\boldsymbol{a}}$ in $L^2(S;\R^3)$,}\qquad \hspace{4mm} \text{$\widehat{\boldsymbol{b}}_j \to \widehat{\boldsymbol{b}}$ in $L^2(S;\R^3)$.} 
\end{align}
If we set
\begin{equation*}
	\widehat{\boldsymbol{A}}_j \coloneqq  \left( \renewcommand\arraystretch{1.2} \begin{array}{@{}c|c@{}}    \sym \nabla'\widehat{\boldsymbol{u}}_j  & \boldsymbol{0}'\\ \hline  (\boldsymbol{0}')^\top & 0 \end{array} \right) + \widehat{\boldsymbol{a}}_j \otimes \boldsymbol{e}_3 + \boldsymbol{e}_3 \otimes \widehat{\boldsymbol{a}}_j, \qquad \widehat{\boldsymbol{B}}_j \coloneqq - \left( \renewcommand\arraystretch{1.2} \begin{array}{@{}c|c@{}}   (\nabla')^2 \widehat{v}_j  & \boldsymbol{0}'\\ \hline  (\boldsymbol{0}')^\top & 0 \end{array} \right) + \widehat{\boldsymbol{b}}_j \otimes \boldsymbol{e}_3 + \boldsymbol{e}_3 \otimes \widehat{\boldsymbol{b}}_j,
\end{equation*}
then, by \eqref{eqn:recovery-uv-convergence}--\eqref{eqn:recovery-ab-convergence}, we immediately have
\begin{equation}
\label{eqn:Lambda-Theta}
    \text{$\widehat{\boldsymbol{A}}_j \to \widehat{\boldsymbol{A}}$ in $L^2(S;\rtt)$, \quad $\widehat{\boldsymbol{B}}_j \to \widehat{\boldsymbol{B}}$ in $L^2(S;\rtt)$,}
\end{equation}
as $j \to \infty$. Also, thanks to \cite[Theorem 2.1]{hajlasz}, there exists  a sequence $(\widehat{\boldsymbol{\zeta}}_j)\subset C^1(\closure{S};\S^2)$ such that
\begin{equation}
\label{eqn:approx-sphere-valued}
    \text{$\widehat{\boldsymbol{\zeta}}_j \to \widehat{\boldsymbol{\zeta}}$ in $W^{1,2}(S;\R^3)$,}
\end{equation}
as $j \to \infty$, and, in turn, there holds
\begin{equation}
\label{eqn:bb}
    \text{$\widehat{\boldsymbol{\zeta}}_j\otimes \widehat{\boldsymbol{\zeta}}_j \to \widehat{\boldsymbol{\zeta}} \otimes \widehat{\boldsymbol{\zeta}}$ in $L^1(S;\rtt)$,}
\end{equation}
as $j \to \infty$.

By Step 1, for every fixed $j\in \N$,  there exists a sequence $(\widehat{\boldsymbol{q}}_h^{(j)})$ with $\widehat{\boldsymbol{q}}_h^{(j)}=(\widehat{\boldsymbol{y}}_h^{(j)},\widehat{\boldsymbol{m}}_h^{(j)})\in\mathcal{Q}_h$ such that the following convergences hold, as $h\to 0^+$:
\begin{align}
\label{eqn:recovery-horizontal-j}
\text{$\widehat{\boldsymbol{u}}_h^{(j)}\coloneqq \mathcal{U}_h(\widehat{\boldsymbol{q}}_h^{(j)})$}&\text{$\to\widehat{\boldsymbol{u}}_j$ in $W^{1,2}(S;\rt)$;}\\
\label{eqn:recovery-vertical-j}
\text{$\widehat{v}_h^{(j)}\coloneqq \mathcal{V}_h(\widehat{\boldsymbol{q}}_h^{(j)})$}&\text{$\to \widehat{v}_j$ in $W^{2,2}(S)$;}\\
\label{eqn:recovery-moment-j}
\text{$\widehat{\boldsymbol{w}}_h^{(j)}\coloneqq \mathcal{W}_h(\widehat{\boldsymbol{q}}_h^{(j)})$}&\text{$\to -\frac{1}{12}\left ( \begin{matrix}  \nabla'\widehat{v}_j \\ 0\end{matrix}   \right )$ in $W^{1,2}(S;\rt)$,}\\
\label{eqn:recovery-eta-j}
\text{$\widehat{\boldsymbol{\mu}}_h^{(j)}\coloneqq \mathcal{M}_h(\widehat{\boldsymbol{q}}_h^{(j)})$}&\text{$\to \chi_\Omega \widehat{\boldsymbol{\zeta}}_j$ in $L^2(\rt;\rt)$,}\\
\label{eqn:recovery-H-j}
\text{$\widehat{\boldsymbol{N}}_h^{(j)}\coloneqq \mathcal{N}_h(\widehat{\boldsymbol{q}}_h^{(j)})$}&\text{$\to \chi_\Omega (\nabla'\widehat{\boldsymbol{\zeta}}_j|\boldsymbol{0})$ in $L^2(\rt;\rtt)$,}\\
\label{eqn:recovery-composition-j}
\text{$\widehat{\boldsymbol{m}}_h^{(j)} \circ \widehat{\boldsymbol{y}}_h^{(j)}$}&\text{$\to\widehat{\boldsymbol{\zeta}}_j$ in $L^1(\Omega;\rt)$,}\\
\label{eqn:recovery-lagrangian-magnetization-j}
\text{$\widehat{\boldsymbol{z}}_h^{(j)}\coloneqq \mathcal{Z}_h(\widehat{\boldsymbol{q}}_h^{(j)})$}&\text{$\to \widehat{\boldsymbol{\zeta}}_j$ in $L^1(\Omega;\rt)$.}
\end{align}
Moreover, we also have
\begin{equation}
\label{eqn:recovery-ernregy-j}
\begin{split}
\lim_{h\to 0^+} E_h(\widehat{\boldsymbol{q}}_h^{(j)})=\frac{1}{2}&\int_S Q (\widehat{\boldsymbol{A}}_j-c_0\widehat{\boldsymbol{\zeta}}_j\otimes\widehat{\boldsymbol{\zeta}}_j)\,\d\boldsymbol{x}+\frac{1}{24}\int_S Q (\widehat{\boldsymbol{B}}_j)\,\d\boldsymbol{x}\\
+&\int_S |\nabla'\widehat{\boldsymbol{\zeta}}_j|^2\,\d\boldsymbol{x}'+\frac{1}{2}\int_S |\widehat{\zeta}^{\,3}_j|^2\,\d\boldsymbol{x}'.
\end{split}
\end{equation}
In view of \eqref{eqn:recovery-uv-convergence}--\eqref{eqn:recovery-ernregy-j}, we select a subsequence $(h_j)$ such that, setting $\widehat{\boldsymbol{q}}_{h_j}\coloneqq(\widehat{\boldsymbol{y}}_{h_j},\widehat{\boldsymbol{m}}_{h_j})\in\mathcal{Q}_{h_j}$ with $\widehat{\boldsymbol{y}}_{h_j}\coloneqq \widehat{\boldsymbol{y}}_{h_j}^{(j)}$ and $\widehat{\boldsymbol{m}}_{h_j}\coloneqq \widehat{\boldsymbol{m}}_{h_j}^{(j)}$, the convergences in \eqref{eqn:recovery-horizontal}--\eqref{eqn:recovery-lagrangian-magnetization} hold for $h=h_j$, as $j \to \infty$, as well as
\begin{equation}
\label{eqn:diag3}
\begin{split}
	\lim_{j \to \infty} E_{h_j}^{\mathrm{el}}(\widehat{\boldsymbol{q}}_{h_j})=\frac{1}{2}&\int_S Q(\widehat{\boldsymbol{A}}-\widehat{\boldsymbol{\zeta}}\otimes \widehat{\boldsymbol{\zeta}})\,\d\boldsymbol{x}'+\frac{1}{24}\int_S Q(\widehat{\boldsymbol{B}})\,\d\boldsymbol{x}'+\int_S |\nabla'\,\widehat{\boldsymbol{\zeta}}|^2\,\d\boldsymbol{x}'+\frac{1}{2}\int_S |\widehat{\zeta}^{\,3}|^2\,\d\boldsymbol{x}'.
\end{split}
\end{equation}
As the right-hand side of \eqref{eqn:diag3} equals $E_0^{\mathrm{el}}(\widehat{\boldsymbol{q}}_0)$ by \eqref{eqn:recovery-choice}, this establishes \eqref{eqn:recovery-limit} for the subsequence indexed by $(h_j)$.
\end{proof}

{\MMM At this point, our first main result has been basically proved.}

\begin{proof}[Proof of Theorem \ref{thm:gamma-conv}]
{\MMM Since claim (ii) has already been proved in Proposition \ref{prop:recovery}, we only have to show claim (i). Suppose \eqref{eqn:be}.
By  applying Proposition \ref{prop:compactness}, we find $\boldsymbol{q}_0=(\boldsymbol{u},v,\boldsymbol{\zeta})\in\mathcal{Q}_0$ such that, up to subsequences, the convergences in   \eqref{eqn:compactness-clamped-horizontal}--\eqref{eqn:compactness-clamped-G} as well as \eqref{eqn:compactness-clamped-G-structure} hold true. In particular, this shows \eqref{eqn:gamma-horizontal}--\eqref{eqn:gamma-lagrangian}. Then, given \eqref{eqn:compactness-clamped-eta}--\eqref{eqn:compactness-clamped-H}, \eqref{eqn:compactness-clamped-lagrangian} and \eqref{eqn:compactness-clamped-R}--\eqref{eqn:compactness-clamped-G-structure}, by Proposition \ref{prop:lb}, we establish \eqref{eqn:gamma-liminf}.}
\end{proof}

\subsection{Convergence of almost minimizers}
\label{subsec:cam}
Henceforth, we consider applied loads determined by body forces and 
by  external magnetic fields. For simplicity, applied surface forces are not considered but these can be easily included in the analysis. According to the assumption of {dead loads}, the work of mechanical forces is described by a Lagrangian term. Conversely, the energy contribution corresponding to the external magnetic field, usually called {Zeeman energy}, is of Eulerian type. 

Given $h>0$, let $\boldsymbol{f}_h\in L^2(S;\R^2)$, $\boldsymbol{g}_h \in L^2(S)$ and $\boldsymbol{h}_h\in L^2(\R^3;\R^3)$ represent a horizontal force, a vertical force and an external magnetic field, respectively. The work of applied loads is determined by the functional $L_h \colon \mathcal{Q}_h\to\R$ defined by
\begin{equation*}
	L_h(\boldsymbol{q})\coloneqq \frac{1}{h^\beta} \int_{\Omega} \boldsymbol{f}_h \cdot (\boldsymbol{y}'-\boldsymbol{x}')\,\d \boldsymbol{x}+\frac{1}{h^\beta}\int_\Omega g_h\,y^3\,\d\boldsymbol{x}+\frac{1}{h}\int_{\Omega^{\boldsymbol{y}}} \boldsymbol{h}_h \cdot \boldsymbol{m}\,\d\boldsymbol{\xi},
\end{equation*}
where $\boldsymbol{q}=(\boldsymbol{y},\boldsymbol{m})$. Thus, the total energy  $F_h\colon \mathcal{Q}_h\to \R$ reads
\begin{equation}
    \label{eqn:energy-F_h}
	F_h(\boldsymbol{q})\coloneqq E_h(\boldsymbol{q})-L_h(\boldsymbol{q}),
\end{equation}
where we recall \eqref{eqn:energy-E_h}. 
Regarding the asymptotic behavior of the applied loads, we assume 
that there exist $\boldsymbol{f}_0\in L^2(S;\R^2)$, $g_0\in L^2(S)$ and $\boldsymbol{h}_0\in L^2(\R^2;\R^3)$ such that
the following convergences hold, as $h \to 0^+$:
\begin{align}
    \label{eqn:load-f}
    \text{${h^{-\beta/2}}\boldsymbol{f}_h$}&\text{$\wk \boldsymbol{f}_0$ in $L^2(S;\R^2)$,}\\
    \label{eqn:load-g}
    \text{${h^{-\beta/2-1}}g_h$}&\text{$\wk g_0$ in $L^2(S)$,}\\
    \label{eqn:load-h}
    \text{$\boldsymbol{h}_h\circ \boldsymbol{\pi}_h$}&\text{$\wk \chi_I \boldsymbol{h}_0$ in $L^2(\R^3;\R^3)$.}
\end{align}
We stress that the limiting magnetic field $\boldsymbol{h}$ is assumed not to depend on the variable $x_3$.

The limiting total energy $F_0\colon \mathcal{Q}_0\to\R$ is defined as 
\begin{equation}
    \label{eqn:energy-F}
	F_0(\boldsymbol{q}_0)\coloneqq E_0(\boldsymbol{q}_0)-L_0(\boldsymbol{q}_0).
\end{equation}
Here, the functional $L_0 \colon \mathcal{Q}_0\to \R$ is given by
\begin{equation*}
	L_0(\boldsymbol{q}_0)\coloneqq \int_S \boldsymbol{f}_0\cdot \boldsymbol{u}\,\d\boldsymbol{x}'+\int_S g_0\,v\,\d\boldsymbol{x}'+\int_S \boldsymbol{h}_0\cdot \boldsymbol{\zeta}\,\d \boldsymbol{x}',
\end{equation*}
where $\boldsymbol{q}_0=(\boldsymbol{u},v,\boldsymbol{\zeta})$, and we recall \eqref{eqn:energy-E_0}.
Observe that the limiting total energy is purely Lagrangian.

Our second main result claims that sequences of almost minimizers of  $(F_h)$ in \eqref{eqn:energy-F_h} converge, as $h \to 0^+$, to minimizers of the energy $F_0$ in \eqref{eqn:energy-F}. 

\begin{theorem}[Convergence of almost minimizers]
	\label{thm:convergence-almost-minimizers}
	Assume $p>3$ and $\beta>6\vee p$. Suppose that the elastic energy density $W_h$ has the form in \eqref{eqn:density-W_h}, where the function $\Phi$ satisfies \eqref{eqn:normalization-Phi}--\eqref{eqn:regularity-Phi}, and that the applied loads satisfy \eqref{eqn:load-f}-\eqref{eqn:load-h}. Let  $(\boldsymbol{q}_h)$ with $\boldsymbol{q}_h=(\boldsymbol{y}_h,\boldsymbol{m}_h)\in\mathcal{Q}_h$ for every $h>0$ be such that 
	\begin{equation}
	    \label{eqn:almost-minimizer}
		\lim_{h \to 0^+} \left \{F_h(\boldsymbol{q}_h)-\inf_{\mathcal{Q}_h}F_h  \right\}=0.
	\end{equation}
	Then, there exist $\boldsymbol{q}_0=(\boldsymbol{u},v,\boldsymbol{\zeta})\in \mathcal{Q}_0$ such that, up to subsequences, the following convergences hold, as $h \to 0^+$:
	\begin{align}
	    \label{eqn:cam-horizontal}
		\text{$\boldsymbol{u}_h\coloneqq \mathcal{U}_h(\boldsymbol{q}_h)$}&\text{$\wk \boldsymbol{u}$ in $W^{1,2}(S;\R^3)$,}\\
		\label{eqn:cam-vertical}
		\text{$v_h\coloneqq \mathcal{V}_h(\boldsymbol{q}_h)$}&\text{$\to v$ in $W^{1,2}(S)$,}\\
		\label{eqn:cam-lagrangian}
		\text{$\boldsymbol{z}_h\coloneqq\mathcal{Z}_h(\boldsymbol{q}_h)$}&\text{$\to \boldsymbol{\zeta}$ in $L^1(\Omega;\R^3)$.}
	\end{align}
	Moreover, $\boldsymbol{q}_0\in \mathcal{Q}_0$ is a minimizer of $F_0$ in $\mathcal{Q}_0$. 
\end{theorem}

We mention that the weak convergence in \eqref{eqn:cam-horizontal} can be improved to strong convergence by arguing as in \cite[Subsection 7.2]{friesecke.james.mueller2}. {\MMM Also, more general boundary conditions as in Remark \ref{rem:dirichlet-bc} can be imposed.}

\begin{remark}[Existence of minimizers for the reduced model]
\label{rem:existence-min-reduced}
The existence of minimizer of $F_0$ in $\mathcal{Q}_0$ is a consequence of Theorem \ref{thm:convergence-almost-minimizers}. However, under our assumptions, this can be established directly.
First, note that the functional $F_0$ is lower semicontinuous with respect to the product weak topology in view of the convexity of $Q_{\mathrm{red}}$. Thus, in order to apply the Direct Method, one only has to show that the functional $F_0$ is coercive on $\mathcal{Q}_0$. 
This is done by exploiting the positive definiteness of  $Q_{\mathrm{red}}$ on symmetric matrices in \eqref{eqn:Q-red-coe} and applying Korn and Poincaré inequalities in view of the homogeneous boundary conditions in \eqref{eqn:class-Q0}.
\end{remark}

The major difficulty in proving Theorem \ref{thm:convergence-almost-minimizers} is to deduce the equi-boundedness of the elastic energy starting from the equi-boundedness of the total energy. This is accomplished by arguing by contradiction, similarly to \cite[Theorem 4]{lecumberry.mueller}.

{\MMM
\begin{lemma}[Energy scaling]
	\label{lem:energy-scaling}
	Let $M>0$ and let $(\boldsymbol{q}_h)$ with $\boldsymbol{q}_h\in{\mathcal{Q}}_h$ be such that
	\begin{equation}
	\label{eqn:bddf}
	\sup_{h>0} F_h(\boldsymbol{q}_h)\leq M.
	\end{equation}
	Then, there exist $C(M)>0$ and $\overline{h}>0$, where the former does not depend on  $(\boldsymbol{q}_h)$, such that
	\begin{equation}
	\label{eqn:bdde}
	\sup_{h\leq \overline{h}} E_h(\boldsymbol{q}_h)\leq C(M).
	\end{equation}
\end{lemma}
\begin{proof}
	We introduce some further notation. Given $h>0$ and $\boldsymbol{q}=(\boldsymbol{y},\boldsymbol{m})\in{\mathcal{Q}}_h$, we set
	\begin{equation}
	\label{eqn:Ih}
	I_h(\boldsymbol{q})\coloneqq \int_\Omega W_h(\nabla_h\boldsymbol{y},\boldsymbol{m}\circ\boldsymbol{y})\,\d \boldsymbol{x},
	\end{equation}
	\begin{equation*}
	J_h(\boldsymbol{q})=I_h(\boldsymbol{q})-\int_{\Omega} \boldsymbol{f}_h \cdot (\boldsymbol{y}'-\boldsymbol{x}')\,\d \boldsymbol{x}+\int_\Omega g_h\,y^3\,\d\boldsymbol{x}.
	\end{equation*}
	For convenience of the reader, the proof is subdivided into three steps.
	
	\textbf{Step 1 (Preliminary estimates).}
	First, we check that
	\begin{equation}
		\label{eqn:Ih-est}
		I_h(\boldsymbol{q}_h)\leq C(M)h^{\beta/2}.
	\end{equation}
	 Let $\boldsymbol{q}_h=(\boldsymbol{y}_h,\boldsymbol{m}_h)$ and, for simplicity, set
	\begin{equation*}
	\boldsymbol{F}_h\coloneqq \nabla_h \boldsymbol{y}_h, \quad \boldsymbol{\lambda}_h\coloneqq \boldsymbol{m}_h \circ \boldsymbol{y}_h, \quad \boldsymbol{L}_h\coloneqq \boldsymbol{L}_h(\boldsymbol{\lambda}_h(\det\boldsymbol{F}_h)),\quad \boldsymbol{\Xi}_h\coloneqq \boldsymbol{L}_h^{-1}\boldsymbol{F}_h.
	\end{equation*}
	As in \eqref{eqn:rigidity}, for every $h>0$, we have
	\begin{equation}
		\label{eqn:rigidity-bis}
		r_h\coloneqq\mathcal{R}_h(\boldsymbol{y}_h)\leq C h^\beta+I_h(\boldsymbol{q}_h),
	\end{equation}
	which yields 
	\begin{equation*}
	\begin{split}
		\|\boldsymbol{F}_h\|_{L^2(\Omega;\rtt)}&\leq C \left(\|\dist(\boldsymbol{F}_h;SO(3))\|_{L^2(\Omega)}+1\right)\\
		&\leq C \left(\sqrt{r_h}+1    \right) \\
		&\leq C \left (\sqrt{I_h(\boldsymbol{q}_h)}+1\right ).
	\end{split}
	\end{equation*}
	Then, using  the Poincar\'{e} inequality with trace term, we obtain
	\begin{equation}
	\label{eqn:en-scal-estimate}
	\|\boldsymbol{y}_h-\boldsymbol{\pi}_h\|_{L^2(\Omega;\rt)}\leq C  \|\boldsymbol{F}_h-\boldsymbol{I}\|_{L^2(\Omega;\rtt)}\leq C \left( \|\boldsymbol{F}_h\|_{L^2(\Omega;\rtt)}+1\right)\leq C \left(\sqrt{I_h(\boldsymbol{q}_h)} +1 \right).
	\end{equation}
	Given \eqref{eqn:load-f} and \eqref{eqn:en-scal-estimate}, applying the H\"{o}lder inequality, we estimate
	\begin{equation}
	\label{eqn:scaling-f}
	\begin{split}
	\left | \frac{1}{h^\beta}\int_\Omega \boldsymbol{f}_h\cdot (\boldsymbol{y}_h'-\boldsymbol{x}')\,\d\boldsymbol{x} \right |
	&\leq \frac{1}{h^\beta}\|\boldsymbol{f}_h\|_{L^2(S;\R^2)}\,\|\boldsymbol{y}_h'-\boldsymbol{x}'\|_{L^2(\Omega;\R^2)}\\
	&\leq \frac{C}{h^{\beta/2}} \,\|\boldsymbol{y}_h'-\boldsymbol{x}'\|_{L^2(\Omega;\R^2)}\\
	&\leq \frac{C}{h^{\beta/2}}\left (\sqrt{I_h(\boldsymbol{q}_h)}+1\right).
	\end{split}
	\end{equation}
	Similarly, exploiting \eqref{eqn:load-g}, we obtain
	\begin{equation}
	\label{eqn:scaling-g}
	\begin{split}
	\left |\frac{1}{h^\beta}\int_\Omega g_h\cdot y_h^3\,\d\boldsymbol{x} \right | &\leq \frac{1}{h^\beta}\|g_h\|_{L^2(S)}\,\|y_h^3\|_{L^2(\Omega)}\\
	&\leq \frac{C}{h^{\beta/2-1}} \left( \|y_h^3-h x_3\|_{L^2(\Omega)}+1\right) \\  
	&\leq \frac{C}{h^{\beta/2-1}}\left (\sqrt{I_h(\boldsymbol{q}_h)}+1\right).
	\end{split}
	\end{equation}
	Set $\boldsymbol{\mu}_h\coloneqq \mathcal{M}_h(\boldsymbol{q}_h)$. Recalling \eqref{eqn:load-h} and using  the change-of-variable formula, we estimate
	\begin{equation}
	\label{eqn:scaling-h}
	\begin{split}
	\left |\frac{1}{h}\int_{\Omega^{\boldsymbol{y}_h}}\boldsymbol{h}_h\cdot \boldsymbol{m}_h\,\d\boldsymbol{\xi} \right |&= \left |\int_{\rt} \boldsymbol{h}_h\circ \boldsymbol{\pi}_h\cdot \boldsymbol{\mu}_h\,\d\boldsymbol{x}\right |\\
	&\leq \|\boldsymbol{h}_h\circ \boldsymbol{\pi}_h\|_{L^2(\rt;\rt)}\,\|\boldsymbol{\mu}_h\|_{L^2(\rt;\rt)}\\
	&\leq C\left(\sqrt{I_h(\boldsymbol{q}_h)}+1\right).
	\end{split}
	\end{equation}
	Indeed, as in \eqref{eqn:a2}--\eqref{eqn:a3}, there holds
	\begin{equation*}
	\int_{\rt}|\boldsymbol{\mu}_h|^2\,\d\boldsymbol{x}=\int_\Omega \frac{1}{\det\boldsymbol{F}_h}\d\boldsymbol{x}\leq C \int_\Omega \frac{1}{|\det\boldsymbol{F}_h|^a}\d\boldsymbol{x}\leq C(I_h(\boldsymbol{q}_h)+1).
	\end{equation*}
	
	The combination of \eqref{eqn:scaling-f}--\eqref{eqn:scaling-h} yields
	\begin{equation}
		\label{eqn:load-bdd}
		\left |L_h(\boldsymbol{q}_h)\right | \leq \frac{C}{h^{\beta/2}}\left (\sqrt{I_h(\boldsymbol{q}_h)}+1\right),
	\end{equation}
	so that, applying the Young inequality, we obtain
	\begin{equation*}
	\label{eqn:tot-en-below}
	F_h(\boldsymbol{q}_h)\geq \frac{1}{h^\beta} I_h(\boldsymbol{q}_h)-\frac{C}{h^{\beta/2}} \left (\sqrt{I_h(\boldsymbol{q}_h)}+1\right)\geq \frac{C}{h^\beta}  I_h(\boldsymbol{q}_h) - \frac{C'}{h^{\beta/2}}.
	\end{equation*}
	This entails \eqref{eqn:Ih-est} thanks to \eqref{eqn:bddf}. From \eqref{eqn:rigidity-bis}, we immediately get
	\begin{equation}
		\label{eqn:R_h}
		r_h\leq C(M) h^{\beta/2}.
	\end{equation}
	Also, we obtain
	\begin{equation}
		\label{eqn:Jh-est}
		J_h(\boldsymbol{q}_h)\leq C(M)h^\beta,
	\end{equation}
	since
	\begin{equation*}
	h^{-\beta}J_h(\boldsymbol{q}_h)\leq F_h(\boldsymbol{q}_h)+\frac{1}{h}\int_{\Omega^{\boldsymbol{y}_h}}\boldsymbol{h}_h\cdot\boldsymbol{m}_h\,\d\boldsymbol{\xi}\leq M+ C\left(\sqrt{I_h(\boldsymbol{q}_h)}+1\right)\leq C(M),
	\end{equation*}
	thanks to \eqref{eqn:bddf}, \eqref{eqn:Ih-est} and \eqref{eqn:scaling-h}.
	
	\vspace{2mm}
	
	\textbf{Step 2 (Improved estimates).}
	We claim that $I_h(\boldsymbol{q}_h)$ is actually of order $h^\beta$. By contradiction, suppose  that 
	\begin{equation}
	\label{eqn:contradiction}
	\limsup_{h\to 0^+} \frac{1}{h^\beta}I_h(\boldsymbol{q}_h)=+\infty.
	\end{equation}
	Note that $r_h/h^2\to 0$, as $h\to 0^+$ by \eqref{eqn:R_h} since $\beta>4$.
	By means of Lemma \ref{lem:ar}, we find $(\boldsymbol{R}_h)\subset W^{1,p}(S;SO(3))$ and $(\boldsymbol{Q}_h)\subset SO(3)$ such that the following estimates hold:
	\begin{align*}
	\|{\boldsymbol{F}}_h-{\boldsymbol{R}}_h\|_{L^2(\Omega;\rtt)}&\leq C \sqrt{r_h}, \qquad   \hspace{4,5mm}\|\nabla'{\boldsymbol{R}}_h\|_{L^2(S;\R^{3 \times 3 \times 3})}
	\leq C h^{-1}\sqrt{r_h},\\
	\|{\boldsymbol{R}}_h-{\boldsymbol{Q}}_h\|_{L^2(S;\rtt)}&\leq C h^{-1}\sqrt{r_h}, \quad  \|{\boldsymbol{F}}_h-{\boldsymbol{Q}}_h\|_{L^2(\Omega;\rtt)}\leq C h^{-1}\sqrt{r_h}.
	\end{align*}
	Let $e_h= I_h(\boldsymbol{q}_h)$. Given \eqref{eqn:contradiction}, up to subsequence, we have 
	\begin{equation}
	\label{eqn:contradiction-bis}
		\lim_{h \to 0^+} \frac{h^\beta}{e_h}=0,
	\end{equation}
	which, together with \eqref{eqn:rigidity-bis}, yields
	\begin{equation}
	\label{eqn:contradiction-tris}
		\limsup_{h \to 0^+} \frac{r_h}{e_h}\leq 1.
	\end{equation}
	Also, applying Lemma \ref{lem:clamped}, we deduce 
	\begin{equation*}
		|\boldsymbol{Q}_h-\boldsymbol{I}|\leq C   \frac{\sqrt{r_h}}{h}.
	\end{equation*}
	Now, define $\boldsymbol{U}_h\colon S \to \R^2$ and $V_h\colon S \to \R$ by setting
	\begin{align*}
	\boldsymbol{U}_h(\boldsymbol{x}')&\coloneqq\frac{h^2}{e_h}\wedge \frac{1}{\sqrt{e_h}}\int_I \left(\boldsymbol{y}_h'(\boldsymbol{x}',x_3)-\boldsymbol{x}' \right)\,\d x_3,\\
	V_h(\boldsymbol{x}')&\coloneqq\frac{h}{\sqrt{e_h}}\int_I y_h^3(\boldsymbol{x}',x_3)\,\d x_3.
	\end{align*}
	As the sequence  $(r_h/e_h)$ is bounded because of  \eqref{eqn:contradiction-tris}, Proposition \ref{prop:cd} entails that 	 $(\boldsymbol{U}_h)$ and $(V_h)$ are bounded in $W^{1,2}(S;\R^2)$ and $W^{1,2}(S)$, respectively, for $h\ll 1$. Hence, 
	taking into account \eqref{eqn:load-f}--\eqref{eqn:load-g} and \eqref{eqn:contradiction-bis}, we compute
	\begin{equation*}
	\begin{split}
	&\lim_{h\to 0^+} \left \{ \frac{1}{e_h}\int_\Omega \boldsymbol{f}_h\cdot (\boldsymbol{y}_h'-\boldsymbol{x}')\,\d\boldsymbol{x}+\frac{1}{e_h}\int_\Omega g_h y_h^3\,\d\boldsymbol{x}  \right\}\\
	&=\lim_{h\to 0^+} \left \{ \left ( h^{\beta/2-2}\vee \frac{h^{\beta/2}}{\sqrt{e_h}} \right) \int_S h^{-\beta/2}\boldsymbol{f}_h \cdot \boldsymbol{U}_h\,\d\boldsymbol{x}'+\frac{h^{\beta/2}}{\sqrt{e_h}}\int_S h^{-\beta/2-1}g_h\,V_h\,\d\boldsymbol{x}'\right \}=0.
	\end{split}
	\end{equation*}
	This, in turn, gives
	\begin{equation*}
	\begin{split}
	1&=\lim_{h\to 0^+} \frac{1}{e_h} I_h(\boldsymbol{q}_h)\\
	&= \lim_{h\to 0^+} \left \{\frac{1}{e_h} J_h(\boldsymbol{q}_h)+ \frac{1}{e_h}\int_\Omega \boldsymbol{f}_h\cdot (\boldsymbol{y}_h'-\boldsymbol{x}')\,\d\boldsymbol{x}+\frac{1}{e_h}\int_\Omega g_h y_h^3\,\d\boldsymbol{x}  \right\}\\
	&=\lim_{h\to 0^+} \frac{1}{e_h} J_h(\boldsymbol{q}_h)\leq C(M) \lim_{h\to 0^+} \frac{h^\beta}{e_h}=0,
	\end{split}
	\end{equation*}
	where in the last line we exploited \eqref{eqn:Jh-est}  and \eqref{eqn:contradiction-bis}. 
	This provides a contradiction and, in turn, we necessarily have 
	\begin{equation}
	\label{eqn:L-limsup}
		L\coloneqq \limsup_{h \to 0^+} \frac{1}{h^\beta} I_h(\boldsymbol{q}_h)<+\infty.
	\end{equation}
	
	\textbf{Step 3 (Bound on the constant).}
	Clearly, the constant $L\geq 0$ in \eqref{eqn:L-limsup} depends on the sequence $(\boldsymbol{q}_h)$. We claim that 
	\begin{equation}
		\label{eqn:bound-constant}
		L\leq C(M)
	\end{equation}
	for some constant $C(M)>0$ depending on $M$ but not on $(\boldsymbol{q}_h)$. 
	First, observe that, in view of \eqref{eqn:rigidity-bis} and \eqref{eqn:L-limsup}, there hold
	\begin{equation}
	\label{eqn:ratio-est}
	\limsup_{h \to 0^+} \frac{r_h}{h^\beta}\leq C + L,  \qquad \lim_{h \to 0^+} \frac{r_h}{h^4}=0. 
	\end{equation}
	Set $\boldsymbol{u}_h\coloneqq\mathcal{U}_h(\boldsymbol{q}_h)$and $v_h\coloneqq \mathcal{V}_h(\boldsymbol{q}_h)$.
	By applying Proposition \ref{prop:cd} to $\widehat{\boldsymbol{y}}_h=\boldsymbol{y}_h$ with $e_h=h^\beta$, we deduce the two estimates:
	\begin{align}
		\label{eqn:u-L}
			\|\boldsymbol{u}_h\|_{W^{1,2}(S;\R^2)}&\leq C \left(  \frac{\sqrt{r_h}}{h^{\beta/2}}+\frac{r_h}{h^{\beta/2+2}}\right),
		\\
		\label{eqn:v-L}
		\|v_h\|_{W^{1,2}(S)}&\leq C  \frac{\sqrt{ r_h}}{h^{\beta/2}}.
	\end{align}
	Then, recalling \eqref{eqn:load-f}--\eqref{eqn:load-g} and \eqref{eqn:Jh-est} and employing the H\"{o}lder inequality, we estimate
	\begin{equation*}
	\label{eqn:bound-L}
		\begin{split}
			\frac{1}{h^\beta}I_h(\boldsymbol{q}_h)&=\frac{1}{h^\beta}J_h(\boldsymbol{q}_h)+\int_S h^{-\beta/2}\boldsymbol{f}_h\cdot \boldsymbol{u}_h\,\d\boldsymbol{x}'+\int_S h^{-\beta/2-1}g_h\,v_h\,\d\boldsymbol{x}'\\
			&\leq C(M)+ C \left( \|\boldsymbol{u}_h\|_{L^2(S;\R^2)} + \|v_h\|_{L^2(S)} \right).
		\end{split}
	\end{equation*}
	Exploiting \eqref{eqn:u-L}--\eqref{eqn:v-L}, we take the superior limit, as $h\to 0^+$, at both sides of the previous inequality with the aid of \eqref{eqn:L-limsup} and \eqref{eqn:ratio-est}. This yields
	\begin{equation*}
		L \leq C(M)+ C \sqrt{L},
	\end{equation*}
	which entails \eqref{eqn:bound-constant} by applying the Young inequality.
	 
	Therefore, in view of \eqref{eqn:L-limsup}, we have proved that
	\begin{equation*}
		I_h(\boldsymbol{q}_h)\leq C(M)h^\beta
	\end{equation*}
	for $h\ll 1$ depending on $(\boldsymbol{q}_h)$. 
	At this point, from \eqref{eqn:load-bdd},  we infer $L_h(\boldsymbol{q}_h)\leq C(M)$ and, recalling \eqref{eqn:bddf}, also \eqref{eqn:bdde}.
\end{proof}
}

We are now ready to prove our second main result.

\begin{proof}[Proof of Theorem \ref{thm:convergence-almost-minimizers}]
Given $h>0$, let $\boldsymbol{n}_h\colon \Omega_h \to \S^2$ be constantly equal to some fixed $\boldsymbol{e}\in\S^2$. In this case $\overline{\boldsymbol{q}}_h\coloneqq(\boldsymbol{\pi}_h,\boldsymbol{n}_h)\in\mathcal{Q}_h$. We claim that $F_h(\overline{\boldsymbol{q}}_h)\leq C$ and, in turn, $\inf_{\mathcal{Q}_h}F_h\leq C$. To see this, using \eqref{eqn:Taylor-Phi}, we compute
\begin{equation*}
\begin{split}
    E_h^{\rm el}(\overline{\boldsymbol{q}}_h)&=\frac{1}{h^\beta}\int_\Omega \Phi \left( \boldsymbol{I}-\frac{h^{\beta/2}}{1+h^{\beta/2}}\boldsymbol{e}\otimes\boldsymbol{e} \right)\d\,\boldsymbol{x}\\
    &=\frac{1}{2}\int_\Omega Q \left (-\frac{1}{1+h^{\beta/2}}\boldsymbol{e}\otimes\boldsymbol{e}\right)\d\,\boldsymbol{x}+\frac{1}{h^\beta} \int_\Omega \omega \left (-\frac{h^{\beta/2}}{1+h^{\beta/2}}\boldsymbol{e}\otimes\boldsymbol{e} \right)\d\,\boldsymbol{x}\leq C.
\end{split}
\end{equation*}
Denote by $\overline{\psi}_h$ the stray-field potential corresponding to $\overline{\boldsymbol{q}}_h$. By \eqref{eqn:maxwell-stability} and the change-of-variable formula, we have
\begin{equation*}
    E_h^{\rm mag}(\overline{\boldsymbol{q}}_h)=\frac{1}{h}\int_{\Omega_h} |\nabla \,\overline{\psi}_h|^2\,\boldsymbol{\xi}\leq h^{-1}\leb(\Omega_h)=\lebt(S).
\end{equation*}
Thus, $E_h(\overline{\boldsymbol{q}}_h)\leq C$. Moreover, by \eqref{eqn:load-h} and the change-of-variable formula, there holds
\begin{equation*}
    \left |L_h(\overline{\boldsymbol{q}}_h)\right |=\left |\frac{1}{h}\int_{\Omega_h}\boldsymbol{h}_h\cdot\boldsymbol{e}\,\d\boldsymbol{\xi}\right |=\int_\Omega |\boldsymbol{h}_h\circ\boldsymbol{\pi}_h|\,\d\boldsymbol{x}\leq C\,||\boldsymbol{h}_h\circ\boldsymbol{\pi}_h||_{L^2(\R^3;\R^3)}\leq C.
\end{equation*}
Therefore, the claim is proved.

At this point, \eqref{eqn:almost-minimizer} yields the bound $\sup_{h>0}F_h(\boldsymbol{q}_h)\leq C$. Then, by Lemma \ref{lem:energy-scaling}, there holds $\sup_{h \ll 1}E_h(\boldsymbol{q}_h)\leq C$. Thanks to Proposition \ref{prop:compactness}, we find $\boldsymbol{q}_0=(\boldsymbol{u},v,\boldsymbol{\zeta})\in \mathcal{Q}_0$ and $\boldsymbol{\chi}\in L^2(\rt;\rt)$ such that  \eqref{eqn:compactness-clamped-horizontal}--\eqref{eqn:compactness-clamped-G-structure} hold true. In particular, this shows \eqref{eqn:cam-horizontal}--\eqref{eqn:cam-lagrangian}.

Given \eqref{eqn:compactness-clamped-eta}--\eqref{eqn:compactness-clamped-lagrangian} and  \eqref{eqn:compactness-clamped-R}--\eqref{eqn:compactness-clamped-G-structure}, we apply Proposition \ref{prop:lb} to $\widehat{\boldsymbol{q}}_h=\boldsymbol{q}_h$ and we conclude that
\begin{equation*}
    E_0(\boldsymbol{q}_0)\leq \liminf_{h\to0^+} E_h(\boldsymbol{q}_h), 
\end{equation*}
while exploiting \eqref{eqn:compactness-clamped-horizontal}--\eqref{eqn:compactness-clamped-vertical} and \eqref{eqn:compactness-clamped-eta}, and applying the change-of-variable formula, we see that
\begin{equation*}
    L_0(\boldsymbol{q}_0)=\lim_{h\to 0^+} L_h(\boldsymbol{q}_h).
\end{equation*}
These two facts together clearly entail
\begin{equation}
\label{eqn:cam-liminf}
    F_0(\boldsymbol{q}_0)\leq \liminf_{h\to0^+} F_h(\boldsymbol{q}_h).
\end{equation}
Now, let $\widehat{\boldsymbol{q}}_0=(\widehat{\boldsymbol{u}},\widehat{v},\widehat{\boldsymbol{\zeta}})\in \mathcal{Q}_0$. By Proposition \ref{prop:recovery}, there exists $(\widehat{\boldsymbol{q}}_h)$ with $\widehat{\boldsymbol{q}}_h\in \mathcal{Q}_h$ such that \eqref{eqn:recovery-horizontal}--\eqref{eqn:recovery-vertical}, \eqref{eqn:recovery-eta} and \eqref{eqn:recovery-limit} hold true. Hence, we have
\begin{equation}
\label{eqn:cam-recovery}
    F_0(\widehat{\boldsymbol{q}}_0)=\lim_{h\to 0^+} F_h(\widehat{\boldsymbol{q}}_h).
\end{equation}
Eventually, combining  \eqref{eqn:almost-minimizer} with \eqref{eqn:cam-liminf} and \eqref{eqn:cam-recovery}, we obtain
\begin{equation*}
    F_0(\boldsymbol{q}_0)\leq \liminf_{h \to 0^+} F_h(\boldsymbol{q}_h)\leq \liminf_{h \to 0^+} F_h(\widehat{\boldsymbol{q}}_h)=F_0(\widehat{\boldsymbol{q}}_0).
\end{equation*}
Since $\widehat{\boldsymbol{q}}_0$ is arbitrary, this shows that $\boldsymbol{q}_0$ is a minimizer of $F_0$ on $\mathcal{Q}_0$. 
\end{proof}

\section{Quasistatic setting}
\label{sec:quasistatic-setting}
In this last section, we study the variational model under quasistatic evolution in presence of dissipative effects. The evolution is driven by time-dependent applied loads and our framework is the theory of {rate-independent systems} \cite{mielke.roubicek}.

\subsection{The quasistatic model}
We start by describing the setting. Let $T>0$ be the time horizon. Given $h>0$, let 
\begin{equation*}
    \boldsymbol{f}_h\in W^{1,1}(0,T;L^2(S;\R^2)), \quad g_h \in W^{1,1}(0,T;L^2(S)), \quad \boldsymbol{h}_h\in W^{1,1}(0,T;L^2(\R^3;\R^3)),
\end{equation*}
represent a time-dependent horizontal force, vertical force and external magnetic field, respectively. Without loss of generality, we assume that all these functions are absolutely continuous in time.
The corresponding functional $\mathcal{L}_h\colon [0,T]\times \mathcal{Q}_h\to \R$ is defined by
\begin{equation*}
    \mathcal{L}_h(t,\boldsymbol{q})\coloneqq 
    \frac{1}{h^\beta}\int_\Omega \boldsymbol{f}_h(t)\cdot (\boldsymbol{y}'-\boldsymbol{x}')\,\d\boldsymbol{x}+\frac{1}{h^\beta}\int_\Omega g_h(t)\,y^3\,\d\boldsymbol{x}+
    \frac{1}{h}\int_{\Omega^{\boldsymbol{y}}} \boldsymbol{h}_h(t)\cdot \boldsymbol{m}\,\d\boldsymbol{\xi},
\end{equation*}
where $\boldsymbol{q}=(\boldsymbol{y},\boldsymbol{m})$, 
and the total energy $\mathcal{F}_h \colon [0,T]\times \mathcal{Q}_h\to \R$ reads
\begin{equation}
\label{eqn:F_h-time}
    \mathcal{F}_h(t,\boldsymbol{q})=E_h(\boldsymbol{q})-\mathcal{L}_h(t,\boldsymbol{q}).
\end{equation}
Using the notation introduced in \eqref{eqn:lagrangian-magnetization}, we define the dissipation distance  $\mathcal{D}_h\colon \mathcal{Q}_h\times\mathcal{Q}_h\to [0,+\infty)$ by setting
\begin{equation}
    \label{eqn:D_h}
    \mathcal{D}_h(\boldsymbol{q},\widehat{\boldsymbol{q}})\coloneqq \int_\Omega |\mathcal{Z}_h(\boldsymbol{q})-\mathcal{Z}_h(\widehat{\boldsymbol{q}})|\,\d\boldsymbol{x}.
\end{equation}
Thus, the energy dissipated along an evolution $\boldsymbol{q}\colon [0,T]\to \mathcal{Q}_h$ within the time interval $[r,s]\subset [0,T]$ is given by
\begin{equation*}
    \mathrm{Var}_{\mathcal{D}_h}(\boldsymbol{q};[r,s])\coloneqq \sup \left \{\sum_{i=1}^N \mathcal{D}_h(\boldsymbol{q}(t^i),\boldsymbol{q}(t^{i-1})):\:\text{$\Pi=(t^0,\dots,t^{N})$ partition of $[r,s]$} \right\}.
\end{equation*}
By partition of the time interval $[r,s]\subset [0,T]$ we mean any finite ordered set $\Pi=(t^0,t^1,\dots,t^N)\subset [r,s]^N$ with $r=t^0<t^1<\dots<t^N=s$. Also,  the size of the partition is given by $$|\Pi|\coloneqq \max\{t^i-t^{i-1}:\:i=1,\dots,N\}.$$

For the reduced model, we also have an evolution driven by time-dependent applied loads. Precisely, we assume that there exist
\begin{equation*}
    \boldsymbol{f}_0\in W^{1,1}(0,T;L^2(S;\R^2)), \quad  g_0\in W^{1,1}(0,T;L^2(S)), \quad \boldsymbol{h}_0\in W^{1,1}(0,T;L^2(\R^2;\R^3)),
\end{equation*}
such that, as $h \to 0^+$, the following convergences hold:
\begin{align}
    \label{eqn:load-f-time}
    \text{$h^{-\beta/2}\boldsymbol{f}_h$}&\text{$\to \boldsymbol{f}_0$ in $W^{1,1}(0,T;L^2(S;\R^2))$,}\\
    \label{eqn:load-g-time}
    \text{$h^{-\beta/2-1}g_h$}&\text{$\to g_0$ in $W^{1,1}(0,T;L^2(S))$,}\\
    \label{eqn:load-h-time}
    \text{$\boldsymbol{h}_h\circ \boldsymbol{\pi}_h$}&\text{$\to \chi_I\boldsymbol{h}_0$ in $W^{1,1}(0,T;L^2(\R^3;\R^3))$.}
\end{align}
Also here, we assume that the functions $\boldsymbol{f}_0, g_0$ and $\boldsymbol{h}_0$ are all absolutely continuous in time.
In \eqref{eqn:load-h-time}, we trivially set $\boldsymbol{h}_h \circ \boldsymbol{\pi}_h(t)\coloneqq \boldsymbol{h}_h(t)\circ \boldsymbol{\pi}_h$ for every $t\in [0,T]$. This gives a map in $W^{1,1}(0,T;L^2(\R^3;\R^3))$  whose time derivative is given by $\dot{\boldsymbol{h}}_h \circ \boldsymbol{\pi}_h(t)\coloneqq \dot{\boldsymbol{h}}_h(t) \circ \boldsymbol{\pi}_h$ for every $t \in (0,T)$. Note that the limiting magnetic field $\boldsymbol{h}$ is a priori assumed to be independent on $x_3$.

We define the functional  $\mathcal{L}_0\colon [0,T]\times \mathcal{Q}_0\to \R$  by setting
\begin{equation*}
    \mathcal{L}_0(t,\boldsymbol{q})\coloneqq \int_S \boldsymbol{f}_0(t)\cdot \boldsymbol{u}\,\d\boldsymbol{x}'+\int_S g_0(t)\,v\,\d\boldsymbol{x}'+\int_S \boldsymbol{h}_0(t)\cdot \boldsymbol{\zeta}\,\d\boldsymbol{x},
\end{equation*}
where $\boldsymbol{q}_0=(\boldsymbol{u,v,\boldsymbol{\zeta}})$,
so that the limiting total energy $\mathcal{F}_0\colon [0,T]\times \mathcal{Q}_0\to \R$ reads
\begin{equation}
\label{eqn:F_0-time}
    \hspace*{5mm}\mathcal{F}_0(t,\boldsymbol{q}_0)\coloneqq E_0(\boldsymbol{q}_0)-\mathcal{L}_0(t,\boldsymbol{q}_0).
\end{equation}
The dissipation distance $\mathcal{D}_0\colon \mathcal{Q}_0 \times \mathcal{Q}_0\to [0,+\infty)$ for the reduced model is defined as 
\begin{equation}
\label{eqn:D_0}
    \mathcal{D}_0(\boldsymbol{q}_0,\widehat{\boldsymbol{q}}_0)=\int_S |\boldsymbol{\zeta}-\widehat{\boldsymbol{\zeta}}|\,\d\boldsymbol{x}',
\end{equation}
where $\boldsymbol{q}_0=(\boldsymbol{u},v,\boldsymbol{\zeta})$ and $\widehat{\boldsymbol{q}}_0=(\widehat{\boldsymbol{u}},\widehat{v},\widehat{\boldsymbol{\zeta}})$. Therefore, the energy dissipated by an evolution $\boldsymbol{q}_0\colon [0,T]\to\mathcal{Q}_0$, where $\boldsymbol{q}_0(t)=(\boldsymbol{u}(t),v(t),\boldsymbol{\zeta}(t))$, within the time interval $[r,s]\subset[0,T]$ is given by
\begin{equation*}
    \mathrm{Var}_{\mathcal{D}_0}(\boldsymbol{q}_0;[r,s])\coloneqq \sup \left \{\sum_{i=1}^N \mathcal{D}_0(\boldsymbol{q}_0(t^i),\boldsymbol{q}_0(t^{i-1})):\:\text{$\Pi=(t^0,\dots,t^N)$ partition of $[r,s]$} \right\}.
\end{equation*}

In the theory of rate-independent systems, one defines {energetic solutions} as time evolutions satisfying two requirements: a global stability condition and an energy-dissipation balance \cite[Definition 2.1.2]{mielke.roubicek}.
{\MMM We recall below the definition for both the bulk model and the reduced model.

\begin{definition}[Energetic solution to the bulk quasistatic model]
	\label{def:energetic-solution-bulk}
	A function $\boldsymbol{q}_h\colon[0,T]\to {\mathcal{Q}}_h$ is termed an energetic solution to the bulk quasistatic model at thickness $h>0$ if  $t \mapsto \partial_t\mathcal{F}_h(t,\boldsymbol{q}_h(t))$ 
	belongs to $L^1(0,T)$ and, for every $t\in[0,T]$, the following two conditions are satisfied:
	\begin{enumerate}[\bf (i)]
		\item \textbf{Bulk global stability at thickness $\boldsymbol{h}$:}
		\begin{equation}
		\label{eqn:bulk-stability}
		\forall\, \widehat{\boldsymbol{q}}_h\in{\mathcal{Q}}_h,\quad \mathcal{F}_h(t,\boldsymbol{q}_h(t))\leq \mathcal{F}_h(t,\widehat{\boldsymbol{q}}_h)+\mathcal{D}_h(\boldsymbol{q}_h(t),\widehat{\boldsymbol{q}}_h);
		\end{equation}
		\item \textbf{Bulk energy-dissipation balance at thickness $\boldsymbol{h}$:}
		\begin{equation}
		\label{eqn:bulk-energy-balance}
		\mathcal{F}_h(t,\boldsymbol{q}_h(t))+\text{Var}_{\mathcal{D}_h}(\boldsymbol{q}_h;[0,t])=\mathcal{F}_h(0,\boldsymbol{q}_h(0))+\int_0^t \partial_t \mathcal{F}_h(\tau,\boldsymbol{q}_h(\tau))\,\d\tau.
		\end{equation}
	\end{enumerate}
\end{definition}

\begin{definition}[Energetic solution to the reduced quasistatic model]
	\label{def:energetic-solution-reduced}
	A function  $\boldsymbol{q}_0\colon[0,T]\to {\mathcal{Q}}_0$ is termed an energetic solution to the reduced quasistatic model if  $t \mapsto \partial_t\mathcal{F}_0(t,\boldsymbol{q}_0(t))$ 
	belongs to $L^1(0,T)$ and, for every $t\in[0,T]$, the following two conditions are satisfied:
	\begin{enumerate}[\bf (i)]
		\item \textbf{Reduced global stability:}
		\begin{equation}
		\label{eqn:reduced-stability}
		\forall\, \widehat{\boldsymbol{q}}_0\in {\mathcal{Q}}_0,\quad \mathcal{F}_0(t,\boldsymbol{q}_0(t))\leq \mathcal{F}_0(t,\widehat{\boldsymbol{q}}_0)+\mathcal{D}_0(\boldsymbol{q}_0(t),\widehat{\boldsymbol{q}}_0);
		\end{equation}
		\item \textbf{Reduced energy-dissipation balance:}
		\begin{equation}
		\label{eqn:reduced-energy-balance}
		\mathcal{F}_0(t,\boldsymbol{q}_0(t))+\text{Var}_{\mathcal{D}_0}(\boldsymbol{q}_0;[0,t])=\mathcal{F}_0(0,\boldsymbol{q}_0(0))+\int_0^t \partial_t \mathcal{F}_0(\tau,\boldsymbol{q}_0(\tau))\,\d\tau.
		\end{equation}
	\end{enumerate}
\end{definition}

The main results of the section are Theorem \ref{thm:evolutionary-gamma-convergence} and Theorem \ref{thm:convergence-AIMP} which represent the counterparts of Theorem \ref{thm:gamma-conv} and Theorem \ref{thm:convergence-almost-minimizers} in the quasistatic setting. These  will be presented in the next two subsections.
}

\subsection{\texorpdfstring{Evolutionary $\boldsymbol{\Gamma}$}{gamma}-convergence}
\label{subsec:egamma}
{\MMM
Our third main result states the evolutionary $\Gamma$-convergence of the sequence $(\mathcal{F}_h)$ to the functional $\mathcal{F}_0$, as $h \to 0^+$. More explicitly, we prove that sequences of energetic solutions to the bulk model converge to energetic solutions to the reduced model.

\begin{theorem}[Evolutionary $\boldsymbol{\Gamma}$-convergence]
	\label{thm:evolutionary-gamma-convergence}
Let $(\boldsymbol{q}_h^0)$ with  $\boldsymbol{q}_h^0=(\boldsymbol{y}_h^0,\boldsymbol{m}_h^0)\in{\mathcal{Q}}_h$. Suppose that, for every $h>0$, the following condition is satisfied:  
\begin{equation}
\label{eqn:egamma-id-stability}
\forall\, \widehat{\boldsymbol{q}}_h \in \mathcal{Q}_h,\quad \mathcal{F}_h(0,\boldsymbol{q}_h^0)\leq \mathcal{F}_h(0,\widehat{\boldsymbol{q}}_h)+\mathcal{D}_h(\boldsymbol{q}_h^0,\widehat{\boldsymbol{q}}_h).
\end{equation}
Also, assume the existence of $\boldsymbol{q}^0_0=(\boldsymbol{u}^0,v^0,\boldsymbol{\zeta}^0) \in {\mathcal{Q}}_0$ such that the following convergences hold, as $h \to 0^+$:
\begin{align}
\label{eqn:egamma-datum-horizonal}
\text{$\boldsymbol{u}_h^0\coloneqq\mathcal{U}_h(\boldsymbol{q}_h^0)$}&\text{$\wk\boldsymbol{u}^0$ in $W^{1,2}(S;\R^2)$,}\\
\label{eqn:egamma-datum-vertical}
\text{$v_h^0\coloneqq \mathcal{V}_h(\boldsymbol{q}_h^0)$}&\text{$\to v^0$ in $W^{1,2}(S)$,}\\
\label{eqn:egamma-datum-lagrangian}
\text{$\boldsymbol{z}_h^0\coloneqq \mathcal{Z}_h(\boldsymbol{q}_h^0)$}&\text{$\to\boldsymbol{\zeta}^0$ in $L^1(\Omega;\rt)$,}\\
\label{eqn:egamma-datum-energy}
\text{$\mathcal{F}_h(0,\boldsymbol{q}_h^0)$}&\text{$\to \mathcal{F}_0(0,\boldsymbol{q}_0^0)$.}
\end{align}
Let $(\boldsymbol{q}_h)$ with $\boldsymbol{q}_h\colon [0,T]\to{\mathcal{Q}}_h$ an energetic solution to the bulk quasistatic model at thickness $h>0$.

Then, there exists a function $\boldsymbol{q}_0\colon [0,T]\to {\mathcal{Q}}_0$ with $\boldsymbol{q}_0(t)=(\boldsymbol{u}(t),v(t),\boldsymbol{\zeta}(t))$ for every $t\in [0,T]$ satisfying the initial condition $\boldsymbol{q}_0(0)=\boldsymbol{q}_0^0$ which is an energetic solution to the reduced quasistatic model. 
Moreover, up to subsequences, the following convergences hold, as $h\to 0^+$:
\begin{align}
\label{eqn:egamma-lagrangian-magnetization}
\text{$\forall\,t\in [0,T],\quad \boldsymbol{z}_h(t)\coloneqq \mathcal{Z}_h(\boldsymbol{q}_h(t))$}&\text{$\to \boldsymbol{\zeta}(t)$ in $L^1(\Omega;\rt)$,}\\
\label{eqn:egamma-dissipation}
\text{$\forall\, t\in[0,T],\quad \mathrm{Var}_{\mathcal{D}_h}(\boldsymbol{q}_h;[0,t])$}&\text{$\to\mathrm{Var}_{\mathcal{D}_0}(\boldsymbol{q}_0;[0,t])$,}\\
\label{eqn:egamma-energy}
\text{$\forall\,t\in[0,T],\quad \mathcal{F}_h(t,\boldsymbol{q}_h(t))$}&\text{$\to\mathcal{F}_0(t,\boldsymbol{q}_0(t))$,}\\
\label{eqn:egamma-power}
\text{$\partial_t \mathcal{F}_h(\cdot,\boldsymbol{q}_h)$}&\text{$\to\partial_t\mathcal{F}_0(\cdot,\boldsymbol{q}_0)$ in $L^1(0,T)$.}
\end{align}
In particular, the function $\boldsymbol{q}_0$ is measurable and bounded. Also,   $t\mapsto \boldsymbol{\zeta}(t)$ belongs to $BV([0,T];L^1(\Omega;\rt))$, while  $t\mapsto \mathcal{F}_0(t,\boldsymbol{q}_0(t))$ belongs to $BV([0,T])$.
\end{theorem}

In the previous theorem, the measurability of $\boldsymbol{q}_0$ is meant with respect to the Borel sets of the space $W^{1,2}_0(S;\R^2)\times W^{2,2}_0(S)\times W^{1,2}(S;\rt)$ equipped with the weak product topology. Similarly, the boundedness of the map $\boldsymbol{q}_0$ is understood as the one of the function
\begin{equation*}
	t\mapsto \|\boldsymbol{u}(t)\|_{W^{1,2}(S;\R^2)}+\|v(t)\|_{W^{2,2}(S)}+\|\boldsymbol{\zeta}(t)\|_{W^{1,2}(S;\rt)}.
\end{equation*}

\begin{remark}[Existence of energetic solutions to the bulk model]
\label{rem:ex-ensol-bulk}
We stress that Theorem \ref{thm:evolutionary-gamma-convergence} is just a convergence result: the existence of energetic solutions to the bulk quasistatic model is part of the assumptions. However, our setting is compatible with the existence of energetic solutions. Indeed, recalling Remark \ref{rem:existence-min-bulk}, if the function $\Phi$ in \eqref{eqn:density-W_h} satisfies a feasible polyconvexity assumption, then the existence of energetic solution to the bulk model can be established \cite{bresciani.quasistatic,bresciani.davoli.kruzik}.   
\end{remark}

\begin{remark}[Time-dependent boundary conditions]
\label{rem:time-dependent-bc}
So far, we are not able to treat time-dependent Dirichlet boundary conditions in the proof of Theorem \ref{thm:convergence-AIMP}. Indeed, given the Eulerian character of some energy terms, the approach developed in \cite[Section 4]{francfort.mielke} seems not to be applicable in our setting.
Still, time-dependent boundary conditions like the ones in Remark \ref{rem:dirichlet-bc} can be enforced in a relaxed form as we will briefly discuss.

Let $\overline{\boldsymbol{u}}\in W^{1,1}(0,T;W^{1,\infty}(S;\R^2))$ and $\overline{v}\in W^{1,1}(0,T;W^{2,\infty}(S))$. For every $h>0$, we define the time-dependent deformation $\overline{\boldsymbol{y}}_h\in W^{1,1}(0,T;W^{1,\infty}(\Omega;\R^3))$ by setting
\begin{equation*}
    \overline{\boldsymbol{y}}_h(t)\coloneqq \boldsymbol{\pi}_h + h^{\beta/2} \left ( \begin{matrix}  \overline{\boldsymbol{u}}(t)  \\ 0\end{matrix}   \right )+ h^{\beta/2-1} \left ( \begin{matrix}  \boldsymbol{0}'  \\ \overline{v}(t)\end{matrix}   \right )-h^{\beta/2} x_3 \left ( \begin{matrix}  \nabla'\overline{v}(t)  \\ 0\end{matrix}   \right ).
\end{equation*}
Let $\Gamma \subset \partial S$ be measurable with respect to the one-dimensional Hausdorff measure with $\mathscr{H}^1(\Gamma)>0$. For every $h>0$, the boundary condition 
\begin{equation*}
    \text{$\forall\,t \in [0,T],\quad \boldsymbol{y}=\overline{\boldsymbol{y}}_h(t)$ on $\Gamma \times I$}
\end{equation*}
on admissible deformations is imposed in a relaxed form by augmenting the energy $\mathcal{F}_h$ by the term
\begin{equation*}
    (t,\boldsymbol{q})\mapsto \frac{1}{h^{\beta/2}} \int_{\Gamma \times I} |\boldsymbol{y}'-\overline{\boldsymbol{y}}_h'(t)|\,\d\boldsymbol{a}+\frac{1}{h^{\beta/2-1}} \int_{\Gamma \times I} |\boldsymbol{y}^3-\overline{\boldsymbol{y}}_h^3(t)|\,\d\boldsymbol{a},
\end{equation*}
where $\boldsymbol{q}=(\boldsymbol{y},\boldsymbol{m})$. Of course, in this case, we remove the clamped boundary conditions from the definition of $\mathcal{Q}_h$ in \eqref{eqn:class-Qh}.  
The scalings are chosen in such a way that the corresponding term in the reduced model, which has to be added to $\mathcal{F}_0$, is given by
\begin{equation*}
    (t,\boldsymbol{q}_0)\mapsto \int_{\Gamma \times I} |\boldsymbol{u}-\overline{\boldsymbol{u}}(t)+x_3(\nabla'v-\nabla'\,\overline{v}(t))|\,\d\boldsymbol{a}+\int_\Gamma |v-\overline{v}(t)|\,\d\boldsymbol{l},
\end{equation*}
where $\boldsymbol{q}_0=(\boldsymbol{u},v,\boldsymbol{\zeta})$.
This latter term imposes in a relaxed form  the following limiting boundary conditions:
\begin{equation*}
    \text{$\forall\,t\in [0,T],$ \quad $\boldsymbol{u}=\overline{\boldsymbol{u}}(t)$ on $\Gamma$, \quad $v=\overline{v}(t)$ on $\Gamma$, \quad $\nabla'v=\nabla'\,\overline{v}(t)$ on $\Gamma$.}
\end{equation*}
Clearly, Lemma \ref{lem:clamped} must be suitably modified. 
Note that, contrary to Remark \ref{rem:dirichlet-bc}, no particular assumption on  $\Gamma$ is required in this case.
\end{remark}

The rest of the subsection is devoted to the proof of Theorem \ref{thm:evolutionary-gamma-convergence}. We begin with some preliminary results. The first one constitutes the analogue of Lemma \ref{lem:energy-scaling} for the quasistatic setting.

\begin{lemma}[Energy scaling]
	\label{lem:time-energy-scaling}
	Let $M>0$ and let $(\boldsymbol{q}_h)$ with $\boldsymbol{q}_h\colon [0,T]\to {\mathcal{Q}}_h$   satisfy
	\begin{equation}
		\label{eqn:tot-en-bdd-above}
		\sup_{h>0} \, \sup_{t\in [0,T] \vphantom{h>0}} \mathcal{F}_h(t,\boldsymbol{q}_h(t))\leq M.
	\end{equation}
	Then, there exist $C(M)>0$ and $\overline{h}>0$, where the former does not depend on $(\boldsymbol{q}_h)$, such that
	\begin{align}
		\label{eqn:en-bdd-above}
		\sup_{h \leq \overline{h}}\,\sup_{t\in [0,T] \vphantom{h\leq\overline{h}}} E_h(\boldsymbol{q}_h(t))&\leq C(M).
	\end{align}
\end{lemma}
\begin{proof}
The proof of Lemma \ref{lem:time-energy-scaling} works like the one of Lemma \ref{lem:energy-scaling}. In this case, we consider 
\begin{equation*}
    r_h\coloneqq \sup_{t\in [0,T]} \mathcal{R}_h(\boldsymbol{y}_h(t)), \qquad e_h\coloneqq \sup_{t\in [0,T]} I_h(\boldsymbol{q}_h(t)),
\end{equation*}
where $\boldsymbol{q}_h(t)=(\boldsymbol{y}_h(t),\boldsymbol{m}_h(t))$. Here, we employ the notation in \eqref{eqn:rig} and \eqref{eqn:Ih}. Then, we follow the same strategy of Lemma \ref{lem:energy-scaling}.
\end{proof}

The second preliminary result shows that the sequence of functionals $(\mathcal{F}_h)$ satisfy suitable controls with respect to time. These  represent one of the main assumptions of the theory of evolutionary $\Gamma$-convergence for rate-independent processes \cite{mielke.roubicek.stefanelli}. For convenience, we denote by $Z\subset (0,T)$ the complement of the set of times at which all the functions $\boldsymbol{f}_h$, $g_h$ and $\boldsymbol{h}_h$ for every $h>0$ as well as the functions $\boldsymbol{f}_0$, $g_0$ and $\boldsymbol{h}_0$ are differentiable. Thus, there holds $\mathscr{L}^1(Z)=0$.  

\begin{lemma}[Time-control of the total energy]
	\label{lem:time-control}
Let $M>0$. Then, there exist two constants $C(M),L(M)>0$ such that, defining $\kappa_h \in L^1(0,T)$ by setting
\begin{equation*}
	\kappa_h(t)\coloneqq C(M) \left(h^{-\beta/2} \|\dot{\boldsymbol{f}}_h(t)\|_{L^2(S;\R^2)}\hspace*{-2pt}+\hspace*{-2pt}h^{-\beta/2-1}\|\dot{g}_h(t)\|_{L^2(S)}\hspace*{-2pt} +\hspace*{-2pt} \|\dot{\boldsymbol{h}}_h\circ \boldsymbol{\pi}_h(t)\|_{L^2(\rt;\rt)} \right),
\end{equation*}
we have that, for every $(\widehat{\boldsymbol{q}}_h)$ with $\widehat{\boldsymbol{q}}_h\in {\mathcal{Q}}_h$  satisfying
\begin{equation}
	\label{eqn:bounded-energy}
	\sup_{h>0} \,  		  E_h(\widehat{\boldsymbol{q}}_h)\leq M,	
\end{equation}
there exists $\overline{h}>0$ such that the following estimates hold for every $h\leq \overline{h}$: 
\begin{align}
	\label{eqn:dr-gronwall1}
	\forall\, t\in(0,T)\setminus Z, \quad   |\partial_t \mathcal{F}_h(t,\widehat{\boldsymbol{q}}_h)|&\leq \kappa_h(t) \left( \mathcal{F}_h(t,\widehat{\boldsymbol{q}}_h) + L(M) \right),\\
	\label{eqn:dr-gronwall2}
	 \forall s,t\in[0,T],\quad \mathcal{F}_h(t,\widehat{\boldsymbol{q}}_h)+L(M)&\leq  \left(\mathcal{F}_h(s,\widehat{\boldsymbol{q}}_h)+L(M)\right) \mathrm{e}^{|K_h(t)-K_h(s)|},\\
	 	\label{eqn:dr-gronwall3}
		\forall t\in (0,T)\setminus Z,\,\forall s\in [0,T], \quad 
		|\partial_t \mathcal{F}_h(t,\widehat{\boldsymbol{q}}_h)|&\leq \kappa_h(t) \left(\mathcal{F}_h(s,\widehat{\boldsymbol{q}}_h)+L(M)\right) \mathrm{e}^{|K_h(t)-K_h(s)|}.
\end{align}
In the previous estimates,  we define $K_h\in AC([0,T])$ by setting
\begin{equation*}
	K_h(t)\coloneqq \int_0^t K_h(\tau)\,\d\tau.
\end{equation*}
In particular, the constant $L(M)>0$ satisfies
\begin{equation}
	\label{eqn:L-below}
	\inf_{\vphantom{t \in [0,T]}h>0} \inf_{t \in [0,T]} \mathcal{F}_h(t,\widehat{\boldsymbol{q}}_h)\geq -L(M).
\end{equation}
Eventually, defining $\kappa_0\in L^1(0,T)$ and $K_0\in AC([0,T])$ as 
\begin{align*}
\kappa_0(t)&\coloneqq C(M) \left( \|\dot{\boldsymbol{f}}_0(t)\|_{L^2(S;\R^2)}+\|\dot{g}_0(t)\|_{L^2(S)} + \|\chi_I\dot{\boldsymbol{h}}_0(t)\|_{L^2(\rt;\rt)} \right),\\
K_0(t)&\coloneqq \int_0^t \kappa_0(\tau)\,\d\tau,
\end{align*}
 there hold:
\begin{align}
	\label{eqn:kh-conv}
	\text{$\kappa_h$}&\text{$\to \kappa_0$ in $L^1(0,T)$,}\\
	\label{eqn:Kh-conv}
	K_h&\text{$\to K_0$ uniformly in $[0,T]$.}
\end{align}
\end{lemma}
We stress that, in the previous statement, $\overline{h}$ depends on $(\widehat{\boldsymbol{q}}_h)$ while the constants $C(M)$ and $L(M)$ depend only on $M$.
\begin{proof}
First of all, \eqref{eqn:kh-conv}--\eqref{eqn:Kh-conv} come directly from \eqref{eqn:load-f-time}--\eqref{eqn:load-h-time}.  By \eqref{eqn:bounded-energy} and Remark \ref{rem:norm}, setting $\widehat{\boldsymbol{u}}_h\coloneqq \mathcal{U}_h(\widehat{\boldsymbol{q}}_h)$ and $\widehat{v}_h\coloneqq \mathcal{V}_h(\widehat{\boldsymbol{q}}_h)$, we have 
\begin{equation}
	\label{eqn:n1}
	\|\widehat{\boldsymbol{u}}_h\|_{W^{1,2}(S;\R^2)}+\|\widehat{v}_h\|_{W^{1,2}(S)}\leq C(M) \left (\sqrt{E_h^{\rm el}(\widehat{\boldsymbol{q}}_h)}+1 \right).
\end{equation}
Let $\widehat{\boldsymbol{q}}_h=(\widehat{\boldsymbol{y}}_h,\widehat{\boldsymbol{m}}_h)$ and set $\widehat{\boldsymbol{\mu}}_h\coloneqq \mathcal{M}_h(\widehat{\boldsymbol{q}}_h)$.
Recalling \eqref{eqn:prestrain-determinant}, \eqref{eqn:coercivity-Phi-det} and \eqref{eqn:a3},   we estimate
\begin{equation*}
		\int_{\rt} |\widehat{\boldsymbol{\mu}}_h|^2\,\d\boldsymbol{x}=\int_{\Omega} \frac{1}{\det \nabla_h\widehat{\boldsymbol{y}}_h}\,\d\boldsymbol{x}\leq C \int_{\Omega} \frac{1}{(\det \nabla_h\widehat{\boldsymbol{y}}_h)^a}\,\d\boldsymbol{x}\leq C \left(h^\beta  E_h^{\rm el}(\widehat{\boldsymbol{q}}_h) +1 \right),
\end{equation*}
so that
\begin{equation}
	\label{eqn:n2}
	\|\widehat{\boldsymbol{\mu}}_h\|_{L^2(\rt;\rt)}\leq C \left( \sqrt{E_h^{\rm el}(\widehat{\boldsymbol{q}}_h)}+1  \right).
\end{equation}
Using \eqref{eqn:n1}--\eqref{eqn:n2} and the H\"{o}lder inequality, we control the work of applied loads. For $t\in (0,T)\setminus Z$, we have 
\begin{equation*}
	\begin{split}
		|\mathcal{L}_h(t,\widehat{\boldsymbol{q}}_h)|&\leq \left | \int_S h^{-\beta/2}\boldsymbol{f}_h(t)\cdot \widehat{\boldsymbol{u}}_h\,\d \boldsymbol{x}' \right |+ \left |  \int_S h^{-\beta/2-1}g_h(t)\,\widehat{v}_h\,\d\boldsymbol{x}' \right |
		+\left |\int_{\rt} \boldsymbol{h}_h \circ \boldsymbol{\pi}_h(t)\cdot \widehat{\boldsymbol{\mu}}_h\,\d\boldsymbol{x} \right |\\
		&\leq C(M) \left( \|\widehat{\boldsymbol{u}}_h\|_{W^{1,2}(S;\R^2)}+\|\widehat{v}_h\|_{W^{1,2}(S)} + \|\widehat{\boldsymbol{\mu}}_h\|_{L^2(\rt;\rt)}   \right)
		\leq C(M) \left(\sqrt{E_h^{\rm el}(\widehat{\boldsymbol{q}}_h)}+1    \right).
	\end{split}
\end{equation*}
Here, we exploited the uniform boundedness of $(h^{-\beta/2}\boldsymbol{f}_h)$, $(h^{-\beta/2-1}g_h)$ and $(\boldsymbol{h}_h \circ \boldsymbol{\pi}_h)$ with respect to time which follows from \eqref{eqn:load-f-time}--\eqref{eqn:load-h-time} by the Morrey embedding. Then, using the Young inequality, we obtain
\begin{equation}
	\label{eqn:sr1}
	\begin{split}
		\mathcal{F}_h(t,\widehat{\boldsymbol{q}}_h)&\geq E_h(\widehat{\boldsymbol{q}}_h)-|\mathcal{L}_h(t,\widehat{\boldsymbol{q}}_h)|\geq E_h(\widehat{\boldsymbol{q}}_h) - C(M)\left(\sqrt{E_h^{\rm el}(\widehat{\boldsymbol{q}}_h)}+1    \right)
		\geq C(M) E_h(\widehat{\boldsymbol{q}}_h)-L(M),
	\end{split}
\end{equation}
for some constants $C(M), L(M)>0$. 

Now, for convenience, set 
\begin{equation*}
	\widetilde{\kappa}_h(t)\coloneqq h^{-\beta/2}\|\dot{\boldsymbol{f}}_h(t)\|_{L^2(S;\R^2)}+h^{-\beta/2-1}\|\dot{g}_h(t)\|_{L^2(S)}+\|\dot{\boldsymbol{h}}_h\circ \boldsymbol{\pi}_h(t)\|_{L^2(\rt;\rt)}.     
\end{equation*}
Making use of the H\"{o}lder inequality together with \eqref{eqn:n1}--\eqref{eqn:n2}, we estimate
\begin{equation*}
	\begin{split}
		|\partial_t \mathcal{F}_h(t,\widehat{\boldsymbol{q}}_h)|&\leq \left | \int_S h^{-\beta/2}\dot{\boldsymbol{f}}_h(t)\cdot \widehat{\boldsymbol{u}}_h\,\d \boldsymbol{x}' \right |+ \left |  \int_S h^{-\beta/2-1}\dot{g}_h(t)\,\widehat{v}_h\,\d\boldsymbol{x}' \right |
		+\left |\int_{\rt} \dot{\boldsymbol{h}}_h \circ \boldsymbol{\pi}_h(t)\cdot \widehat{\boldsymbol{\mu}}_h\,\d\boldsymbol{x} \right |\\
		&\leq  \widetilde{\kappa}_h(t) \left( \|\widehat{\boldsymbol{u}}_h\|_{W^{1,2}(S;\R^2)}+\|\widehat{v}_h\|_{W^{1,2}(S)} + \|\widehat{\boldsymbol{\mu}}_h\|_{L^2(\rt;\rt)}   \right)\\
		&\leq C(M) \widetilde{\kappa}_h(t) \left(\sqrt{E_h^{\rm el}(\widehat{\boldsymbol{q}}_h)}+1    \right).
	\end{split}
\end{equation*}
Combining this with \eqref{eqn:sr1}, we obtain
\begin{equation*}
	|\partial_t \mathcal{F}_h(t,\widehat{\boldsymbol{q}}_h)|\leq C(M) \widetilde{\kappa}_h(t) \left( \mathcal{F}_h(t,\widehat{\boldsymbol{q}}_h)+L(M)   \right),
\end{equation*}
which proves \eqref{eqn:dr-gronwall1}. 
Note that \eqref{eqn:L-below} holds in view of \eqref{eqn:sr1}. From \eqref{eqn:dr-gronwall1}, we deduce \eqref{eqn:dr-gronwall2} thanks to the Gronwall inequality. Eventually, \eqref{eqn:dr-gronwall3} follows.
\end{proof}

We now proceed with the proof of Theorem \ref{thm:evolutionary-gamma-convergence}.

\begin{proof}[Proof of Theorem \ref{thm:evolutionary-gamma-convergence}]
We rigorously follow the scheme in \cite{mielke.roubicek.stefanelli}. Therefore, 
 the proof is subdivided into six steps.

\textbf{Step 1 (A priori estimates).} Let $h>0$ and $t \in [0,T]$. Let $\overline{\boldsymbol{q}}_h\in\mathcal{Q}_h$ be defined as in the proof of Theorem  \ref{thm:convergence-almost-minimizers}. 
Testing \eqref{eqn:bulk-stability} with $\widehat{\boldsymbol{q}}_h=\overline{\boldsymbol{q}}_h$,  we obtain
\begin{equation}
\label{eqn:egamma-apriori1}
	\mathcal{F}_h(t,\boldsymbol{q}_h(t))\leq \mathcal{F}_h(t,\overline{\boldsymbol{q}}_h)+\mathcal{D}_h(\boldsymbol{q}_h(t),\overline{\boldsymbol{q}}_h).
\end{equation}
As the applied loads are uniformly bounded with respect to $h>0$ and $t \in [0,T]$ thanks to \eqref{eqn:load-f-time}--\eqref{eqn:load-h-time} and the Morrey embedding, we check that the first term on the right-hand side  of \eqref{eqn:egamma-apriori1} is uniformly bounded with respect to both $h>0$ and $t\in [0,T]$ by arguing as in the proof of Theorem  \ref{thm:convergence-almost-minimizers}.  The second term on the right-hand side  of \eqref{eqn:egamma-apriori1} is also uniformly bounded with respect to $h>0$ and $t\in [0,T]$ because both $\boldsymbol{q}_h(t)$ and $\overline{\boldsymbol{q}}_h$ satisfy the magnetic saturation constrain. Therefore, we deduce \eqref{eqn:tot-en-bdd-above} for some $M>0$ and, by applying Lemma \ref{lem:time-energy-scaling}, also \eqref{eqn:en-bdd-above} for some $\overline{h}>0$.

Applying Lemma \ref{lem:time-control} to $\widehat{\boldsymbol{q}}_h=\boldsymbol{q}_h(T)$, we see that
\begin{equation}
	\label{eqn:iL}
	\inf_{h\leq \overline{h}}\mathcal{F}_h(T,\boldsymbol{q}_h(T))\geq -L(M).
\end{equation}
and, for $h \leq \overline{h}$, there holds
\begin{equation}
	\label{eqn:ii}
	\forall\,t \in (0,T)\setminus Z, \quad |\partial_t \mathcal{F}_h(t,\boldsymbol{q}_h(t))\leq \kappa_h(t) \left( \mathcal{F}_h(t,\boldsymbol{q}_h(t))+L(M)\right).
\end{equation}
By \eqref{eqn:bulk-energy-balance},  we have
\begin{equation}
	\label{eqn:egamma-apriori2}
	\mathrm{Var}_{\mathcal{D}_h}(\boldsymbol{q}_h;[0,T])=\mathcal{F}_h(0,\boldsymbol{q}_h^0)-\mathcal{F}_h(T,\boldsymbol{q}_h(T))+\int_0^T \partial_t \mathcal{F}_h(\tau,\boldsymbol{q}_h(\tau))\,\d\tau.
\end{equation}
The first term on the right-hand side of \eqref{eqn:egamma-apriori2} is uniformly bounded  by \eqref{eqn:egamma-datum-energy} and so is the second one because of \eqref{eqn:iL}. Additionally, thanks to \eqref{eqn:ii},  we estimate
\begin{equation}
\label{eqn:power-bdd}
\begin{split}
	\int_0^T \left | \partial_t \mathcal{F}_h(\tau,\boldsymbol{q}_h(\tau)) \right |\,\d\tau &\leq  \int_0^T \kappa_h(\tau) \left( \mathcal{F}_h(\tau,\boldsymbol{q}_h(\tau))+L(M) \right) \,\mathrm{e}^{K_h(\tau)}\d\tau\\
	&\leq  \left(M+L(M) \right) \left( \mathrm{e}^{K_h(T)}-1 \right)\leq C(M,T),
\end{split}
\end{equation}
where in the last line we used \eqref{eqn:tot-en-bdd-above} and the uniform boundedness of   $(K_h)$  which comes from \eqref{eqn:Kh-conv}. Therefore, we obtained the following a priori estimate:
\begin{equation}
	\label{eqn:apriori-estimate}
	\sup_{h>0} \, \sup_{0\leq t\leq T \vphantom{h>0}} \mathcal{F}_h(t,\boldsymbol{q}_h(t)) + \sup_{h\leq \overline{h}} \mathrm{Var}_{\mathcal{D}_h}(\boldsymbol{q}_h;[0,T])\leq C(M).
\end{equation}
For every $h>0$, we define  $f_h\colon [0,T]\to \R$ by setting $f_h(t)\coloneqq \mathcal{F}_h(t,\boldsymbol{q}_h(t))$. In view of the previous estimate, we already know that $(f_h)$ is uniformly bounded.
We claim that this sequence has also uniformly bounded total variation.  Observe that, by \eqref{eqn:bulk-energy-balance}, for every $s,t\in[0,T]$ with $s<t$, there holds
\begin{equation*}
	\mathcal{F}_h(t,\boldsymbol{q}_h(t))+\mathrm{Var}_{\mathcal{D}_h}(\boldsymbol{q}_h;[s,t])=\mathcal{F}_h(s,\boldsymbol{q}_h(s))+\int_s^t \partial_t \mathcal{F}_h(\tau,\boldsymbol{q}_h(\tau))\,\d\tau.
\end{equation*}
Let $\Pi=(t^0,\dots,t^N)$ be a partition of $[0,T]$. In this case, we estimate
\begin{equation*}
	\begin{split}
		\sum_{i=1}^N \left | f_h(t^i)-f_h(t^{i-1})\right |&\leq \sum_{i=1}^N \mathrm{Var}_{\mathcal{D}_h}(\boldsymbol{q}_h;[t^{i-1},t^i])
		+\sum_{i=1}^N \int_{t^{i-1}}^{t^i}\left |\partial_t \mathcal{F}_h(\tau,\boldsymbol{q}_h(\tau)) \right|\,\d\tau\\
		&=\mathrm{Var}_{\mathcal{D}_h}(\boldsymbol{q}_h;[0,T])+\int_0^T |\partial_t \mathcal{F}_h(\tau,\boldsymbol{q}_h(\tau)) |\,\d\tau.
	\end{split}
\end{equation*}
As the right-hand side is uniformly by \eqref{eqn:power-bdd}--\eqref{eqn:apriori-estimate}, given the arbitrariness of $\Pi$, we deduce
\begin{equation}
\label{eqn:tot-en-tot-var}
	\sup_{h \ll 1} \mathrm{Var} \left(f_h;[0,T]  \right)\leq C(M).
\end{equation}

\textbf{Step 2 (Compactness).} For every $h>0$, set
\begin{equation*}
	\mathcal{H}_h\coloneqq \bigcup_{0\leq\widehat{t}\leq T } \left \{ \widehat{\boldsymbol{q}}_h\in{\mathcal{Q}}_h:\:\mathcal{F}_h(\widehat{t},\widehat{\boldsymbol{q}}_h)\leq M      \right \}, \qquad \mathcal{K}_h\coloneqq \left \{\mathcal{Z}_h(\widehat{\boldsymbol{q}}_h):\:\widehat{\boldsymbol{q}}_h\in\mathcal{H}_h     \right \}.
\end{equation*}
Then, define
\begin{equation*}
	\mathcal{H}\coloneqq \bigcup_{h>0} \mathcal{H}_h, \qquad \mathcal{K}\coloneqq \bigcup_{h>0} \mathcal{K}_h.
\end{equation*}
Observe that, by Lemma \ref{lem:time-energy-scaling}, every sequence $(\widehat{\boldsymbol{q}}_h) \subset \mathcal{H}$ satisfies 
\begin{equation*}
	\sup_{h \ll 1}\,\sup_{0\leq t\leq T \vphantom{h>0}} E_h(\widehat{\boldsymbol{q}}_h)\leq C(M).
\end{equation*}
Therefore, by Proposition \ref{prop:compactness},  the sequence $(\widehat{\boldsymbol{z}}_h)$, where $\widehat{\boldsymbol{z}}_h\coloneqq \mathcal{Z}_h(\widehat{\boldsymbol{q}}_h)$ for every $h>0$, is compact in $L^1(\Omega;\rt)$. This shows that $\mathcal{K}$ is a compact subset of $L^1(\Omega;\rt)$. 

Now, in view of \eqref{eqn:tot-en-bdd-above}, the sequence $(\boldsymbol{q}_h)$ takes values in $\mathcal{H}$.
For convenience, for every $h>0$, we define
\begin{equation*}
	\boldsymbol{z}_h\colon [0,T]\to L^1(\Omega;\rt), \quad \delta_h\colon [0,T]\to [0,+\infty),
\end{equation*}
by setting 
\begin{equation*}
	 \boldsymbol{z}_h(t)\coloneqq \mathcal{Z}_h(\boldsymbol{q}_h(t)), \quad \delta_h(t)\coloneqq \mathrm{Var}_{\mathcal{D}_h}(\boldsymbol{q}_h;[0,t]).
\end{equation*}
By \eqref{eqn:apriori-estimate}--\eqref{eqn:tot-en-tot-var}, the sequence $(f_h)$ is uniformly bounded in $BV([0,T])$. 
By construction, $(\boldsymbol{z}_h)$ takes values in $\mathcal{K}$ while, by \eqref{eqn:apriori-estimate}, it has uniformly bounded variation. Eventually, by definition, each map $\delta_h$ is increasing and the sequence $(\delta_h)$ is uniformly bounded by \eqref{eqn:apriori-estimate}. 
At this point, we apply the Helly compactness theorem \cite[Theorem 3.2]{mainik.mielke}. This yields the existence of three maps
\begin{equation*}
	f \colon [0,T]\to [0,+\infty), \quad \boldsymbol{z}\colon [0,T]\to L^1(\Omega;\R^3), \quad \delta\colon [0,T]\to [0,+\infty),
\end{equation*}
with $f\in BV([0,T])$, $\boldsymbol{z}\in BV(0,T;L^1(\Omega;\rt))$ and $\delta$ increasing, such that, up to subsequences, the following convergences hold:
\begin{align}
	\label{eqn:helly-f}
	\forall\, t\in [0,T], \quad f_h(t)&\to f(t),\\
	\label{eqn:helly-z}
	\forall\, t\in [0,T], \quad \boldsymbol{z}_h(t)&\to\text{$\boldsymbol{z}(t)$ in $L^1(\Omega;\rt)$,}\\
	\label{eqn:helly-delta}
	\forall\, t\in [0,T], \quad \delta_h(t)&\to \delta(t).
\end{align}

To construct the candidate solution $\boldsymbol{q}_0\colon [0,T]\to {\mathcal{Q}}_0$ of the reduced model we proceed as follows.
For every $h>0$ , define the functions
\begin{equation*}
	\begin{split}
		&\boldsymbol{u}_h\colon [0,T]\to W^{1,2}(S;\R^2), \quad  v_h\colon [0,T]\to W^{1,2}(S), \quad  \boldsymbol{w}_h\colon [0,T]\to W^{1,2}(S;\rt),\\
		&\hspace{20mm} \boldsymbol{\mu}_h\colon [0,T] \to L^2(\rt;\rt), \quad \boldsymbol{N}_h\colon [0,T]\to L^2(\rt;\rtt).
	\end{split}	
\end{equation*}
by setting
\begin{equation*}
	\begin{split}
		&\boldsymbol{u}_h(t)\coloneqq \mathcal{U}_h(\boldsymbol{q}_h(t)), \quad v_h(t)\coloneqq \mathcal{V}_h(\boldsymbol{q}_h(t)),\quad \boldsymbol{w}_h(t)\coloneqq \mathcal{W}_h(\boldsymbol{q}_h(t)),\\
		&\hspace{20mm}\boldsymbol{\mu}_h(t)\coloneqq \mathcal{M}_h(\boldsymbol{q}_h(t)), \quad \boldsymbol{N}_h(t)\coloneqq \mathcal{N}_h(\boldsymbol{q}_h(t)).
	\end{split}
\end{equation*}
Recalling \eqref{eqn:en-bdd-above}, by Proposition \ref{prop:compactness}, we have that
\begin{equation}
\label{eqn:time-bdd-cmpt}
	\begin{split}
		\sup_{h\leq \overline{h}} \sup_{t\in [0,T]}  \left \{ \|\boldsymbol{u}_h(t)\|_{W^{1,2}(S;\R^2)}+\|v_h(t)\|_{W^{1,2}(S)}+\|\boldsymbol{w}_h(t)\|_{W^{1,2}(S;\rt)} \right \}&\\
		+\sup_{h\leq \overline{h}} \sup_{t\in [0,T]}  \left \{\|\boldsymbol{\mu}_h(t)\|_{L^2(\rt;\rt)}+\|\boldsymbol{N}_h(t)\|_{L^2(\rt;\rtt)} \right \}& \leq C(M).
	\end{split}
\end{equation}
Now, define $\mathcal{X}$ as the set of all quintuplets
\begin{equation*}
	(\widehat{\boldsymbol{u}},\widehat{v},\widehat{\boldsymbol{w}},\widehat{\boldsymbol{\mu}},\widehat{\boldsymbol{N}})\in W^{1,2}_0(S;\R^2)\times W^{1,2}_0(S)\times W^{1,2}_0(S;\rt)\times L^2(\rt;\rt)\times L^2(\rt;\rtt)
\end{equation*}
with
\begin{equation*}
	\begin{split}
		\sup_{h>0}  \left \{        \|\widehat{\boldsymbol{u}}\|_{W^{1,2}(S;\R^2)}+\|\widehat{v}\|_{W^{1,2}(S)}+\|\widehat{\boldsymbol{w}}\|_{W^{1,2}(S;\rt)} + \|\widehat{\boldsymbol{\mu}}\|_{L^2(\rt;\rtt)}+\|\widehat{\boldsymbol{N}}\|_{L^2(\rt;\rtt)}    \right \}&\leq C(M),
	\end{split}
\end{equation*}
where $C(M)>0$ is the same constant in \eqref{eqn:time-bdd-cmpt}. The space $\mathcal{X}$ is endowed with the product weak topology, which makes it a complete and separable metric space. For every $h>0$, we define the set-valued map $\mathfrak{S}_h\colon [0,T]\to \mathcal{P}(\mathcal{X})$ by setting $\mathfrak{S}_h(t)\coloneqq \{\boldsymbol{s}_h(t)\}$, where
\begin{equation*}
	\boldsymbol{s}_h(t)\coloneqq \left(\boldsymbol{u}_h(t),v_h(t),\boldsymbol{w}_h(t),\boldsymbol{\mu}_h(t),\boldsymbol{N}_h(t)    \right).
\end{equation*}
In view of \eqref{eqn:time-bdd-cmpt}, this map takes indeed values in $\mathcal{X}$. Consider $\mathfrak{S}\colon [0,T]\to \mathcal{P}(\mathcal{X})$ with $\mathfrak{S}(t)$ defined as the set of all limit points of the sequence $(\boldsymbol{s}_h(t))$ in $\mathcal{X}$.  By definition,  this is a closed subset of $\mathcal{X}$ for every $t \in [0,T]$, while the  map $\mathfrak{S}$ is measurable thanks to \cite[Theorem 8.2.5]{aubin.frankowska}. From \eqref{eqn:time-bdd-cmpt}, by weak compactness, we deduce that the set $\mathfrak{S}(t)$ is nonempty for every $t\in [0,T]$. Therefore, by the measurable selection theorem \cite[Theorem 8.1.3]{aubin.frankowska}, there exists a measurable map $\boldsymbol{s}\colon [0,T]\to\mathcal{X}$ with $\boldsymbol{s}(t)\in \mathfrak{S}(t)$ for every $t \in [0,T]$. 
For any such $t$, let $\boldsymbol{s}(t)=(\boldsymbol{u}(t),v(t),\boldsymbol{w}(t),\boldsymbol{\mu}(t),\boldsymbol{N}(t))$.  By definition of $\mathfrak{S}(t)$, there exists a subsequence $(h_k)$, possibly depending on $t$, with $h_k\to 0$, as $k\to \infty$,  such that the following convergences hold, as $k\to \infty$:
\begin{align}
	\label{eqn:egamma-cmpt-u}
	\boldsymbol{u}_{h_k}(t)&\text{$\wk \boldsymbol{u}(t)$ in $W^{1,2}(S;\R^2)$,}\\
	\label{eqn:egamma-cmpt-v}
	v_{h_k}(t)&\text{$\wk v(t)$ in $W^{1,2}(S)$,}\\
	\boldsymbol{w}_{h_k}(t)&\text{$\wk \boldsymbol{w}(t)$ in $W^{1,2}(S;\rt)$,}\\
	\label{eqn:egamma-cmpt-mu}
	\boldsymbol{\mu}_{h_k}(t)&\text{$\wk \boldsymbol{\mu}(t)$ in $L^2(\rt;\rt)$,}\\
	\label{eqn:egamma-cmpt-nu}
	\boldsymbol{N}_{h_k}(t)&\text{$\wk \boldsymbol{N}(t)$ in $L^2(\rt;\rtt)$.}
\end{align}
Applying Proposition \ref{prop:compactness} to the sequence $(\boldsymbol{q}_{h_k}(t))$ and appealing to the Urysohn property, we deduce several facts. First of all, 
\begin{equation*}
	\boldsymbol{w}(t)=-\frac{1}{12} \left(\hspace*{-5pt}\begin{array}{cc}
	\nabla' v(t) \\ 0
	\end{array}   \hspace*{-5pt}\right),
\end{equation*}
so that $v\in W^{2,2}_0(S)$.
Second, there exist $\boldsymbol{\zeta}(t)\in W^{1,2}(S;\S^2)$ and $\boldsymbol{\chi}(t)\in L^2(\rt;\rt)$ such that
\begin{equation}
	\label{eqn:egamma-mu-nu-id}
	\boldsymbol{\mu}(t)=\chi_\Omega \boldsymbol{\zeta}(t), \quad \boldsymbol{N}(t)=\chi_\Omega \left(\nabla'\boldsymbol{\zeta}(t)\vert \boldsymbol{\chi}(t)   \right).
\end{equation}
Third, we have
\begin{equation}
    \label{eqn:egamma-cmpt-z}
    \text{$\boldsymbol{z}_{h_k}(t)\to\boldsymbol{\zeta}(t)$ in $L^1(\Omega;\rt),$}
\end{equation}
which, together with \eqref{eqn:helly-z}, entails
\begin{equation}
    \label{eqn:egamma-cmpt-z-id}
	\boldsymbol{z}(t)=\boldsymbol{\zeta}(t).
\end{equation}
Eventually,  Proposition \ref{prop:compactness}  ensures the existence of $(\boldsymbol{R}_{h}(t))\subset W^{1,p}(S;SO(3))$ and $\boldsymbol{G}(t)\in L^2(\Omega;\rtt)$ such that, setting $\boldsymbol{F}_h(t)\coloneqq \nabla_h\boldsymbol{y}_h(t)$, we have
\begin{align}
    \label{eqn:time-R}
    \boldsymbol{R}_h(t)&\text{$\to \boldsymbol{I}$ in $L^2(\Omega;\rtt)$,}\\
    \label{eqn:time-G}
    \boldsymbol{G}_h(t)\coloneqq h^{-\beta/2} \left (\boldsymbol{R}_h(t)^\top \boldsymbol{F}_h(t)-\boldsymbol{I} \right )&\text{$\wk \boldsymbol{G}(t)$ in $L^2(\Omega;\rtt)$,}
\end{align}
and, for almost every $\boldsymbol{x}\in\Omega$, there holds
\begin{equation}
\label{eqn:time-G-structure}
    \boldsymbol{G}''(t,\boldsymbol{x})=\sym \nabla'\boldsymbol{u}(t,\boldsymbol{x}')-((\nabla')^2 v(t,\boldsymbol{x}'))x_3.
\end{equation}

From the measurability of the map $t \mapsto \nabla\boldsymbol{w}(t)$ from $[0,T]$ to $L^2(\rtt)$, we infer the measurability of $v$ as a map from $[0,T]$ to $W^{2,2}(S)$. Similarly, from the measurability of $\boldsymbol{\mu}$ and $\boldsymbol{N}$, we deduce the measurability of $t \mapsto \boldsymbol{\zeta}(t)$ as a map from $[0,T]$ to $W^{1,2}(S;\rt)$. Hence, defining $\boldsymbol{q}_0\colon [0,T]\to {\mathcal{Q}}_0$ by setting $\boldsymbol{q}_0(t)\coloneqq (\boldsymbol{u}(t),v(t),\boldsymbol{\zeta}(t))$, this map results to be measurable. Also, given \eqref{eqn:time-bdd-cmpt}--\eqref{eqn:egamma-cmpt-nu}, by  lower semicontinuity, we see that map $\boldsymbol{q}_0$ is bounded.

For every $h>0$, let $P_h \colon (0,T)\to \R$ be defined by setting $P_h(t)\coloneqq \partial_t \mathcal{F}_h(t,\boldsymbol{q}_h(t))$. By Definition \ref{def:energetic-solution-bulk},  $P_h\in L^1(0,T)$ . Thanks to \eqref{eqn:tot-en-bdd-above} and \eqref{eqn:ii}, for every $t \in (0,T)\setminus Z$, we have
\begin{equation}
	\label{eqn:Ph-control}
	|P_h(t)|\leq \kappa_h(t) \left(\mathcal{F}_h(t,\boldsymbol{q}_h(t)) + L(M)  \right)\leq C(M) \kappa_h(t).
\end{equation} 
In view of the previous inequality, the equi-integrability of $(\kappa_h)$, which comes from \eqref{eqn:kh-conv} by the Vitali convergence theorem, implies the one of $(P_h)$. Thus, by the Dunford-Pettis theorem \cite[Theorem 2.54]{fonseca.leoni}, there exists $P\in L^1(0,T)$ such that, up to subsequences, we have
\begin{equation}
\label{eqn:weak-conv-power}
	\text{$P_h \wk P$ in $L^1(0,T)$.}
\end{equation} 
Define $\widehat{P}\colon (0,T) \to \R$ by setting
\begin{equation*}
	\widehat{P}(t)\coloneqq \limsup_{h\to 0^+} P_h(t).
\end{equation*}
Exploiting \eqref{eqn:kh-conv} and\eqref{eqn:Ph-control}, we check that $\widehat{P} \in L^1(0,T)$. Also, by the reverse Fatou lemma \cite[Corollary 5.35]{wheeden.zygmund}, we see that $P\leq \widehat{P}$ almost everywhere in $(0,T)$.

Define $P_0\colon (0,T)\to\R$ by setting $P_0(t)=\partial_t \mathcal{F}_0(t,\boldsymbol{q}_0(t))$. We claim that $\widehat{P}=P_0$ almost everywhere in $(0,T)$. Recalling \eqref{eqn:load-f-time}--\eqref{eqn:load-h-time}, for almost every $t \in (0,T)\setminus Z$, we have
\begin{align*}
	h^{-\beta/2}\dot{\boldsymbol{f}}_h(t)&\text{$\to \dot{\boldsymbol{f}}_0(t)$ in $L^2(S;\R^2)$,}\\
	h^{-\beta/2-1}\dot{g}_h(t)&\text{$\to \dot{g}_0(t)$ in $L^2(S)$,}\\
	\dot{\boldsymbol{h}}_h\circ \boldsymbol{\pi}_h(t)&\text{$\to \chi_I \dot{\boldsymbol{h}}_0(t)$ in $L^2(\rt;\rt)$,}\\
\end{align*}
and 
\begin{equation*}
\begin{split}
	P_h(t)=&-\frac{1}{h^\beta}\int_\Omega \dot{\boldsymbol{f}}_h(t)\cdot \left(\boldsymbol{y}_h'(t)-\boldsymbol{x}' \right)\,\d\boldsymbol{x}-\frac{1}{h^\beta}\int_\Omega \dot{g}_h(t)\,y_h^3(t)\,\d\boldsymbol{x}'
	-\frac{1}{h}\int_{\Omega^{\boldsymbol{y}_h(t)}} \dot{\boldsymbol{h}}_h(t)\cdot \boldsymbol{m}_h(t)\,\d\boldsymbol{\xi}\\
	=&-\int_S h^{-\beta/2}\dot{\boldsymbol{f}}_h(t) \cdot \boldsymbol{u}_h(t)\,\d\boldsymbol{x}'-\int_S h^{-\beta/2-1}\dot{g}_h(t)\,v_h(t)\,\d\boldsymbol{x}'
	-\int_{\rt} \dot{\boldsymbol{h}}_h \circ \boldsymbol{\pi}_h(t)\cdot \boldsymbol{\mu}_h(t)\,\d\boldsymbol{x}.
\end{split}
\end{equation*}
Without loss of generality, we can assume that the subsequence $(h_k)$ in \eqref{eqn:egamma-cmpt-u}--\eqref{eqn:egamma-cmpt-nu} additionally satisfies $P_{h_k}(t)\to \widehat{P}(t)$, as $k\to \infty$. Thus, taking the limit along the subsequence $(h_k)$, as $k\to \infty$, at both sides of the previous identity, we obtain
\begin{equation*}
	\widehat{P}(t)=-\int_S \dot{\boldsymbol{f}}_0(t)\cdot \boldsymbol{u}(t)\,\d\boldsymbol{x}-\int_S \dot{g}_0(t)\,v(t)\,\d\boldsymbol{x}'-\int_S \dot{\boldsymbol{h}}_0(t)\cdot \boldsymbol{\zeta}(t) \,\d\boldsymbol{x}'.
\end{equation*}
As the right-hand side of the previous equation coincides with $P_0(t)$  for almost every $t\in (0,T)\setminus Z$, this proves the claim.

\textbf{Step 3 (Reduced stability).} Fix $t\in [0,T]$ and consider the subsequence $(h_k)$ in \eqref{eqn:egamma-cmpt-u}--\eqref{eqn:egamma-cmpt-nu}. Let $\widehat{\boldsymbol{q}}_0=(\widehat{\boldsymbol{u}},\widehat{v},\widehat{\boldsymbol{\zeta}})\in{\mathcal{Q}}_0$ and let $(\widehat{\boldsymbol{q}}_h)$ be the sequence provided by Proposition \ref{prop:recovery}. For every $k\in\N$, the global stability condition \eqref{eqn:bulk-stability} gives 
\begin{equation}
\label{eqn:bulk-stab}
	\mathcal{F}_{h_k}(t,\boldsymbol{q}_{h_k}(t))\leq \mathcal{F}_{h_k}(t,\widehat{\boldsymbol{q}}_{h_k})+\mathcal{D}_{h_k}(\boldsymbol{q}_{h_k}(t),\widehat{\boldsymbol{q}}_{h_k}).
\end{equation}
For the left-hand side, in view of \eqref{eqn:load-f-time}--\eqref{eqn:load-h-time},  \eqref{eqn:egamma-cmpt-u}--\eqref{eqn:egamma-cmpt-v} and \eqref{eqn:egamma-cmpt-mu}--\eqref{eqn:time-G-structure}, we obtain
\begin{equation}
	\label{eqn:improved1}
	\mathcal{F}_0(t,\boldsymbol{q}_0(t))\leq \liminf_{k\to\infty} \mathcal{F}_{h_k}(t,\boldsymbol{q}_{h_k}(t))=f(t).
\end{equation} 
Here, we applied Proposition \ref{prop:lb} and we also exploited \eqref{eqn:helly-f}. 
For the right-hand side, given \eqref{eqn:load-f-time}--\eqref{eqn:load-h-time},  \eqref{eqn:recovery-horizontal}--\eqref{eqn:recovery-vertical}, \eqref{eqn:recovery-limit} and \eqref{eqn:recovery-eta}--\eqref{eqn:recovery-H}, there holds
\begin{equation*}
	\lim_{k\to \infty} \mathcal{F}_{h_k}(t,\widehat{\boldsymbol{q}}_{h_k}) = \mathcal{F}_0(t,\widehat{\boldsymbol{q}}_0),
\end{equation*}
while \eqref{eqn:recovery-lagrangian-magnetization},  \eqref{eqn:helly-z} and \eqref{eqn:egamma-cmpt-z-id} yield
\begin{equation*}
	\lim_{k\to \infty} \mathcal{D}_{h_k}(\boldsymbol{q}_{h_k}(t),\widehat{\boldsymbol{q}}_{h_k})=\mathcal{D}_0(\boldsymbol{q}_0(t),\widehat{\boldsymbol{q}}_0).
\end{equation*} 
Therefore, taking the inferior limit, as $k\to\infty$, in \eqref{eqn:bulk-stab}, we obtain
\begin{equation*}
	\mathcal{F}_0(t,\boldsymbol{q}_0(t))\leq \mathcal{F}_0(t,\widehat{\boldsymbol{q}}_0)+\mathcal{D}_0(\boldsymbol{q}_0(t),\widehat{\boldsymbol{q}}_0).
\end{equation*}
As $\widehat{\boldsymbol{q}}_0$ is arbitrary, this proves \eqref{eqn:reduced-stability} for $t$ fixed.

\textbf{Step 4 (Upper reduced energy-dissipation inequality).} Fix $t\in [0,T]$ and let $(h_k)$ be the sequence in \eqref{eqn:egamma-cmpt-u}--\eqref{eqn:egamma-cmpt-nu}. By \eqref{eqn:bulk-energy-balance}, for every $k\in\N$, there holds
\begin{equation}
\label{eqn:bulk-en-bal}
\mathcal{F}_{h_k}(t,\boldsymbol{q}_{h_k}(t))+\mathrm{Var}_{\mathcal{D}_{h_k}}(\boldsymbol{q}_{h_k};[0,t])=\mathcal{F}_{h_k}(0,\boldsymbol{q}_{h_k}^0)+\int_0^t P_{h_k}(\tau)\,\d\tau.
\end{equation}
For the left-hand side, we recall \eqref{eqn:improved1}.
Also,  \eqref{eqn:helly-z}--\eqref{eqn:helly-delta} and \eqref{eqn:egamma-cmpt-z-id}  entail
\begin{equation}
	\label{eqn:improved3}
	\mathrm{Var}_{\mathcal{D}_0}(\boldsymbol{q}_0;[0,t])\leq \liminf_{k\to \infty} \mathrm{Var}_{\mathcal{D}_{h_k}}(\boldsymbol{q}_{h_k};[0,t])=\delta(t).
\end{equation}
For the right-hand side, we have \eqref{eqn:egamma-datum-energy}. Also, given \eqref{eqn:weak-conv-power}, there holds
\begin{equation*}
	\lim_{k\to \infty} \int_0^t P_{h_k}(\tau)\,\d\tau =\int_0^t P(\tau)\,\d\tau\leq \int_{0}^{t} P_0(\tau)\,\d\tau,
\end{equation*}
where we employed the inequality $P\leq P_0$ almost everywhere in $(0,T)$. Therefore, taking the inferior limit, as $k\to \infty$, in \eqref{eqn:bulk-en-bal}, we obtain
\begin{equation}
\label{eqn:improved2}
	\begin{split}
		\mathcal{F}_0(t,{\boldsymbol{q}}_0(t))+\mathrm{Var}_{\mathcal{D}_0}(\boldsymbol{q}_0;[0,t])&\leq f(t) + \delta(t)	\\
		&\leq \mathcal{F}_0(0,\boldsymbol{q}_0^0)+ \int_0^t P(\tau)\,\d\tau
		\leq \mathcal{F}_0(0,\boldsymbol{q}_0^0)+ \int_0^t P_0(\tau)\,\d\tau,
	\end{split}
\end{equation}
which is the upper reduced energy-dissipation inequality for  fixed $t\in [0,T]$.

\textbf{Step 5 (Lower reduced energy-dissipation inequality).} We claim that, for every $t\in [0,T]$, there holds
\begin{equation*}
	\mathcal{F}_0(t,\boldsymbol{q}_0(t))+\mathrm{Var}_{\mathcal{D}_0}(\boldsymbol{q}_0;[0,t])\geq \mathcal{F}_0(0,\boldsymbol{q}_0^0)+\int_0^t \partial_t \mathcal{F}_0(\tau,\boldsymbol{q}_0(\tau))\,\d\tau.
\end{equation*}
Thanks to \eqref{eqn:reduced-stability}, the claims follows by applying \cite[Proposition 2.1.2.3]{mielke.roubicek}. 

\textbf{Step 6 (Improved convergences).} We are left to prove \eqref{eqn:egamma-dissipation}--\eqref{eqn:egamma-power}. First, in view of  \eqref{eqn:reduced-energy-balance} and \eqref{eqn:improved2}, we have $f(t)+\delta(t)=\mathcal{F}_0(t,\boldsymbol{q}_0(t))+\mathrm{Var}_{\mathcal{D}_0}(\boldsymbol{q}_0;[0,t])$ for every $t\in [0,T]$. Recalling \eqref{eqn:improved1} and \eqref{eqn:improved3}, this entails $f(t)=\mathcal{F}_0(t,\boldsymbol{q}_0(t))$ and $\delta(t)=\mathrm{Var}_{\mathcal{D}_0}(\boldsymbol{q}_0;[0,t])$, so that \eqref{eqn:egamma-dissipation}--\eqref{eqn:egamma-energy} are proved. Finally, \eqref{eqn:reduced-energy-balance} and \eqref{eqn:improved2} yield $P=P_0$ almost everywhere on $(0,T)$. Thus, \eqref{eqn:egamma-power} follows by applying \cite[Lemma 3.5]{francfort.mielke}.
\end{proof}

}

\subsection{Convergence of solutions of the approximate incremental minimization problem}
\label{subsec:conv-AIMP}

{\MMM
In order to state our fourth main result, we introduce the approximate incremental minimization problem. This is a relaxed version of the incremental minimization problem that has been introduced in order to cope with the possible lack of energy minimizers \cite{mielke.stefanelli}. This is exactly our  situation since, without further assumptions, minimizers of the total energy do to necessarily exist, see Remark \ref{rem:ex-ensol-bulk}.
}

\begin{definition}[Approximate incremental minimization problem]
\label{def:aimp}
Given $h>0$, let $\Pi_h=(t^0_h,\dots,t^{N_h}_h)$ be a partition of $[0,T]$,  $\alpha_h>0$ and  $\boldsymbol{q}^0_h\in \mathcal{Q}_h$. The {approximate incremental minimization problem} (AIMP) determined by $\Pi_h$ with tolerance $\alpha_h$ and initial datum $\boldsymbol{q}^0_h$ reads as follows: for every $i\in\{1,\dots,N_h\}$, find $\boldsymbol{q}^i_h\in\mathcal{Q}_h$ such that
\begin{equation}
    \label{eqn:aimp}
    \mathcal{F}_h(t^i_h,\boldsymbol{q}^i_h)+\mathcal{D}_h(\boldsymbol{q}^{i-1}_h,\boldsymbol{q}^{i}_h)\leq (t^i_h-t^{i-1}_h)\alpha_h+\inf_{\mathcal{Q}_h} \left \{\mathcal{F}_h(t^i_h,\cdot)+\mathcal{D}_h(\boldsymbol{q}^{i-1}_h,\cdot) \right\}.
\end{equation}
\end{definition}

Our fourth main result claims that, for a sequence of partitions whose sizes vanish jointly with the thickness of the plate together with a sequence of tolerances, solutions to the approximate incremental minimization problem, or better their piecewise-constant interpolants, converge to energetic solutions to the reduced model. 

\begin{theorem}[Convergence of solutions of the AIMP]
\label{thm:convergence-AIMP}
Assume $p>3$ and $\beta>6\vee p$. Suppose that the elastic energy density $W_h$ has the form in \eqref{eqn:density-W_h}, where the function $\Phi$ satisfies \eqref{eqn:normalization-Phi}--\eqref{eqn:regularity-Phi} and that the applied loads  satisfy \eqref{eqn:load-f-time}--\eqref{eqn:load-h-time}. Let $(\Pi_h)$ be a sequence of partitions of $[0,T]$ such that  $|\Pi_h|\to 0$, as $h\to 0^+$, and let $(\alpha_h)\subset\R$ with $\alpha_h>0$ be such that $\alpha_h\to0$, as $h \to 0^+$.
Let $(\boldsymbol{q}_h^0)$ with $\boldsymbol{q}_h^0=(\boldsymbol{y}_h^0,\boldsymbol{m}_h^0)\in\mathcal{Q}_h$  be such that \eqref{eqn:egamma-id-stability} holds.
Moreover, assume that there exist $\boldsymbol{q}_0^0=(\boldsymbol{u}^0,v^0,\boldsymbol{\zeta}^0)\in\mathcal{Q}_0$ such that the convergences in \eqref{eqn:egamma-datum-horizonal}--\eqref{eqn:egamma-datum-energy} hold true, as $h\to 0^+$.
For every $h>0$, consider a solution to the AIMP determined by $\Pi_h$ with tolerance $\alpha_h$ and initial datum $\boldsymbol{q}_h^0$ and denote by $\boldsymbol{q}_h\colon [0,T]\to\mathcal{Q}_h$ its right-continuous piecewise-constant interpolant. 

Then, there exists a measurable function $\boldsymbol{q}_0\colon [0,T]\to\mathcal{Q}_0$ with $\boldsymbol{q}_0(t)=(\boldsymbol{u}(t),v(t),\boldsymbol{\zeta}(t))$ for every $t\in[0,T]$ satisfying the initial condition $\boldsymbol{q}_0(0)=\boldsymbol{q}_0^0$ which is an energetic solution to the reduced quasistatic model. Moreover, up to subsequences, the  convergences in \eqref{eqn:egamma-lagrangian-magnetization}--\eqref{eqn:egamma-power} hold true, as $h \to 0^+$.
In particular, the function $\boldsymbol{q}_0$ is measurable and bounded. Also,  $t\mapsto \boldsymbol{\zeta}(t)$ belongs to $BV([0,T];L^1(\Omega;\rtt))$ while  $t\mapsto \mathcal{F}_0(t,\boldsymbol{q}_0(t))$ belongs to $BV([0,T])$.
\end{theorem}

We mention that also Theorem \ref{thm:convergence-AIMP} can be adapted by  imposing time-dependent boundary conditions as in Remark \ref{rem:time-dependent-bc}. 

\begin{remark}[Existence of energetic solutions for the reduced model]
As a byproduct of Theorem \ref{thm:evolutionary-gamma-convergence}, we obtain the existence of energetic solutions for the reduced quasistatic model. However, under our assumptions, this can be established directly. Indeed, the limiting total energy $\mathcal{F}_0$ satisfies suitable compactness properties in view of the coercivity of $Q_{\rm red}$ noted in \eqref{eqn:Q-red-coe} and
the dissipation distance $\mathcal{D}_0$ is continuous on the sublevels of $\mathcal{F}_0$, so that the existence of energetic solutions for the reduced model can be proved following the usual scheme \cite[Theorem 2.1.6]{mielke.roubicek}.
\end{remark}

We move towards the proof of Theorem \ref{thm:convergence-AIMP}. We begin with a preliminary result concerning solutions to the AIMP.

\begin{proposition}[Solutions of the AIMP]
	\label{lem:prop-am}
Let $(\Pi_h)$ be a sequence of partitions of $[0,T]$ and let $(\alpha_h)\subset \R$ with $\alpha_h>0$  be bounded. Also, let $(\boldsymbol{q}_h^0)$ with $\boldsymbol{q}_h^0\in \mathcal{Q}_h$  be such that
    \begin{equation}
    \label{eqn:aimp-bdd}
    \sup_{h>0} {\mathcal{F}}_h(0,\boldsymbol{q}_h^0)\leq C.
    \end{equation}
For every $h>0$, let $\Pi_h=(t_h^0,\dots,t_h^{N_h})$ and let $(\boldsymbol{q}_h^1,\dots,\boldsymbol{q}_h^{N_h})$ be a solution to the AIMP determined by $\Pi_h$ with tolerance $\alpha_h$ and initial datum $\boldsymbol{q}_h^0$.
Then, there exists $M>0$ such that
	\begin{equation}
	\label{eqn:aimp-sublevel}
	\sup_{h>0} \sup_{i\in\{0,\dots,N_h\}} {\mathcal{F}}_h(t_h^i,\boldsymbol{q}_h^i)\leq M.
	\end{equation}
Moreover, for every $h\ll 1$ and  $i\in\{1,\dots,N_h\}$, the following estimates hold:
\begin{align}
    \label{eqn:aimp1}
		\forall\,\widehat{\boldsymbol{q}}_h\in \mathcal{Q}_h,\quad {\mathcal{F}}_h(t_h^i,\boldsymbol{q}_h^i)&\leq(t_h^i-t_h^{i-1})\alpha_h+{\mathcal{F}}_h(t_h^i,\widehat{\boldsymbol{q}}_h)+\mathcal{D}_h(\boldsymbol{q}_h^i,\widehat{\boldsymbol{q}}_h),\\
	\label{eqn:aimp2}
			{\mathcal{F}}_h(t_h^i,\boldsymbol{q}_h^i)+\mathcal{D}_h(\boldsymbol{q}_h^{i-1},\boldsymbol{q}_h^i)&\leq (t_h^i-t_h^{i-1})\alpha_h  +{\mathcal{F}}_h(t_h^{i-1},\boldsymbol{q}_h^{i-1})	+\int_{t_h^{i-1}}^{t_h^i}\partial_t {\mathcal{F}}_h(\tau,\boldsymbol{q}_h^{i-1})\,\d\tau,\\
	\label{eqn:aimp3}
		\mathcal{F}_h(t_h^i,\boldsymbol{q}_h^i)+L(M)+\sum_{j=1}^i \mathcal{D}_h(\boldsymbol{q}_h^{j-1},\boldsymbol{q}_h^j)&\leq \left (\mathcal{F}_h(0,\boldsymbol{q}_h^0)+L(M)+t_i\alpha_h \right)\mathrm{e}^{K_h(t_h^i)}.
\end{align}
Eventually, if $\boldsymbol{q}_h^0$ satisfies \eqref{eqn:egamma-id-stability} for $h\ll 1$, then, for every $i \in \{1,\dots,N_h\}$, there holds
	\begin{equation}
	\label{eqn:aimp2bis}
	\begin{split}
	|{\mathcal{F}}_h(t_h^i,\boldsymbol{q}_h^i)-{\mathcal{F}}_h(t_h^{i-1},\boldsymbol{q}_h^{i-1})+\mathcal{D}_h(\boldsymbol{q}_h^{i-1},\boldsymbol{q}_h^i)|&\leq \left(t_h^i-t_h^{i-2}\right)\alpha_h\\
	&+ ({\mathcal{F}}_h(t_h^{i-1},\boldsymbol{q}_h^{i-1})+1)\left (\mathrm{e}^{K_h(t_h^i)-K_h(t_h^{i-1})}-1 \right),
	\end{split}
	\end{equation}
where we set $t_h^{-1}\coloneqq0$. 
\end{proposition}
\begin{proof}
Let $h>0$ and $i \in \{1,\dots,N_h\}$. For simplicity, we set $\alpha_h^i\coloneqq (t_h^i-t_h^{i-1})\alpha_h$.  
Given \eqref{eqn:aimp}, for every $\widehat{\boldsymbol{q}}_h\in\mathcal{Q}_h$, we have
\begin{equation}
\begin{split}
{\mathcal{F}}_h(t_h^i,\boldsymbol{q}_h^i)&\leq \alpha_h^i+{\mathcal{F}}_h(t_h^i,\widehat{\boldsymbol{q}}_h)+\mathcal{D}_h(\boldsymbol{q}_h^{i-1},\widehat{\boldsymbol{q}}_h)-\mathcal{D}_h(\boldsymbol{q}_h^{i-1},\boldsymbol{q}_h^i) \\
&\leq \alpha_h^i+{\mathcal{F}}_h(t_h^i,\widehat{\boldsymbol{q}}_h)+\mathcal{D}_h(\boldsymbol{q}_h^i,\widehat{\boldsymbol{q}}_h),
\end{split}
\end{equation}
where, in the last line, we employed the triangle inequality. This shows \eqref{eqn:aimp1}.
	
We  check \eqref{eqn:aimp2}. For simplicity, set $f_h^i\coloneqq {\mathcal{F}}_h(t_h^i,\boldsymbol{q}_h^i)$ and $d_h^i\coloneqq \mathcal{D}_h(\boldsymbol{q}_h^{i-1},\boldsymbol{q}_h^i)$.From \eqref{eqn:aimp}, by applying the fundamental theorem of calculus, we obtain
\begin{equation}
\begin{split}
f_h^i-f_h^{i-1}+d_h^i&\leq \alpha_h^i-f_h^{i-1}+{\mathcal{F}}_h(t_h^i,\boldsymbol{q}_h^{i-1})\\
&=\alpha_h^i + {\mathcal{F}}_h(t_h^i,\boldsymbol{q}_h^{i-1})-{\mathcal{F}}_h(t_h^{i-1},\boldsymbol{q}_h^{i-1})\\
&=\alpha_h^i+\int_{t_h^{i-1}}^{t_h^i}\partial_t {\mathcal{F}}_h(\tau,\boldsymbol{q}_h^{i-1})\,\d\tau,
\end{split}
\end{equation}
which gives \eqref{eqn:aimp2}.
	
Testing \eqref{eqn:aimp1} with $\widehat{\boldsymbol{q}}_h=\overline{\boldsymbol{q}}_h$, where the latter is defined as in the proof of Theorem \ref{thm:convergence-almost-minimizers}, we have
\begin{equation*}
{\mathcal{F}}_h(t_h^i,\boldsymbol{q}_h^i)\leq \alpha_h^i+{\mathcal{F}}_h(t_h^i,\overline{\boldsymbol{q}}_h)+\mathcal{D}_h(\boldsymbol{q}_h^i,\overline{\boldsymbol{q}}_h).
\end{equation*}
Exploiting the uniform boundedness of the applied loads with respect to time, which follows from \eqref{eqn:load-f-time}--\eqref{eqn:load-h-time} by the Morrey embedding, with computations analogous to the one in the proof of Theorem \ref{thm:convergence-almost-minimizers} we check that the right-hand side of the previous inequality is uniformly bounded with respect to $h>0$ and $i\in \{0,\dots,N_h \}$.
Here, we also take advantage of \eqref{eqn:aimp-bdd} and of the boundedness of $(\alpha_h)$. Therefore,  \eqref{eqn:aimp-sublevel} is proved.
	
Now, for the sake of clarity, we specify the sequence $(h_n)$ such that $h_n\to 0^+$, as $n\to \infty$, in place of $h>0$. Thus, $(\boldsymbol{q}_{h_n}^1,\dots,\boldsymbol{q}_{h_n}^{N_{h_n}})$ is a solution to the AIMP determined by $\Pi_{h_n}$ with tolerance $\alpha_{h_n}$ and initial datum $\boldsymbol{q}_{h_n}^0$ for every $n\in \N$. By applying  Lemma \ref{lem:energy-scaling} with $(\widehat{\boldsymbol{q}}_h)$ given by the sequence 
\begin{equation*}
		\boldsymbol{q}_{h_1}^0, \boldsymbol{q}_{h_1}^1, \dots, \boldsymbol{q}_{h_1}^{N_{h_1}},  \boldsymbol{q}_{h_2}^0, \boldsymbol{q}_{h_2}^1, \dots, \boldsymbol{q}_{h_2}^{N_{h_2}}, \dots, \boldsymbol{q}_{h_n}^0, \boldsymbol{q}_{h_n}^1, \dots, \boldsymbol{q}_{h_n}^{N_{h_n}},\dots  
\end{equation*} 
we deduce the existence of $\overline{n}\in\N$ such that
\begin{equation}
    \sup_{n\geq \overline{n}} \sup_{i\in \{0,1,\dots,N_{h_n}\}} E_{h_n}(\boldsymbol{q}_{h_n}^i)\leq C(M). 
\end{equation}
Here, we exploit once more the uniform boundedness of the applied loads. Moreover, by Lemma \ref{lem:time-control}, for every  $n\geq\overline{n}$ and for every $i\in\{0,1,\dots,N_{h_n}\}$, the following estimates hold:
	\begin{align*}
	\forall\, t\in(0,T)\setminus Z, \quad   |\partial_t \mathcal{F}_{h_n}(t,{\boldsymbol{q}}_{h_n}^i)|&\leq \kappa_{h_n}(t) \left( \mathcal{F}_{h_n}(t,{\boldsymbol{q}}_{h_n}^i) + L(M) \right),\\
	\forall s,t\in[0,T],\quad \mathcal{F}_{h_n}(t,{\boldsymbol{q}}_{h_n}^i)+L(M)&\leq  \left(\mathcal{F}_{h_n}(s,{\boldsymbol{q}}_{h_n}^i)+L(M)\right) \mathrm{e}^{|K_{h_n}(t)-K_{h_n}(s)|},\\
	\forall t\in (0,T)\setminus Z,\,\forall s\in [0,T],\quad 
	|\partial_t \mathcal{F}_{h_n}(t,{\boldsymbol{q}}_{h_n}^i)|&\leq \kappa_{h_n}(t) \left(\mathcal{F}_{h_n}(s,{\boldsymbol{q}}_{h_n}^i)+L(M)\right) \mathrm{e}^{|K_{h_n}(t)-K_{h_n}(s)|}.
	\end{align*}
Henceforth, for simplicity, we go back to writing $h$ as a subscript without specifying the sequence of thicknesses and we set $\overline{h}\coloneqq h_{\overline{n}}$. Accordingly,  for every $h\leq\overline{h}$ and $i\in\{0,1,\dots,N_h\}$, we have:
	\begin{align}
	\label{eqn:dr-gronwall1-aimp}
	\forall\, t\in(0,T)\setminus Z, \quad   |\partial_t \mathcal{F}_{h}(t,{\boldsymbol{q}}_{h}^i)|&\leq \kappa_{h}(t) \left( \mathcal{F}_{h}(t,{\boldsymbol{q}}_{h}^i) + L(M) \right),\\
	\label{eqn:dr-gronwall2-aimp}
	\forall s,t\in[0,T],\quad \mathcal{F}_{h}(t,{\boldsymbol{q}}_{h}^i)+L(M)&\leq  \left(\mathcal{F}_{h}(s,{\boldsymbol{q}}_{h}^i)+L(M)\right) \mathrm{e}^{|K_{h}(t)-K_{h}(s)|},\\
	\label{eqn:dr-gronwall3-aimp}
	\forall t\in (0,T)\setminus Z,\,\forall s\in [0,T],\quad 
	|\partial_t \mathcal{F}_{h}(t,{\boldsymbol{q}}_{h}^i)|&\leq \kappa_{h}(t) \left(\mathcal{F}_{h}(s,{\boldsymbol{q}}_{h}^i)+L(M)\right) \mathrm{e}^{|K_{h}(t)-K_{h}(s)|}.
	\end{align}
	Going back to proof of \eqref{eqn:aimp3}, let $h \leq \overline{h}$  and $i \in \{1,\dots,N_h\}$. For convenience, set $K_h^i\coloneqq K_h(t_h^i)$. 
	Combining \eqref{eqn:aimp2} with  \eqref{eqn:dr-gronwall3-aimp}, we compute
	\begin{equation*}
	\begin{split}
		f_h^i&\leq f_h^i+d_h^i\leq \alpha_h^i + f_h^{i-1}+\int_{t_h^{i-1}}^{t_h^i} \partial_t \mathcal{F}_h(\tau,\boldsymbol{q}_h^{i-1})\,\d\tau\\
		&\leq \alpha_h^i+f_h^{i-1}+ \left(f_h^{i-1}+L(M)\right) \int_{t_h^{i-1}}^{t_h^i} \kappa_h(\tau) \mathrm{e}^{K_h(\tau)-K_h^{i-1}}\,\d\tau\\
		&=\alpha_h^i+f_h^{i-1} + \left(f_h^{i-1}+L(M)\right)  \left(\mathrm{e}^{K_h^i-K_h^{i-1}} -1 \right)\\
		&=\alpha_h^i-L(M)+ \left(f_h^{i-1}+L(M)\right)\mathrm{e}^{K_h^i-K_h^{i-1}}.
	\end{split}
	\end{equation*}
	Thus
	\begin{equation*}
		f_h^i+L(M)\leq \alpha_h^i+\left(f_h^{i-1}+L(M)\right)\mathrm{e}^{K_h^i-K_h^{i-1}},
	\end{equation*}
	from which, by induction,  we obtain
	\begin{equation}
	\label{eqn:aimp5}
	\begin{split}
	f_h^i+L(M)\leq \left ( f_h^0+L(M) +\sum_{j=1}^i \mathrm{e}^{-K_h^j} \alpha_h^j \right) \mathrm{e}^{K_h^i}.
	\end{split}    
	\end{equation}
	Here, we set $f_h^0\coloneqq \mathcal{F}_h(0,\boldsymbol{q}_h^0)$.
	From \eqref{eqn:aimp2}, using \eqref{eqn:dr-gronwall3-aimp}, we estimate
	\begin{equation}
	\label{eqn:aimpt1}
	\begin{split}
	f_h^i-f_h^{i-1}+d_h^i &\leq \alpha_h^i+\int_{t_h^{i-1}}^{t_h^i} \partial_h \mathcal{F}_h(\tau,\boldsymbol{q}_h^{i-1})\,\d\tau\\
	&\leq \alpha_h^i+(f_h^{i-1}+L(M))\int_{t_h^{i-1}}^{t_h^i} \kappa_h(\tau) \mathrm{e}^{K_h(\tau)-K_h(t_h^{i-1})}\,\d\tau\\
	&\leq \alpha_h^i + (f_h^{i-1}+L(M))\left(\mathrm{e}^{K_h(t_h^i)-K_h^{i-1}} - 1\right).
	\end{split}
	\end{equation}
	Then, summing \eqref{eqn:aimpt1}, with $j$ in place of $i$, for $j\in\{1,\dots,i\}$ and employing \eqref{eqn:aimp5}, with $j-1$ in place of $i$, we obtain
	\begin{equation*}
	\begin{split}
	f_h^i+\sum_{j=1}^i d_h^j+L(M)&\leq f_h^0+L(M)+\sum_{j=1}^i \alpha_h^j+\sum_{j=1}^i (f_h^{j-1}+L(M)) \left (\mathrm{e}^{K_h^j-K_h^{j-1}} -1\right)\\
	&\leq f_h^0+L(M)+\sum_{j=1}^i \alpha_h^j
	+\sum_{j=1}^i \left (f_h^0+L(M)+\sum_{k=1}^{j-1} \mathrm{e}^{-K_h^k}\alpha_h^k \right) \left (\mathrm{e}^{K_h^j}-\mathrm{e}^{K_h^{j-1}} \right)\\
	&=\sum_{j=1}^i \alpha_h^j+(f_h^0+L(M))\mathrm{e}^{K_h^i}+\sum_{j=1}^i \left (\mathrm{e}^{K_h^j}-\mathrm{e}^{K_h^{j-1}} \right) \sum_{k=1}^{j-1} \mathrm{e}^{-K_h^k}\alpha_h^k\\
	&\leq \sum_{j=1}^i \alpha_h^j+(f_h^0+L(M))\mathrm{e}^{K_h^i}+\sum_{j=1}^i \left (\mathrm{e}^{K_h^j}-\mathrm{e}^{K_h^{j-1}} \right) \sum_{k=1}^{i} \alpha_h^k\\
	&=\left (f_h^0+L(M)+\sum_{j=1}^i \alpha_h^i \right)\mathrm{e}^{K_h^i},
	\end{split}
	\end{equation*}
	which yields \eqref{eqn:aimp3}.
	
	Finally, we prove \eqref{eqn:aimp2bis}. Testing \eqref{eqn:aimp1} for $i-1$ if $i>1$ or \eqref{eqn:egamma-id-stability} if $i=1$ both with $\widehat{\boldsymbol{q}}_h=\boldsymbol{q}_h^i$, we have
	\begin{equation*}
	f_h^{i-1}\leq \alpha_h^{i-1}+\mathcal{F}_h(t_h^{i-1},\boldsymbol{q}_h^i)+d_h^i.
	\end{equation*}
	Here, in the second case, we set $\alpha_h^0\coloneqq 0$.  From this,
	employing the Fundamental Theorem of Calculus and \eqref{eqn:dr-gronwall3-aimp}, we compute
	\begin{equation}
	\label{eqn:aimpt2}
	\begin{split}
	f_h^{i-1}-f_h^i-d_h^i &\leq\alpha_h^{i-1}- \left(\mathcal{F}_h(t_h^i,\boldsymbol{q}_h^i)-\mathcal{F}_h(t_h^{i-1},\boldsymbol{q}_h^i) \right)\\
	&=\alpha_h^{i-1}-\int_{t_h^{i-1}}^{t_h^i} \partial_t \mathcal{F}_h(\tau,\boldsymbol{q}_h^i)\,\d\tau\\
	&\leq \alpha_h^{i-1}+(f_h^{i-1}+L(M)) \int_{t_h^{i-1}}^{t_h^i} \kappa_h(\tau) \mathrm{e}^{K_h(\tau)-K_h^{i-1}}\,\d\tau\\
	&\leq \alpha_h^{i-1}+(f_h^{i-1}+L(M)) \left(\mathrm{e}^{K_h^i-K_h^{i-1}}-1 \right).
	\end{split}
	\end{equation}
	Combining \eqref{eqn:aimpt1}--\eqref{eqn:aimpt2}, 
	we obtain \eqref{eqn:aimp2bis}.
\end{proof}

We now present the proof of our fourth main result.

\begin{proof}[Proof of Theorem \ref{thm:convergence-AIMP}]
Again, the proof follows the well established scheme in \cite{mielke.roubicek.stefanelli} and it is subdivided into six steps. 

\textbf{Step 1 (A priori estimates).} For every $h>0$, let $(\boldsymbol{q}_h^1,\dots,\boldsymbol{q}_h^{N_h})\in \mathcal{Q}_h^{N_h}$ be the solution to the AIMP determined by $\Pi_h=(t_h^0,\dots,t_h^{N_h})$ with tolerance $\alpha_h>0$ and initial datum $\boldsymbol{q}_h^0\in \mathcal{Q}_h$. 
We introduce the  piecewise-constant interpolant $\boldsymbol{q}_h\colon [0,T]\to \mathcal{Q}_h$ by setting
\begin{equation*}
    \boldsymbol{q}_h(t)\coloneqq 
    \begin{cases}
        \boldsymbol{q}_h^{i-1} & \text{if $t_h^{i-1}\leq t < t_h^i$ for some $i \in \{1,\dots,N_h\}$,}\\
        \boldsymbol{q}_h^{N_h} & \text{if $t=T$.}
    \end{cases}
\end{equation*}
Let $t \in [0,T]$. Given $h>0$,  let $i \in \{1,\dots,N_h\}$ be such that $t_h^{i-1}\leq t<t_h^i$. By definition, we have
\begin{equation}
    \label{eqn:aimp-pc-identities}
    \boldsymbol{q}_h(t)=\boldsymbol{q}_h^{i-1}, \qquad \mathrm{Var}_{\mathcal{D}_h}(\boldsymbol{q}_h;[0,t])=\sum_{j=1}^{i-1}\mathcal{D}_h(\boldsymbol{q}_h^{j-1},\boldsymbol{q}_h^j).
\end{equation}
Thus, using \eqref{eqn:aimp3} and \eqref{eqn:dr-gronwall2-aimp}, we infer
\begin{equation*}
\begin{split}
    \mathcal{F}_h(t,\boldsymbol{q}_h(t))+L(M)+\mathrm{Var}_{\mathcal{D}_h}(\boldsymbol{q}_h;[0,t]) &\leq (\mathcal{F}_h(t_h^{i-1},\boldsymbol{q}_h^{i-1})+L(M))\mathrm{e}^{K_h(t)-K_h(t_h^{i-1})}
    +\sum_{j=1}^{i-1}\mathcal{D}_h(\boldsymbol{q}_h^{j-1},\boldsymbol{q}_h^j)\\
    &\leq \left(\mathcal{F}_h(t_h^{i-1},\boldsymbol{q}_h^{i-1})+L(M)+ \sum_{j=1}^{i-1}\mathcal{D}_h(\boldsymbol{q}_h^{j-1},\boldsymbol{q}_h^j)\right)\mathrm{e}^{K_h(t)-K_h(t_h^{i-1})}\\
    &\leq (\mathcal{F}_h(0,\boldsymbol{q}_h^0)+L(M)+t_h^{i-1}\alpha_h) \left( \mathrm{e}^{K_h(t)}-1\right)\\
    &\leq (\mathcal{F}_h(0,\boldsymbol{q}_h^0)+L(M)+T\alpha_h)\left( \mathrm{e}^{K_h(T)}-1\right).
\end{split}
\end{equation*}
Since the sequences $(\mathcal{F}_h(0,\boldsymbol{q}_h^0))$, $(\alpha_h)$ and $(K_h(T))$ are all bounded, we obtain the estimate
\begin{equation}
\label{eqn:a-priori1}
\sup_{h>0} \sup_{t \in [0,T]} \mathcal{F}_h(t,\boldsymbol{q}_h(t))+\sup_{h>0}\mathrm{Var}_{\mathcal{D}_h}(\boldsymbol{q}_h;[0,T]) \leq C(M,T).
\end{equation}

For every $h>0$, define $f_h\colon [0,T]\to\R$ and $\boldsymbol{z}_h\colon [0,T]\to L^1(\Omega;\rt)$ by setting
$f_h(t)\coloneqq \mathcal{F}_h(t,\boldsymbol{q}_h(t))$ and $\boldsymbol{z}_h(t)\coloneqq \mathcal{Z}_h(\boldsymbol{q}_h(t))$. From \eqref{eqn:a-priori1}, we immediately get
\begin{equation}
\label{eqn:a-priori2}
    \sup_{h>0} \sup_{t \in [0,T]} f_h(t)+\sup_{h>0}\mathrm{Var}_{L^1(\Omega;\R^3)}(\boldsymbol{z}_h;[0,T])\leq C.
\end{equation}
We now establish a uniform bound for the total variation of $f_h$.  
For simplicity,  for every $h>0$ and $i \in \{1,\dots,N_h\}$, set 
\begin{equation*}
    f_h^i\coloneqq \mathcal{F}_h(t_h^i,\boldsymbol{q}_h^i), \quad d_h^i\coloneqq \mathcal{D}_h(\boldsymbol{q}_h^{i-1},\boldsymbol{q}_h^i), \quad K_h^i\coloneqq K_h(t_h^i), \quad  \alpha_h^i\coloneqq (t_h^i-t_h^{i-1})\alpha_h.
\end{equation*}
Also, for convenience of notation, we set $t_h^{-1}\coloneqq 0$.
First, denote by $[f_h]^i$ the jump of $f_h$ at time $t_h^i$.
Exploiting the continuity of $t\mapsto \mathcal{F}_h(t,\boldsymbol{q}_h^{i-1})$ and
employing  \eqref{eqn:aimp2bis} and \eqref{eqn:dr-gronwall2-aimp},  we compute
\begin{equation}
\label{eqn:energy-bv1}
\begin{split}
|[f_h]^i|&=\left |\lim_{s \to 0^+} \left \{f_h(t_h^i)-f_h(t_h^i-s) \right\} \right |
=\left |\lim_{s \to 0^+} \left \{ \mathcal{F}_h(t_h^i,\boldsymbol{q}_h^i)-\mathcal{F}_h(t_h^i-s,\boldsymbol{q}_h^{i-1})\right\} \right |\\
&=|f_h^i-f_h^{i-1}-(\mathcal{F}_h(t_h^i,\boldsymbol{q}_h^{i-1})-\mathcal{F}_h(t_h^{i-1},\boldsymbol{q}_h^{i-1}))|
\leq |f_h^i-f_h^{i-1}|+\int_{t_h^{i-1}}^{t_h^i} | \partial_t \mathcal{F}_h(\tau,\boldsymbol{q}_h^{i-1})|\,\d\tau\\
&\leq |f_h^i-f_h^{i-1}| + (f_h^{i-1}+L(M))\int_{t_h^{i-1}}^{t_h^i} \kappa_h(\tau) \mathrm{e}^{K_h(\tau)-K_h^{i-1}}\,\d\tau\\
&\leq |f_h^i-f_h^{i-1}|+(M+L(M))\left (\mathrm{e}^{K_h^i-K_h^{i-1}}-1 \right)\\
&\leq d_h^i+\alpha_h^{i-1}+\alpha_h^i + 2 (M+L(M))\left (\mathrm{e}^{K_h^i-K_h^{i-1}}-1 \right).
\end{split}
\end{equation}
Second, for $t\in (t_h^{i-1},t_h^i)$, we have $f_h(t)=\mathcal{F}_h(t,\boldsymbol{q}_h^{i-1})$ and, in turn, $\dot{f}_h(t)=\partial_t \mathcal{F}_h(t,\boldsymbol{q}_h^{i-1})$. Making once more use of \eqref{eqn:dr-gronwall2-aimp}, we compute
\begin{equation}
\label{eqn:energy-bv2}
\begin{split}
\int_{t_h^i}^{t_h^{i-1}} |\dot{f}_h(\tau)|\,\d\tau &=\int_{t_h^i}^{t_h^{i-1}} |\partial_t \mathcal{F}_h(\tau,\boldsymbol{q}_h^{i-1})|\,\d\tau\\
&\leq (f_h^{i-1}+L(M)) \int_{t_h^i}^{t_h^{i-1}} \kappa_h(\tau) \mathrm{e}^{K_h(\tau)-K_h^{i-1}}\,\d\tau\\
&\leq (M+L(M)) \left (\mathrm{e}^{K_h^i-K_h^{i-1}}-1 \right).
\end{split}
\end{equation}
Thus, combining \eqref{eqn:energy-bv1}--\eqref{eqn:energy-bv2}, we obtain
\begin{equation}
\label{eqn:bvf}
\begin{split}
\mathrm{Var}(f_h;[0,T])&=\sum_{i=1}^{N_h} \left \{|[f_h]^i|+\int_{t_h^i}^{t_h^{i-1}} |\dot{f}_h(\tau)|\,\d\tau \right \}\\
&\leq \sum_{i=1}^{N_h} \left \{d_h^i+\alpha_h^i+\alpha_h^{i-1}+3 (M+L(M)) \left (\mathrm{e}^{K_h^i-K_h^{i-1}}-1 \right) \right\}\\
&\leq \mathrm{Var}_{\mathcal{D}_h}(\boldsymbol{q}_h;[0,T])+2T\alpha_h+3(M+L(M))\left (\mathrm{e}^{K_h(T)}-1 \right),
\end{split}
\end{equation}
where in the last line we computed
\begin{equation*}
\begin{split}
\sum_{i=1}^{N_h}  \left (\mathrm{e}^{K_h^i-K_h^{i-1}}-1 \right)= \sum_{i=1}^{N_h} \mathrm{e}^{-K_h^{i-1}} \left(\mathrm{e}^{K_h^i}-\mathrm{e}^{K_h^{i-1}} \right)
\leq \sum_{i=1}^{N_h}  \left(\mathrm{e}^{K_h^i}-\mathrm{e}^{K_h^{i-1}} \right)
= \mathrm{e}^{K_h(T)}-1.
\end{split}
\end{equation*}
Thus, recalling the boundedness of $(\alpha_h)$ and $(K_h(T))$, where the latter follows from \eqref{eqn:Kh-conv}, the inequality \eqref{eqn:bvf} together with \eqref{eqn:a-priori1}  gives
\begin{equation}
\label{eqn:a-priori3}
\sup_{h>0}\mathrm{Var}(f_h;[0,T])\leq C(M,T).
\end{equation}

\textbf{Step 2 (Compactness).} For every $h>0$, define $\delta_h\colon [0,T]\to [0,+\infty)$ and the functions
\begin{equation*}
\begin{split}
&\boldsymbol{u}_h\colon [0,T]\to W^{1,2}(S;\R^2), \quad  v_h\colon [0,T]\to W^{1,2}(S), \quad  \boldsymbol{w}_h\colon [0,T]\to W^{1,2}(S;\rt),\\
&\hspace{20mm}\boldsymbol{\mu}_h\colon [0,T] \to L^2(\rt;\rt), \quad \boldsymbol{N}_h\colon [0,T]\to L^2(\rt;\rtt).
\end{split}	
\end{equation*}
as in in Step 2 of the proof of Theorem \ref{thm:evolutionary-gamma-convergence}. 
Exploiting the a priori estimates \eqref{eqn:a-priori1}, \eqref{eqn:a-priori2} and \eqref{eqn:a-priori3}, with the aid of the Helly compactness theorem \cite[Theorem 3.2]{mainik.mielke}, 
we establish the existence of functions  $f\in BV([0,T])$,  $\boldsymbol{z}\in BV([0,T];L^1(\Omega;\rt))$ and $\delta \colon [0,T]\to [0,+\infty)$ such that, up to subsequences, the convergences in \eqref{eqn:helly-f}--\eqref{eqn:helly-delta} hold true.
Moreover, employing the measurable selection theorem, we identify  a function $\boldsymbol{q}_0 \colon [0,T]\to {\mathcal{Q}}_0$ with $\boldsymbol{q}_0(t)=(\boldsymbol{u}(t),v(t),\boldsymbol{\zeta}(t))$ which is measurable and bounded. This function is characterized by the property that, for every $t\in [0,T]$, there exists a subsequence $(h_k)$, possibly depending on $t$, such that \eqref{eqn:egamma-cmpt-u}--\eqref{eqn:egamma-cmpt-v} hold true and also
\begin{align*}
	\boldsymbol{\mu}_{h_k}(t)&\text{$\wk \chi_\Omega\boldsymbol{\zeta}(t)$ in $L^2(\rt;\rt)$,}\\
	\boldsymbol{N}_h(t)&\text{$\wk\chi_\Omega(\nabla'\boldsymbol{\zeta}(t)\vert \boldsymbol{\chi}(t))$ in $L^2(\rt;\rtt)$,}
\end{align*}
for some $\boldsymbol{\chi} \colon [0,T]\to L^2(\rt;\rt)$ measurable and bounded. 
Then, for every $h>0$, we define $P_h\colon (0,T)\to \R$ as in Step 2 of the proof of Theorem \ref{thm:evolutionary-gamma-convergence} and, with the same arguments, we prove that \eqref{eqn:weak-conv-power} holds true for some $P\in L^1(0,T)$. Furthermore, defining $P_0\colon (0,T)\to \R$ and $\widehat{P}\colon (0,T)\to \R$ as in Step 2 of the proof of Theorem \ref{thm:evolutionary-gamma-convergence}, we show that $P\leq \widehat{P}$ and  $\widehat{P}= P_0$ almost everywhere in $(0,T)$.

\textbf{Step 3 (Reduced stability).} 
In view of \eqref{eqn:aimp1}, for every $h>0$, there holds
\begin{equation}
\label{eqn:aimp1-bis}
	\forall\,t_h\in\Pi_h,\:\forall\,\widehat{\boldsymbol{q}}_h\in \mathcal{Q}_h,\quad \mathcal{F}_h(t_h,\boldsymbol{q}_h(t_h))\leq |\Pi_h|\alpha_h
	+\mathcal{F}_h(t_h,\widehat{\boldsymbol{q}}_h)+\mathcal{D}_h(\boldsymbol{q}_h(t_h),\widehat{\boldsymbol{q}}_h).
\end{equation}
Define $\tau_h\colon [0,T]\to \Pi_h$ by setting $\tau_h(t)\coloneqq \{s \in \Pi_h:\:s\leq t\}$. Since $|\Pi_h|\to 0$, as $h\to 0^+$, we have $\tau_h(t)\to t$, as $h\to 0^+$.
Fix $t \in [0,T]$ and let $(h_k)$ be the subsequence such that \eqref{eqn:egamma-cmpt-u}--\eqref{eqn:egamma-cmpt-nu}  hold. By definition, we have $\boldsymbol{q}_{h_k}(t)=\boldsymbol{q}_{h_k}(\tau_{h_k}(t))$. Let $\widehat{\boldsymbol{q}}_0\in {\mathcal{Q}}_0$ and let $(\widehat{\boldsymbol{q}}_h)$ a corresponding recovery sequence  given by Proposition \ref{prop:recovery}. By \eqref{eqn:aimp1-bis}, there holds
\begin{equation}
\label{eqn:s1}
\mathcal{F}_{h_k}(\tau_{h_k}(t),\boldsymbol{q}_{h_k}(t))\leq |\Pi_{h_k}|\, \alpha_{h_k}+ \mathcal{F}_{h_k}(\tau_{h_k}(t),\widehat{\boldsymbol{q}}_{h_k})+\mathcal{D}_{h_k}(\boldsymbol{q}_{h_k}(t),\widehat{\boldsymbol{q}}_{h_k}).
\end{equation}
We focus on the left-hand side. First, arguing as in Step 3 of the proof of Theorem \ref{thm:evolutionary-gamma-convergence}, we obtain
\begin{equation}
    \label{eqn:red-stab1}
    \mathcal{F}_0(t,\boldsymbol{q}_0(t))\leq \liminf_{k\to \infty} \mathcal{F}_{h_k}(t,\boldsymbol{q}_{h_k}(t))=f(t).
\end{equation}
Second, applying the fundamental theorem of calculus and employing \eqref{eqn:dr-gronwall3-aimp} and \eqref{eqn:a-priori1}, we estimate
\begin{equation}
\label{eqn:red-stab2}
\begin{split}
|\mathcal{F}_h(t,\boldsymbol{q}_h(t))&-\mathcal{F}_h(\tau_h(t),\boldsymbol{q}_h(\tau_h(t)))|\leq \int_{\tau_h(t)}^t |\partial_t \mathcal{F}_h(\tau,\boldsymbol{q}_h(\tau_h(t)))|\,\d\tau\\
&\leq (\mathcal{F}_h(\tau_h(t),\boldsymbol{q}_h(\tau_h(t)))+L(M)) \int_{\tau_h(t)}^t \kappa_h(\tau) \mathrm{e}^{K_h(\tau)-K_h(\tau_h(t))}\,\d\tau\\
&=(M+L(M)) \left (\mathrm{e}^{K_h(t)-K_h(\tau_h(t))}-1 \right).
\end{split}
\end{equation}
Thanks to the equi-integrability of $(\kappa_h)$, which follows from \eqref{eqn:kh-conv}, the right-hand side of the previous equations goes to zero, as $k\to \infty$. Hence, combining \eqref{eqn:red-stab1}--\eqref{eqn:red-stab2}, we obtain
\begin{equation}
\label{eqn:red-stab3}
\mathcal{F}_0(t,\boldsymbol{q}_0(t))\leq \liminf_{k \to \infty} \mathcal{F}_{h_k}(\tau_{h_k}(t),\boldsymbol{q}_{h_k}(t))=f(t).
\end{equation}
Similarly for the right-hand side of \eqref{eqn:s1}, by arguing as in Step 3 of the proof of Theorem \ref{thm:evolutionary-gamma-convergence} and exploiting the continuity of the applied loads with respect to time, we compute
\begin{equation}
\label{eqn:red-stab4}
\mathcal{F}_0(t,\widehat{\boldsymbol{q}}_0)+\mathcal{D}_0(\boldsymbol{q}_0(t),\widehat{\boldsymbol{q}}_0)=\lim_{k \to \infty} \left \{ \mathcal{F}_{h_k}(\tau_{h_k}(t),\widehat{\boldsymbol{q}}_{h_k})+\mathcal{D}_{h_k}(\boldsymbol{q}_{h_k}(t),\widehat{\boldsymbol{q}}_{h_k}) \right \}.
\end{equation}
Thus, in view of \eqref{eqn:red-stab3}--\eqref{eqn:red-stab4}, taking the inferior limit, as $k\to \infty$, at both sides of \eqref{eqn:s1}, we obtain
\begin{equation*}
\begin{split}
\mathcal{F}_0(t,\boldsymbol{q}_0(t))&\leq \liminf_{k \to \infty} \mathcal{F}_{h_k}(\tau_{h_k}(t),\boldsymbol{q}_{h_k}(t))\\
&\leq \liminf_{k \to \infty} \left \{ |\Pi_{h_k}|\, \alpha_{h_k}+ \mathcal{F}_{h_k}(\tau_{h_k}(t),\widehat{\boldsymbol{q}}_{h_k})+\mathcal{D}_{h_k}(\boldsymbol{q}_{h_k}(t),\widehat{\boldsymbol{q}}_{h_k}) \right\}\\
&=\mathcal{F}_0(t,\widehat{\boldsymbol{q}}_0)+\mathcal{D}_0(\boldsymbol{q}_0(t),\widehat{\boldsymbol{q}}_0),
\end{split}
\end{equation*}
which proves \eqref{eqn:reduced-stability} for $t$ fixed.

\textbf{Step 4 (Upper energy-dissipation inequality).}
We prove the following:
\begin{equation}
\label{eqn:uee}
	\forall\,t \in [0,T], \quad \mathcal{F}_0(t,\boldsymbol{q}_0(t))+\mathrm{Var}_{\mathcal{D}_0}(\boldsymbol{q}_0;[0,t])\leq \mathcal{F}_0(0,\boldsymbol{q}_0^0)	+\int_0^t \partial_t \mathcal{F}_0(\tau,\boldsymbol{q}_0(\tau))\,\d\tau. 
\end{equation}
First, fix $t \in [0,T]$.  Given $h>0$, let $\tau_h(t)=t_h^i$ where $i\in \{1,\dots,N_h\}$ and recall \eqref{eqn:aimp-pc-identities}. Summing the inequality \eqref{eqn:aimp2}, with $j$ in place of $i$, for $j\in \{1, \dots,i-1\}$ we obtain:
\begin{equation}
\label{eqn:aimp2-new}
 \mathcal{F}_h(\tau_h(t),\boldsymbol{q}_h(\tau_h(t)))+\mathrm{Var}_{\mathcal{D}_h}(\boldsymbol{q}_h;[0,\tau_h(t)])\leq \tau_h(t) \alpha_h +\mathcal{F}_h(0,\boldsymbol{q}_h^0)
+\int_{0}^{\tau_h(t)} \partial_t \mathcal{F}_h(\tau,\boldsymbol{q}_h(\tau))\,\d \tau.
\end{equation}
By definition
\begin{equation*}
	\mathrm{Var}_{\mathcal{D}_h}(\boldsymbol{q}_h;[0,t])=\mathrm{Var}_{\mathcal{D}_h}(\boldsymbol{q}_h;[0,\tau_h(t)]).	
\end{equation*}
Thus, \eqref{eqn:red-stab2} and \eqref{eqn:aimp2-new} yield
\begin{equation}
\label{eqn:se}
\begin{split}
\mathcal{F}_h(t,\boldsymbol{q}_h(t))+\mathrm{Var}_{\mathcal{D}_h}(\boldsymbol{q}_h;[0,t])&\leq \mathcal{F}_h(\tau_h(t),\boldsymbol{q}_h(\tau_h(t)))+\mathrm{Var}_{\mathcal{D}_h}(\boldsymbol{q}_h;[0,\tau_h(t)])\\
&+(M+L(M)) \left (\mathrm{e}^{K_h(t)-K_h(\tau_h(t))}-1 \right)\\
&\leq \tau_h(t)\alpha_h+\mathcal{F}_h(0,\boldsymbol{q}_h^0)+\int_0^{\tau_h(t)}P_h(\tau)\,\d \tau\\
&+(M+L(M)) \left (\mathrm{e}^{K_h(t)-K_h(\tau_h(t))}-1 \right).
\end{split}
\end{equation}
Now, let $(h_k)$ be the subsequence such that \eqref{eqn:egamma-cmpt-u}--\eqref{eqn:egamma-cmpt-nu} hold. For the first term on the left-hand side,  we have \eqref{eqn:red-stab1}. For the second one, by lower semicontinuity, \eqref{eqn:helly-z} entails 
\begin{equation*}
\begin{split}
\mathrm{Var}_{\mathcal{D}_0}(\boldsymbol{q}_0;[0,t])&=\mathrm{Var}_{L^1(\Omega;\rt)}(\boldsymbol{z};[0,t])
\leq \liminf_{k \to \infty} \mathrm{Var}_{L^1(\Omega;\rt)}(\boldsymbol{z}_{h_k};[0,t])\\
&=\liminf_{k\to \infty} \mathrm{Var}_{\mathcal{D}_{h_k}}(\boldsymbol{q}_{h_k};[0,t])=\delta(t).
\end{split}
\end{equation*}
Thus, taking the inferior limit along the subsequence $(h_k)$ , as $k \to \infty$, at both sides of \eqref{eqn:se}, we obtain
\begin{equation*}
\begin{split}
\mathcal{F}_0(t,\boldsymbol{q}_0(t))+\mathrm{Var}_{\mathcal{D}_0}(\boldsymbol{q}_0;[0,t])&\leq f(t)+\delta(t)\\
&\leq \mathcal{F}_0(0,\boldsymbol{q}_0^0)+\int_0^t P(\tau)\,\d\tau\\
&\leq \mathcal{F}_0(0,\boldsymbol{q}_0^0)+\int_0^t P_0(\tau)\,\d\tau.
\end{split}
\end{equation*}
This proves \eqref{eqn:uee}.

\textbf{Step 5 (Lower energy-dissipation inequality).} As in Step 5 of the proof of Theorem \ref{thm:evolutionary-gamma-convergence}, the lower energy-dissipation inequality is deduced from the reduced stability \eqref{eqn:reduced-stability} by applying \cite[Proposition 2.1.2.3]{mielke.roubicek}.

\textbf{Step 6 (Improved estimates).} To check \eqref{eqn:egamma-energy}--\eqref{eqn:egamma-dissipation}, we identify the functions $f$ and $V$ in \eqref{eqn:helly-f}--\eqref{eqn:helly-z} by arguing as in Step 6 of the proof of Theorem \ref{thm:evolutionary-gamma-convergence}. Eventually, to show \eqref{eqn:egamma-power}, we check that $P=P_0$ almost everywhere in $(0,T)$ again by arguing as in Step 6 of the proof of Theorem \ref{thm:evolutionary-gamma-convergence}. 
\end{proof}

\subsection{Alternative dissipation}
\label{subsec:alternative-dissipation}
{
\MMM 
We conclude the section by mentioning an alternative notion of dissipation proposed in \cite{bresciani.davoli.kruzik}. In the same paper, it has been observed that the dissipation distance in \eqref{eqn:D_h} allows for the dissipation of energy by means of composition with rigid motions. Although this possibility is discouraged by energy minimization, this fact is certainly questionable from the modeling point of view. 

This observation motivates the introduction of an alternative dissipative variable.
Let $h>0$ and let $\boldsymbol{q}=(\boldsymbol{y},\boldsymbol{m})\in\mathcal{Q}_h$ be an admissible state. Similarly to \eqref{eqn:lagrangian-magnetization}, we define 
\begin{equation}
\label{eqn:Zh-tilde}
    \widetilde{\mathcal{Z}}_h(\boldsymbol{q})\coloneqq(\adj \nabla_h\boldsymbol{y})\boldsymbol{m}\circ \boldsymbol{y},
\end{equation}
where the adjugate matrix simply denotes the transpose of the cofactor matrix.
The quantity in \eqref{eqn:Zh-tilde} constitutes a flux-preserving pull-back of the magnetization $\boldsymbol{m}$ to the reference space and has the appreciable feature of being frame indifferent \cite{bresciani.davoli.kruzik}. 

In view of \eqref{eqn:saturation}, up to assuming $a\geq p/(p-2)$ in  \eqref{eqn:coercivity-Phi-det}, we have  $\widetilde{\mathcal{Z}}_h(\boldsymbol{q})\in L^1(\Omega;\rt)$ for every $\boldsymbol{q}\in \mathcal{Q}_h$ with $E_h(\boldsymbol{q})<+\infty$.  
Therefore, we can define the alternative dissipation distance $\widetilde{\mathcal{D}}_h\colon \mathcal{Q}_h \times \mathcal{Q}_h \to [0,+\infty)$ by setting
\begin{equation}
\label{eqn:Dh-tilde}
    \widetilde{\mathcal{D}}_h(\boldsymbol{q},\widehat{\boldsymbol{q}})\coloneqq \int_\Omega \left | \widetilde{\mathcal{Z}}_h(\boldsymbol{q})- \widetilde{\mathcal{Z}}_h(\widehat{\boldsymbol{q}})  \right |\,\d\boldsymbol{x}.
\end{equation}
Standing the more restrictive growth condition in  \eqref{eqn:coercivity-Phi-det}, it can be shown that the distance $\widetilde{\mathcal{D}}_h$ is lower semicontinuous with respect to the natural topology on $\mathcal{Q}_h$, see \cite{bresciani.quasistatic,bresciani.davoli.kruzik}. However, the existence of energetic solutions to the bulk model is out-of-reach within this framework. This can be achieved by means of a suitable regularization of the energy in the spirit of gradient polyconvexity \cite{benesova.kruzik.schloemerkemper}. We refer to \cite{bresciani.davoli.kruzik} for more details.

Nevertheless, one might aim to prove convergence results analogous to Theorem \ref{thm:evolutionary-gamma-convergence} and Theorem \ref{thm:convergence-AIMP} in this alternative setting. In this case, a substantial obstacle is the absence of a priori bounds on the dissipation distance in \eqref{eqn:Dh-tilde}. Such bounds are crucial in order to be able to establish the compactness of the sequence $(\boldsymbol{q}_h)$ in Theorem \ref{thm:evolutionary-gamma-convergence} and Theorem \ref{thm:convergence-AIMP}. 

A practicable way to overcome this issue is to enforce some uniform bound on the dissipative variable into the definition of the class of admissible states in a form of a locking constraint \cite{benesova.kruzik.schloemerkemper}. This is what is substantially done, in a more implicit way, in \cite{davoli} and \cite{mielke.stefanelli} for the problem of dimension reduction and linearization in finite plasticity, respectively.
Alternatively, one can replace the distance in \eqref{eqn:Dh-tilde} with the function  $\widehat{\mathcal{D}}_h\colon \mathcal{Q}_h \times \mathcal{Q}_h \to [0,+\infty)$ given by
\begin{equation*}
    \widehat{\mathcal{D}}_h(\boldsymbol{q},\widehat{\boldsymbol{q}})\coloneqq \int_\Omega \left |   \frac{\widetilde{\mathcal{Z}}_h(\boldsymbol{q})}{|\widetilde{\mathcal{Z}}_h(\boldsymbol{q})|}- \frac{\widetilde{\mathcal{Z}}_h(\widehat{\boldsymbol{q}})}{|\widetilde{\mathcal{Z}}_h(\widehat{\boldsymbol{q}})|}  \right|\,\d\boldsymbol{x}. 
\end{equation*}
This distance is not lower semicontinuous with respect to the natural topology on $\mathcal{Q}_h$. However, in both cases, convergence results analogous to  Theorem \ref{thm:evolutionary-gamma-convergence} and Theorem \ref{thm:convergence-AIMP} can be established by having $ \widehat{\mathcal{D}}_h$ in place of $\mathcal{D}_h$ without introducing any regularization.

}

\section*{Acknowledgements}
This work has been carried out while the first author was affiliated to TU Wien. 
Both authors are grateful to  Elisa Davoli for helpful discussions during the preparation of this paper. The first author thanks also Maria Giovanna Mora for useful comments about modeling  and for suggesting him the use of the interpolation inequality.  
This work has been supported by the Austrian Science Fund (FWF) and the GA\v{C}R through the  grant I4052-N32/19-29646L, and by the Federal Ministry of Education, Science and Research of Austria (BMBWF) through the OeAD-WTZ project CZ04/2019, the  M\v{S}MT-GA\v{CR} projects 8J19AT013 and 8J22AT017,  and the M\v{S}MT-WTZ project 8J22AT017.


\begin{thebibliography}{50}

\bibitem{acerbi.fusco} E. Acerbi, N. Fusco, \emph{An approximation lemma for $W^{1,p}$ functions}, in {\it Material Instability in Continuum Mechanics and Related Mathematical Problems}, ed. J. M. Ball, Oxford Sc. Pubbl., Oxford, 1988.

\bibitem{agostiniani.desimone}
V. Agostiniani, A. De Simone,
\emph{Rigorous derivation of active plate models for thin sheets of nematic elastomers},
\newblock{ Math. Mech. Solids} \textbf{25} (2017), 1804--1830.

\bibitem{aubin.frankowska}
J.-P. Aubin, H. Frankowska,
\newblock{\em Set-Valued Analysis}, Birk\"{a}user, Boston, 1990.

\bibitem{barchiesi.desimone}
M. Barchiesi, A. De Simone,
\emph{Frank energy for nematic elastometers: a nonlinear model,}
\newblock{ESAIM Control Optim. Calc. Var.} \textbf{21} (2015), 372--377.

\bibitem{barchiesi.henao.moracorral}
M. Barchiesi, D. Henao, C. Mora-Corral,
\emph{Local invertibility in Sobolev spaces with applications to nematic elastomers and magnetoelasticity,}
{Arch. Ration. Mech. Anal.} \textbf{224} (2017), 743--816.

\bibitem{benesova.kruzik.schloemerkemper}
B. Bene\v{s}ov\'{a}, M. Kru\v{z}\'{\i}k, A. Schl\"{o}merkemper,
\emph{A note on locking materials and gradient polyconvexity,}
\newblock{ Math. Mod. Meth. Appl. Sci.} \textbf{28} (2018), 2367--2401.

\bibitem{braides}
A. Braides,
\newblock{\em $\Gamma$-convergence for beginners},
\newblock{Oxford University Press, New York, 2002.}

\bibitem{bresciani}
M. Bresciani,
\emph{Linearized von K\'{a}rm\'{a}n theory for incompressible magnetoelastic plates,} {Math. Mod. Meth. Appl. Sci.} \textbf{31} (2021), 1987--2037.

\bibitem{bresciani.quasistatic}
M. Bresciani, \emph{Quasistatic evolution in magnetoelasticity under subcritical coercivity assumptions},  preprint arXiv:2203.08744.

\bibitem{bresciani.thesis}
M. Bresciani, Existence results and dimension reduction problems in large-strain magnetoelasticity, Ph.\,D. Thesis, University of Vienna, 2022.


\bibitem{bresciani.davoli.kruzik}
M. Bresciani, E. Davoli, M. Kru\v{z}\`{i}k,
\emph{Existence results in large-strain magnetoelasticity},  
 {Ann. Inst. H. Poincar\'{e} Anal. Non Lin\'{e}aire} (2022), online at \url{https://ems.press/journals/aihpc/articles/7168658}.

\bibitem{brown}
W.F. Brown, 
\newblock{\em Magnetoelastic Interactions},
\newblock{Springer, New York, 1966.}

\bibitem{carbou}
	G. Carbou,
	\emph{Thin layers in micromagnetism},
	\newblock{ Math. Models Methods Appl. Sci.} \textbf{9} (2001), 1529--1546.

\bibitem{ciarlet}
P. G. Ciarlet,
\newblock{ \em Mathematical Elasticity I. Three-Dimensional Elasticity,}
\newblock{North-Holland, Amsterdam, 1988.}

\bibitem{davoli}
E. Davoli, \emph{Quasistatic evolution models for thin plates arising as low energy $\Gamma$-limits of finite plasticity},{Math. Mod. Meth. Appl. Sci.} \textbf{24} (2014), 2085-–2153.


\bibitem{davoli.kruzik.piovano.stefanelli}
E. Davoli, M. Kru\v{z}\'{\i}k, P. Piovano, U. Stefanelli,
\emph{Magnetoelastic thin films at large strains},
\newblock{Cont. Mech. Thermodyn.} \textbf{33} (2021), 327--341.

\bibitem{schmidt1}
M. de Benito Delgado, B. Schmidt, \emph{A hierarchy of multilayered plate models},
{ESAIM Control Optim. Calc. Var.} \textbf{27} (2021), Article no. S16.

\bibitem{desimone.james}
A. DeSimone, R. D. James, \emph{A constrained theory of magnetoelasticity}, {J. Mech. Phys. Solids} \textbf{50} (2002), 283--320.

\bibitem{desimone.teresi}
A. De Simone, L. Teresi, \emph{Elastic energies for nematic elastomers}, {Eur. Phys. J. E} \textbf{29} (2009), 191--204.




\bibitem{dorfmann.ogden}
L. Dorfmann, R.W. Ogden,
\newblock{\em Nonlinear Theory of Electroelastic and Magnetoelastic
Interactions},
\newblock{Springer, New York, 2014.}



\bibitem{fonseca.gangbo}
I. Fonseca, W. Gangbo,
\newblock{\em Degree Theory in Analysis and Applications},
\newblock{Oxford University Press, New York, 1995.}

\bibitem{fonseca.leoni}
I. Fonseca, G. Leoni,
\newblock{ \em Modern Methods in the Calculus of Variations: $L^p$ spaces,}
\newblock{Springer, New York, 2007.}

\bibitem{francfort.mielke}
G. Francfort, A. Mielke,
\emph{Existence results for a class of rate-independent material models with nonconvex elastic energies},
\newblock{ J. Reine Angew. Math.} \textbf{595} (2004), 55--91.

\bibitem{friesecke.james.mueller1}
G. Friesecke, R. D. James, S. M\"uller,
\emph{A theorem on geometric rigidity and the derivation of nonlinear plate theory from three-dimensional elasticity,}
\newblock{ Comm. Pure Appl. Math.} \textbf{55} (2002), 1461--1506.

\bibitem{friesecke.james.mueller2}
G. Friesecke, R. D. James, S. M\"{u}ller,
\emph{A hierarchy of plate models derived from nonlinear elasticity by Gamma convergence,}
\newblock{Arch. Ration. Mech. Anal.} \textbf{180} (2006), 183--236.

\bibitem{gioia.james}
G. Gioia, R. D. James,
\emph{Micromagnetics of very thin films},
\newblock{ Proc. Roy. Soc. Lond. Sect. A} \textbf{453} (1997), 213--223.

\bibitem{grafakos}
L. Grafakos,
\newblock{\em Classical Fourier Analysis},
\newblock{Springer, New York, 2014.}

\bibitem{grandi.kruzik.mainini.stefanelli}
D. Grandi, M. Kru\v{z}\'{i}k, E. Mainini, U. Stefanelli,
\emph{A phase-field approach to Eulerian interfacial energies},
\newblock{Arch. Ration. Mech. Anal.} \textbf{234} (2019), 351--373. 

\bibitem{hajlasz}
P. Haj{\l}asz,
\newblock{{\em Sobolev Mappings between Manifolds and Metric Spaces}, in {\em Sobolev Spaces in Mathematics I. Sobolev Type Inequalities}}, ed. V. Maz'ya,
\newblock{Springer, New York, 2009.}

\bibitem{henao.stroffolini}
D. Henao, B. Stroffolini,
\emph{Orlicz-Sobolev nematic elastomers},
\newblock{Nonlinear Anal.} \textbf{194} (2020), 111513.

\bibitem{hubert.schaefer}
A. Hubert, R. Sch\"{a}fer, \emph{Magnetic domains. The analysis of magnetic microstructures}, Springer, New York, 1998.

\bibitem{james.kinderlehrer}
R. D. James, D. Kinderlehrer, \emph{Theory of magnetostriction with applications
to $\mathrm{Tb}_{x} \mathrm{Dy}_{1-x}\mathrm{Fe}_2$}, {Philos. Mag. B.} \textbf{68} (1993), 237–-274.

\bibitem{kruzik.melching.stefanelli}
M. Kru\v{z}\'{\i}k, D. Melching, U. Stefanelli,
\emph{Quasistatic evolution for dislocation-free plasticity},
\newblock{ESAIM Control Optim. Calc. Var.} \textbf{26} (2020), Article no. 123.

\bibitem{kruzik.prohl}
M. Kru\v{z}\'{\i}k, A. Prohl, A,
\emph{Recent developments in modeling, analysis and numerics of ferromagnetism,}
\newblock{SIAM Rev.} \textbf{48} (2006),  439-–483.


\bibitem{kruzik.stefanelli.zanini}
M. Kru\v{z}\'{\i}k, U.  Stefanelli, C. Zanini,
\emph{Quasistatic evolution of magnetoelastic plates via dimension reduction}, \newblock{Discrete
Contin. Dyn. Syst.} \textbf{35} (2015),  2615--2623.

\bibitem{kruzik.stefanelli.zeman}
M. Kru\v{z}{\'i}k, U. Stefanelli, J. Zeman,
\emph{Existence results for incompressible magnetoelasticity,}
\newblock{Discrete Contin. Dyn. Syst.} \textbf{35} (2015), 5999--6013.

\bibitem{lecumberry.mueller}
M. Lecumberry, S. M\"{u}ller,
\emph{Stability of slender bodies under Compression and validity of the von K\'{a}rm\'{a}n theory},
\newblock{Arch. Ration. Mech. Anal.} \textbf{193} (2009), 255--310.

\bibitem{lewicka1}
M. Lewicka, L. Mahadevan, M. Pakzad, \emph{Models for elastic shells with incompatible strains}, {Proc. Roy. Soc. A} \textbf{470} (2014), 20130604.

\bibitem{lewicka2}
M. Lewicka, L. Mahadevan, M. Pakzad, \emph{The F\"{o}ppl-von K\'{a}rm\'{a}n equations for plates with incompatible strains}, {Proc. Roy. Soc. A} \textbf{467} (2011), 402–426.

\bibitem{liakhova.luskin.zhang}
J. Liakhova, M. Luskin, T. Zhang, 
\emph{Computational modeling of ferromagnetic shape memory thin films,} \newblock{Ferroelectrics} \textbf{ 342} (2006), 7382.

\bibitem{luskin.zhang}
M. Luskin, T. Zhang,
\emph{Numerical analysis of a model for ferromagnetic shape memory thin films,}
\newblock{Comput. Methods Appl. Mech. Eng.} \textbf{ 196} (2007),  37–-40.

\bibitem{mainik.mielke}
A. Mainik, A. Mielke, \emph{Existence results for energetic models for rate-independent systems}, {Calc. Var.} \textbf{22} (2005), 73--99.

\bibitem{marcus.mizel}
M. Marcus, V. J. Mizel,
\emph{Transformations by functions in Sobolev spaces and lower semicontinuity for parametric variational problems},
\newblock{Bull. Am. Math. Soc.} \textbf{79} (1973), 790--795.

\bibitem{maugin}
G. A. Maugin, \emph{Continuum mechanics of electromagnetic solids}, North Holland, Amsterdam, 1988.

\bibitem{mielke.roubicek}
A. Mielke, T. Roub\'{i}\v{c}ek,
\newblock{ \em Rate-independent Systems. Theory and Application},
\newblock{Springer, New York, 2015.}

\bibitem{mielke.roubicek.stefanelli}
A. Mielke, T. Roub\'{i}\v{c}ek, U. Stefanelli,
\emph{$\Gamma$-limits and relaxations for rate-independent evolutionary problems},
\newblock{Calc. Var.} \textbf{31} (2008), 387--416.

\bibitem{mielke.stefanelli}
A. Mielke , U. Stefanelli,
\emph{Linearized plasticity is the evolutionary $\Gamma$-limit of finite-plasticity},
\newblock{J. Eur. Math. Soc.} \textbf{15} (2013), 923--948.

\bibitem{necas}
J. Ne\v{c}as,
{\em Direct Methods in the Theory of Elliptic Equations},
Springer-Verlag, Berlin Heidelberg, 2012.

\bibitem{rybka.luskin}
P. Rybka, M. Luskin, \emph{Existence of energy minimizers for magnetostrictive
materials}, {SIAM J. Math. Anal.} \textbf{36} (2005),  2004--2019.

\bibitem{schmidt2}
B. Schmidt, \emph{Plate theory for stressed heterogeneous multilayers of finite bending energy}, {J. Math. Pures Appl.} \textbf{88} (2007), 107--122.

\bibitem{silhavy}
M. Šilhavý, \emph{Equilibrium of phases with interfacial energy: a variational approach}, {J. Elast.} \textbf{105} (2011), 271--303.

\bibitem{stefanelli}
U. Stefanelli, \emph{Existence for dislocation-free finite plasticity}, {ESAIM Control Optim. Calc. Var} \textbf{25}(2019), Article no. 21.

\bibitem{wheeden.zygmund}
R. L. Wheeden, A. Zygmund,
\newblock{\em Measure and Integral. An Introduction to Real Analysis},
\newblock{Marcel Dekker, New York, 1977.}



































\end{thebibliography}
\end{document}